\documentclass[10pt,a4paper]{article}

\usepackage[utf8]{inputenc}
\usepackage[T1]{fontenc}
\usepackage[english]{babel}
\usepackage{lmodern}

\usepackage[
a4paper,
top=2.2cm, bottom=2.2cm,
left=2.2cm, right=2.2cm,
headheight=12pt
]{geometry}
\usepackage{amsmath,amssymb,amsthm,mathtools}
\usepackage{mathrsfs}   % \mathscr
\numberwithin{equation}{section}  % equations numbered as (section.n), e.g. (5.5)
\usepackage{tikz}
\usetikzlibrary{patterns, decorations.pathreplacing, arrows.meta}
\usepackage{microtype}
\usepackage[skip=1pt plus 1pt]{parskip}
\usepackage[compact]{titlesec}

\usepackage{enumitem}
\setlist{noitemsep, topsep=2pt}

\usepackage{booktabs}

\usepackage[dvipsnames]{xcolor}
\definecolor{RoyalRed}{RGB}{157, 16, 45}
\usepackage[
colorlinks = true,
linkcolor  = RoyalBlue,
citecolor  = RoyalRed,
urlcolor   = RoyalRed
]{hyperref}

\usepackage[capitalise, noabbrev]{cleveref}
\crefformat{equation}{(#2#1#3)}                    % \cref{eq:foo}      -> (5.4)
\crefrangeformat{equation}{(#3#1#4--#5#2#6)}       % \crefrange{}{} -> (5.4--5.7)

\newtheorem{theorem}{Theorem}[section]
\newtheorem{proposition}[theorem]{Proposition}
\newtheorem{lemma}[theorem]{Lemma}

\theoremstyle{definition}
\newtheorem{definition}[theorem]{Definition}
\newtheorem{assumption}{Assumption}[section]

\theoremstyle{remark}
\newtheorem{remark}[theorem]{Remark}

\newcommand{\R}{\mathbb{R}}

\newcommand{\omg}{\omega}               % in-plane cross-section
\newcommand{\Ohmh}{\Omega_{h}}         % thin domain
\newcommand{\Iint}{I}                  % thickness interval
\newcommand{\Yp}{Y^{\prime}}           % in-plane reference cell

\newcommand{\Ghpm}{\Gamma_{h}^{\pm}}
\newcommand{\Ghw}{\Gamma_{h}^{w}}
\newcommand{\Ghlat}{\Gamma_{h}^{\mathrm{lat}}}

\newcommand{\Szero}{\mathscr{S}_{0}}
\newcommand{\Tgrad}{\mathbb{T}}
\newcommand{\Hc}{\mathcal{H}}

\newcommand{\bA}{\mathbf{A}}
\DeclareMathOperator{\tr}{tr}           % trace
\newcommand{\abs}[1]{\left\lvert #1 \right\rvert}

\newcommand{\dd}{\,\mathrm{d}}         % upright differential

\newcommand{\eps}{\varepsilon}
\let\e\eps                              % \e is now an alias for \eps

\newcommand{\Fhe}{\mathbf{F}_{h,\eps}}
\newcommand{\mhe}{\mathbf{m}_{h,\eps}}

\newcommand{\Qc}{\mathcal{Q}}

\newcommand{\Rc}{\mathcal{R}}
\newcommand{\Mc}{\mathcal{M}}
\newcommand{\Te}{\mathcal{T}_{\varepsilon}}
\newcommand{\Peh}{\Pi_{h,\varepsilon}}

\newcommand{\wh}[1]{\widehat{#1}}
\newcommand{\wo}[1]{\overline{#1}}
\newcommand{\pe}{\mathrm{per}}          % periodicity subscript

\newcommand{\p}{\partial}              % partial derivative
\newcommand{\X}{\times}                % Cartesian product / cross
\def\O{{\Omega}}
\def\o{{\omega}}
\def\Ac{{\mathcal{A}}}
\def\Cc{{\mathcal{C}}}
\def\GF{{\bf F}}
\def\Gc{{\mathcal{G}}}
\def\Gm{{\bf m}}
\def\Ebulk{{E_{\mathrm{bulk}}}}
\def\surf{\mathrm{surf}}
\def\hQ{\wh Q}
\def\hR{\wh R}
\def\el{\mathrm{el}}

\def\S{\mathbb{S}}

\def\bulk{\mathrm{bulk}}
\def\per{\mathrm{per}}

\usepackage{color}

\def\A{\mathbb{A}}
\def\C{\mathbb{C}}
\def\D{\mathbb{D}}
\def\bn{{\mathbf{n}}}
\def\B{{\mathbb{B}}}

\begin{document}
	%=====================================================================
	
%	\title{\bfseries Homogenisation of Heterogeneous Nematic Thin Films: a Landau--de Gennes Analysis}
	
		\title{\bfseries Multiscale Analysis of a Landau--de Gennes Model for Nematic thin-film composites}
		
	\author{
		Amartya Chakrabortty\thanks{
			Department of Mathematics, Indian Institute of Technology Kharagpur,  721302, Kharagpur
			India.\\
			Email: \href{amartyat26@kgpian.iitkgp.ac.in}{amartyat26@kgpian.iitkgp.ac.in}\\
			This work was largely completed when the author was a PhD student and research assistant in the department of Processes and Material Simulation, Fraunhofer ITWM, Kaiserslautern, Germany.
		}
%		\and
%		Georges Griso\thanks{
%			Laboratoire Jacques-Louis Lions (LJLL), Sorbonne Université, Paris, France.
%			Email: \href{griso@ljll.math.upmc.fr}{griso@ljll.math.upmc.fr}
%		}
	}
	
	\date{\today}
	\maketitle
	\thispagestyle{empty}
	
\begin{abstract}
	We study the simultaneous limits of homogenisation $(\varepsilon\to0)$ and
	dimension reduction $(h\to0)$ for thin heterogeneous nematic liquid crystal
	films in the Landau--de Gennes $Q$-tensor framework. The elastic energy density
	is governed by a tensor $\bA(x'/\varepsilon,x_3/h)$ that is periodic in the
	in-plane fast variable and measurable in the normalised thickness variable.
	Surface anchoring on the top and bottom faces is modelled by a weak anchoring
	energy of strength $h^\gamma$, $\gamma\geq0$, in a general set-valued framework
	covering the principal classical anchoring geometries.
	
	The simultaneous limit reveals two principal features. First, the anchoring
	scaling yields a hierarchy of effective behaviours depending on $\gamma$: a hard
	constraint $\mathcal{Q}\in H^1(\omega;\mathscr{A})$ for $0\leq\gamma<1$, a finite
	surface density contribution for $\gamma=1$, and vanishing anchoring for
	$\gamma>1$. For $0\leq \gamma<1$ with uniaxial anchoring, the homogenised energy reduces on the constrained class to anisotropic Oseen--Frank and Ericksen energies.Second, the effective elastic response depends on the scale ratio
	$\rho=\lim h/\varepsilon\in[0,\infty]$: each regime has
	a distinct corrector structure and yields a two-dimensional LdG energy with
	homogenised tensor $\bA^{\mathrm{hom}}_\rho$, characterised by a
	regime-dependent cell problem.
	
	To the best of our knowledge, this is the first rigorous treatment of the simultaneous homogenisation and dimension reduction limit for the Landau--de Gennes energy, and the first in which the anchoring strength enters as a scaling parameter, producing a hierarchy of qualitatively distinct effective models.
\end{abstract}


{\small	\tableofcontents}

\section{Introduction}
\label{sec:intro}

Nematic liquid crystals are partially ordered materials intermediate between
isotropic liquids and crystalline solids. Their rod-like molecules retain the
positional disorder of a liquid but develop a preferred orientational direction,
called the director. This orientational order underlies characteristic optical
and rheological effects, such as birefringence, flow alignment, and sensitivity
to electric and magnetic fields, with applications in displays, optical devices,
and soft actuators \cite{dGP93}.

Several continuum theories describe nematic order at different levels of
resolution \cite{dGP93,Virga94}. The Oseen--Frank theory \cite{Virga94}
represents the state by a unit director field $\mathbf{n}\in\S^2$ and gives a
geometrically transparent elastic theory away from defects, but its unit-vector
constraint prevents a regular description of defect cores and excludes biaxial
configurations. The Landau--de Gennes (LdG) theory \cite{dGP93,MZ10,BM10,MN14}
resolves this by using a symmetric traceless matrix field $Q\in\Szero$ as the
order parameter, whose eigenvalues and eigenvectors encode the degree and axes of
orientational order. This framework naturally accommodates biaxiality, defect
cores, and configurations not representable by a continuous director field
\cite{BZ11}. The corresponding free energy couples an elastic term, quadratic in
$\nabla Q$, with a non-convex bulk potential $f_b(Q)$ governing the
isotropic-to-nematic transition. Mathematical foundations of the LdG theory,
including existence and regularity of minimisers, the Oseen--Frank limit, and
refined approximation results, are developed in \cite{MZ10,BM10,NZ13}. The intermediate Ericksen theory \cite{Ericksen91,Lin91} retains the uniaxial
structure but allows the scalar order parameter to vary.

We study simultaneous homogenisation and dimension reduction for a thin
heterogeneous nematic film governed by the LdG energy. The film occupies $\Omega_h=\omega\times\left(-\frac h2,\frac h2\right)$,
where $\omega\subset\mathbb R^2$ is a bounded Lipschitz domain. Its thickness is
of order $h$, while the in-plane heterogeneity has period $\varepsilon$. We
consider the simultaneous limit
\begin{equation}\label{eq:rho}
(h,\varepsilon)\to(0,0),
\qquad
\frac h\varepsilon\to\rho\in[0,\infty].
\end{equation}
The limit is a family of two-dimensional LdG energies whose homogenised elastic
tensor depends on $\rho$. In the anchoring-dominated regime, suitable choices of
the effective constraint set yield homogenised anisotropic Oseen--Frank and
Ericksen energies.

The dependence on $\rho$ reflects the interaction between the thickness and
microstructure scales. If $\varepsilon\ll h$, in-plane homogenisation precedes
transverse relaxation; if $h$ and $\varepsilon$ are comparable, the two
relaxations remain coupled; and if $h\ll\varepsilon$, dimension reduction
precedes the in-plane homogenisation. This three-regime structure is familiar
from the simultaneous homogenisation and dimension reduction of elastic plates
\cite{CDG,Neu10,NV13}. Here it is developed for heterogeneous LdG energies with
anisotropic elastic coefficients and surface anchoring.

%The elastic energy density is governed by the heterogeneous tensor
%$\bA(x'/\varepsilon,x_3/h)$, $Y'$-periodic in the in-plane fast variable
%and measurable in the normalised thickness variable
%$x_3/h\in I:=(-\frac{1}{2},\frac{1}{2})$; no periodicity in the thickness
%direction is imposed. No continuity or periodicity in $x_3/h$ is
%required. 
%The bulk potential is the standard quartic LdG polynomial
%\[
%f_b(Q) := a\,\mathrm{tr}(Q^2) - b\,\mathrm{tr}(Q^3)
%+ c\bigl(\mathrm{tr}(Q^2)\bigr)^2,
%\quad a,b\in\mathbb{R},\; c>0,
%\]
%which is the local thermotropic part of the LdG energy \cite{dGP93,MN14,Gar18}.
%Surface anchoring on the top and bottom faces is formulated with strength
%$h^\gamma W(x'/\varepsilon, x_3/h, \nu)$ and exponent $\gamma\geq0$. Rather
%than fixing a single anchoring geometry, we work with a set-valued map
%$\mathcal{A}(y,\nu)$ that is $Y'$-periodic and accommodates heterogeneous
%preferred states on each face independently, with possible top--bottom asymmetry.
%The four classical anchoring models in the liquid crystal literature
%(Rapini--Papoular \cite{RP69}, degenerate planar, Fournier--Galatola \cite{FournierGalatola2005}, and planar
%anchoring with free in-plane director) all fit within this framework, as special
%cases verified in Section~\ref{sec:energy}.

The elastic density is determined by a uniformly coercive tensor $\bA\left(\frac{x'}{\varepsilon},\frac{x_3}{h}\right)$,
which is $Y'$-periodic in the in-plane variable and measurable in the normalised
thickness variable $x_3/h\in I:=(-\frac12,\frac12)$. No periodicity in the
thickness direction is imposed. The bulk potential is the quartic LdG
polynomial
\[
f_b(Q)
:=
a\,\operatorname{tr}(Q^2)
-b\,\operatorname{tr}(Q^3)
+c\bigl(\operatorname{tr}(Q^2)\bigr)^2,
\qquad a,b\in\mathbb R,\quad c>0.
\]
Surface anchoring on the top and bottom faces has strength
\[
h^\gamma
W\left(\frac{x'}{\varepsilon},\frac{x_3}{h},\nu\right),
\qquad \gamma\geq0.
\]
The preferred states are described by a set-valued map
$\mathcal A(y,\nu)\subset\Szero$. This formulation allows heterogeneous and
possibly different anchoring conditions on the two faces. Representative
examples include prescribed $Q$-tensor anchoring, prescribed normal action, and
fixed-order degenerate planar anchoring; see
Proposition~\ref{prop:compatible-examples}.

The relative rate between dimension reduction and homogenisation is encoded by
the parameter $\rho$.
%\begin{equation}
%	\label{eq:rho}
%	\rho := \lim_{(\varepsilon,h)\to(0,0)} \frac{h}{\varepsilon} \in [0,+\infty].
%\end{equation}
Our analysis identifies three distinct asymptotic regimes:
\begin{description}[leftmargin=2.5em,font=\normalfont]
	\item[(R1)\;$\rho=+\infty$:]
	The microstructure is much finer than the thickness, $\varepsilon\ll h$.
	In-plane homogenisation occurs at each fixed thickness level, followed by
	transverse relaxation.
	
	\item[(R2)\;$\rho\in(0,\infty)$:]
	The two scales are comparable. The in-plane and transverse microscopic
	relaxations are coupled, and the cell problem depends explicitly on $\rho$.
	
	\item[(R3)\;$\rho=0$:]
	The thickness is much finer than the microstructure, $h\ll\varepsilon$.
	The thickness reduction dominates, and the in-plane corrector is independent
	of the thickness variable.
\end{description}
Each regime leads to a different corrector space and cell problem. In the
central regime $\rho\in(0,\infty)$, the in-plane and transverse correctors enter
through the coupled field
\[
\widehat Q_\rho=\widehat Q+\rho\widehat R.
\]
The main result is the sequential $\Gamma$-convergence of the rescaled
three-dimensional energies to
\[
\GF^{\mathrm{hom},\gamma}_\rho(\Qc)
=
\frac12\int_\omega
\bA^{\mathrm{hom}}_\rho\nabla'\Qc:\nabla'\Qc\,dx'
+\int_\omega f_b(\Qc)\,dx'
+E^\gamma_{\mathrm{surf}}(\Qc).
\]
The homogenised tensor $\bA^{\mathrm{hom}}_\rho$ is characterised by a
regime-dependent variational cell problem. The convergence is established with
respect to the unfolded convergence of Definition~\ref{def:hom-conv} and is
accompanied by compactness of energy-bounded sequences and convergence of
minimum values and minimisers.

The anchoring contribution depends on $\gamma$. If $0\leq\gamma<1$, bounded
energy forces $\Qc\in H^1(\omega;\mathscr A)$,
where $\mathscr A$ is the effective intersection of the preferred states on the
two faces. If $\gamma=1$, the anchoring energy survives as the finite
contribution
\[
E^\gamma_{\mathrm{surf}}(\Qc) :=
\begin{cases}
	\displaystyle\int_\omega \bar{f}_s(\Qc(x'))\,dx', & \gamma=1,\\[4pt]
	0, & \gamma\in[0,\infty)\setminus\{1\}.
\end{cases}
\]
If $\gamma>1$, the anchoring energy vanishes in the limit. These regimes are
summarised in Table~\ref{tab:anchoring}.

\begin{table}[h]
	\centering
	\renewcommand{\arraystretch}{1.2}
	\begin{tabular}{cccc}
		\hline
		Range of $\gamma$ & Normalisation & Anchoring regime & Effective anchoring \\
		\hline
		$0\leq\gamma<1$ & $h^{-1}$ & Supercritical & Strong (hard constraint) \\
		$\gamma=1$       & $h^{-1}$ & Critical      & Weak (finite surface density) \\
		$\gamma>1$       & $h^{-1}$ & Subcritical   & Vanishing \\
		\hline
	\end{tabular}
	\caption{The three anchoring regimes determined by the exponent $\gamma\geq0$.
		In all cases the natural normalisation of the energy is $h^{-1}$. The effective
		anchoring set $\mathscr{A}\subset\Szero$ appearing in the supercritical regime
		is defined in Assumption~\ref{asm:compatible}.}
	\label{tab:anchoring}
\end{table}
%The anchoring hierarchy recorded in Table~\ref{tab:anchoring} is a theorem: from
%a single microscopic anchoring model with strength
%$h^\gamma W(x'/\varepsilon,x_3/h,\nu)$, the simultaneous limit produces three
%qualitatively distinct macroscopic behaviours, and the critical exponent
%$\gamma=1$ is identified variationally.
In the constrained regime $0\leq\gamma<1$, particular choices of
$\mathscr A$ recover lower-dimensional director theories. For
\[
\mathscr A
=
\left\{
s_0\left(P-\frac13I\right):P\in\mathbb P
\right\},
\qquad s_0\neq0,
\]
the limiting energy is equivalent to a homogenised anisotropic Oseen--Frank
energy for the unoriented director
$[\mathbf n]\in H^1(\omega;\mathbb S^2/\{\pm1\})$; see
Theorem~\ref{thm:OF}. Allowing the scalar order parameter to vary in
$[s_0,s_1]\subset(0,1)$ yields the corresponding homogenised Ericksen energy;
see Theorem~\ref{thm:Er}.

The relevant literature concerns dimension reduction for nematic films,
homogenisation of nematic composites, and simultaneous homogenisation and
dimension reduction for thin structures.

Dimension reduction for nematic thin films has been studied in several related
settings. In the director framework, dimension reduction was considered in
\cite{Canevari2023DimensionalRA}; surface free energies for nematic shells
are treated in \cite{NV12}, and a rigorous variational analysis of nematic shell
models is developed in \cite{SSV16}. In the $Q$-tensor framework, Bauman, Park,
and Phillips \cite{BPP12} analyse a two-dimensional LdG model on planar domains
and study minimisers with disclination lines. 
Golovaty, Montero, and Sternberg \cite{GMS15} establish a
$\Gamma$-convergence result for three-dimensional LdG thin films with spatially
homogeneous elastic coefficients and general weak anchoring conditions on the
top and bottom faces.
Further singular regimes for LdG thin
films, involving the nematic correlation length or textures with tactoids and
disclinations, are studied in \cite{Nov18,Gol+20}.

Homogenisation of nematic composites has been studied in both director and
$Q$-tensor frameworks. In the director setting we refer to \cite{Ber+05,Cal+14};
in particular, Calderer, DeSimone, Golovaty, and Panchenko \cite{Cal+14} derive
an effective model for nematic composites with ferromagnetic inclusions at fixed
positive thickness. In the LdG framework, Canevari and Zarnescu \cite{CZ20}
establish convergence of local minimisers for a nematic host doped with colloidal
inclusions in a dilute regime with Rapini--Papoular anchoring. 
Homogenisation of effective surface energies for rapidly oscillating boundary
conditions is studied by Ceuc\u{a}, Taylor, and Zarnescu \cite{CTZ23} at
fixed positive thickness; this is complementary to the averaged anchoring
contribution obtained here.

The simultaneous limit $(h,\varepsilon)\to(0,0)$ has a substantial history in
elasticity. Rescaling--unfolding methods for thin heterogeneous structures were
developed by Blanchard, Gaudiello, and Griso \cite{BGG07a,BGG07b}, and a
three-regime analysis for linearised elastic plates is presented in
\cite[Chapter~11]{CDG}. The geometrically nonlinear membrane and von
K\'{a}rm\'{a}n settings were treated by Babadjian and Baia \cite{babadjian20063d}, Braides--Fonseca--Francfort \cite{braides20003d},
and Neukamm and Vel\v{c}i\'{c} \cite{NV13}, respectively. The
rescaling--unfolding operator goes back to the Cioranescu--Damlamian--Griso
framework \cite{CDG02,CDG08,CDG}; the thickness decomposition used below is
inspired by Griso's decomposition of plate displacements \cite{GrisoDecomp,CDG}, deformations \cite{blanchard2010decomposition}, and general thin film \cite{GrisoError}. These decompositions have been applied
extensively to homogenised reduced-order models for elastic structures in
\cite{griso2020homogenization,griso2021asymptotic,griso2020asymptotic,falconi2025asymptotic}.
Related simultaneous homogenisation and dimension-reduction methods were
developed in previous joint work of the author
\cite{CGO_soft,CGOC_rigid,CGO_mosaic} for thin composite plates with soft or
disconnected rigid inclusions.

%The following points distinguish the LdG setting: the unknown is already a
%matrix-valued $Q$-tensor field, so $\nabla Q$ takes values in the third-order
%tensor space $\Szero\otimes\mathbb{R}^3$, increasing the dimension of the corrector
%problem and complicating the identification of admissible corrector spaces in
%each regime. Beyond this, 
%%sustaining the measurability-only assumption on $\bA$
%%in the thickness variable throughout a nonlinear LdG energy requires arguments
%%without counterpart in the elastic plate setting. 
%the nonconvex quartic bulk potential
%$f_b(Q)$ interacts with the two-scale structure, requiring strong $L^p$
%convergence for $p<4$ to be established via the thickness decomposition rather
%than equi-integrability. The set-valued anchoring framework with exponent
%$\gamma$ produces the three-regime hierarchy from a single
%microscopic model. Two further contributions complete the analysis: continuity of
%$\rho\mapsto\bA^{\mathrm{hom}}_\rho$ on $[0,\infty]$
%(Theorem~\ref{thm:hom-tensor-continuity}), showing that the three regimes fit
%together into a coherent one-parameter family; and a density result for smooth
%constrained maps in $H^1(\omega;\mathscr{A})$ (Lemma~\ref{lem:density-Astar}),
%established using the Schoen--Uhlenbeck approximation theorem
%\cite{schoen1983boundary} at the critical exponent $p=n=2$ and verified for all
%four principal anchoring models.

The analysis combines several features specific to the present problem. The
elastic tensor is fully anisotropic and is assumed only measurable in the
thickness variable. The thickness decomposition gives strong $L^2$ convergence
of the macroscopic component; interpolation with the uniform $L^4$ bound then
yields strong $L^p$ convergence for $2\leq p<4$, which is used to pass to the
lower-order terms in the non-convex bulk potential. The surface term requires a
separate treatment in the three ranges of $\gamma$, including the identification
of the effective constraint set when $0\leq\gamma<1$.

We also prove that
\[
\rho\longmapsto\bA^{\mathrm{hom}}_\rho
\]
is continuous on the compactified interval $[0,\infty]$; see
Theorem~\ref{thm:hom-tensor-continuity}. Thus the two scale-separated regimes
are recovered as the endpoint limits of the coupled regime. Finally,
Lemma~\ref{lem:density-Astar} establishes the density of smooth constrained maps
in $H^1(\omega;\mathscr A)$ for the anchoring sets considered in
Proposition~\ref{prop:compatible-examples}.

\noindent\textbf{Organisation of the paper.}
Section~\ref{sec:pro} formulates the model and states the main results.
Section~\ref{sec:pre} introduces the rescaling--unfolding operators and the
thickness decomposition. Section~\ref{sec:asymptotic} proves existence, uniform
bounds, compactness, the liminf inequality, and the recovery result.
Section~\ref{sec:proofs} identifies the homogenised cell problems, proves
continuity with respect to $\rho$, and derives the Oseen--Frank and Ericksen
representations. The auxiliary density and coercivity results are collected in
Appendix~\ref{app:den}.

\subsection*{Notation}\label{sec:N}
	The reference cell is $Y := Y' \times I$ where $Y' := (0,1)^2$ 
	and $I := (-\tfrac{1}{2},\tfrac{1}{2})$.
	Periodicity is always understood with respect to the in-plane 
	variable $y' \in Y'$ only; no periodicity in $y_3 \in I$ is 
	imposed or required. Set
	\begin{equation}
		\label{eq:S0}
		\Szero \;:=\; \bigl\{Q\in\R^{3\times 3}
		: Q = Q^{\top},\; \tr Q = 0\bigr\},
	\end{equation}
	We define the periodic Sobolev spaces
	\begin{align*}
		H^1_{\mathrm{per}}(Y';\Szero)
		&:= \bigl\{ f_{|Y'}\,:\,f \in H^1_{\mathrm{loc}}(\mathbb{R}^2;\Szero),\quad f \text{ is } Y'\text{-periodic} \bigr\},
		\\
		H^1_{\mathrm{per}}(Y;\Szero)
		&:= \bigl\{ f \in L^2(I;\,H^1_{\mathrm{per}}(Y';\Szero))
		: \partial_{y_3} f \in L^2(Y;\Szero) \bigr\},
	\end{align*}
	and their zero-mean subspaces
	\begin{align*}
		H^1_{\mathrm{per},0}(Y';\Szero)
		&:= \Bigl\{ f \in H^1_{\mathrm{per}}(Y';\Szero)
		: \int_{Y'} f(y')\,dy' = 0 \Bigr\},\\
		H^1_{\mathrm{per},0}(Y;\Szero)
		&:= \Bigl\{ f \in H^1_{\mathrm{per}}(Y;\Szero)
		: \int_{Y} f(y)\,dy = 0 \Bigr\}, \quad 		H^1_\sharp(I;\,\Szero) 
		:= \Bigl\{f \in H^1(I;\,\Szero) 
		: \textstyle\int_I f\,dy_3 = 0\Bigr\},\\
		H^1_{\mathrm{per},\sharp}(Y;\Szero)
		&:= \Bigl\{ f \in H^1_{\mathrm{per}}(Y;\Szero)
		: \int_{I} f(\cdot,y_3)\,dy_3 = 0
		\text{ a.e.\ in } Y' \Bigr\}.
	\end{align*}
	All four are closed subspaces of their parent Hilbert spaces, 
	hence Hilbert spaces in their own right.

	%=====================================================================
	\section{Problem description and main results}
	\label{sec:pro}
	
	\subsection{Geometry of the nematic thin film}
	\label{sec:geometry}
	
	Let $\omg\subset\R^{2}$ be a {bounded Lipschitz domain} (the
	mid-surface) and let $\Iint := (-\frac{1}{2},\frac{1}{2})$.
	For each $h > 0$ the {thin plate} is
	\begin{equation}
		\label{eq:Ohmh}
		\Ohmh \;:=\; \omg\times h\Iint
		\;=\; \bigl\{x=(x^{\prime},x_3)\in\R^{3}
		: x^{\prime}\in\omg,\; \abs{x_3} < h/2\bigr\}.
	\end{equation}
	We decompose the boundary $\partial\Ohmh$ into three disjoint parts:
	\begin{equation}
		\label{eq:boundary}
		\partial\Ohmh
		\;=\; \Ghpm \;\cup\; \Ghlat,
		\qquad
		\Ghpm := \omg\times\bigl\{\pm h/2\bigr\},
		\qquad
		\Ghlat := \partial\omg\times\bigl(-h/2,\,h/2\bigr).
	\end{equation}
	The top and bottom faces form the {weak-anchoring region}
	\begin{equation}
		\label{eq:Ghw}
		\Ghw \;:=\; \Gamma_{h}^{+}\cup\Gamma_{h}^{-},
		\qquad \mathcal{H}^{2}(\Ghw) = 2\abs{\omg},
	\end{equation}
	whose area is independent of $h$.  The outward unit normal on
	$\Gamma_{h}^{\pm}$ is $\nu = \pm e_3$.  No boundary condition is
	imposed on the lateral face $\Ghlat$.
	The order-parameter space is $\Szero$,
%	\begin{equation}
%		\label{eq:S0}
%		\Szero \;:=\; \bigl\{Q\in\R^{3\times 3}
%		: Q = Q^{\top},\; \tr Q = 0\bigr\},
%	\end{equation}
	which is a five-dimensional real linear subspace of $\R^{3\times 3}$ equipped
	with the Frobenius inner product $Q:P := \tr(QP)$ and the induced
	norm $\abs{Q} := (\tr Q^{2})^{1/2}$.  The {space of admissible
		gradients} is
	\begin{equation}
		\label{eq:T}
		\Tgrad \;:=\; \Szero\otimes\R^{3}
		\;=\; \bigl\{P=(P_{ij\alpha}) :
		P_{ij\alpha} = P_{ji\alpha},\;
		{\textstyle\sum_{i}}P_{ii\alpha}=0\;\forall\alpha\bigr\}.
	\end{equation}
	For $Q\in H^{1}(\Ohmh;\Szero)$, one has $\nabla Q(x)\in\Tgrad$ for
	a.e.\ $x\in\Ohmh$.  The natural admissible class is
	$H^{1}(\Ohmh;\Szero)$.
	
	\begin{figure}[ht]
		\centering
		\begin{tikzpicture}[
			>=Stealth,
			font=\small,
			top face/.style  ={fill=blue!15, draw=blue!60!black, line width=0.9pt},
			bot face/.style  ={fill=blue!7,  draw=blue!45!black, line width=0.7pt},
			lat face/.style  ={fill=gray!18, pattern=north east lines,
				pattern color=gray!60,
				draw=gray!55, line width=0.6pt},
			midsurface/.style={draw=green!55!black, dashed, line width=0.85pt},
			axis line/.style ={->, gray!65!black, line width=0.7pt},
			dim line/.style  ={<->, gray!55!black, line width=0.55pt},
			leader/.style    ={dashed, gray!50, line width=0.45pt},
			int cell/.style  ={fill=teal!28,   draw=gray!45, line width=0.25pt},
			bnd cell/.style  ={fill=orange!22, draw=gray!45, line width=0.25pt},
			hi cell/.style   ={fill=blue!10,   draw=blue!65!black, line width=1.3pt},
			]
			
			% ==============================================================
			%  LEFT PANEL — oblique projection of thin plate Omega_h
			% ==============================================================
			
			\begin{scope}[xshift=-3.6cm, yshift=0.2cm]
				
				% ---- 1. Lateral face (visible right strip, drawn FIRST) ----
				\fill[lat face]
				(1.56, 1.16)
				.. controls (1.76, 0.56) and (1.46,-0.44) ..
				(0.66,-0.99)
				-- (1.04,-1.21)
				.. controls (1.84,-0.66) and (2.14, 0.34) ..
				(1.94, 0.94)
				-- cycle;
				
				\fill[lat face]
				(0.66,-0.99)
				.. controls (-0.14,-1.54) and (-1.19,-1.34) ..
				(-1.94,-0.59)
				-- (-1.56,-0.81)
				.. controls (-0.81,-1.56) and ( 0.24,-1.76) ..
				( 1.04,-1.21)
				-- cycle;
				
				% ---- 2. Bottom face Gamma_h^- ----
				\fill[bot face]
				(-2.09, 0.81)
				.. controls (-1.79, 1.66) and (-0.89, 2.06) ..
				(-0.09, 2.11)
				.. controls ( 0.71, 2.16) and ( 1.36, 1.76) ..
				( 1.56, 1.16)
				.. controls ( 1.76, 0.56) and ( 1.46,-0.44) ..
				( 0.66,-0.99)
				.. controls (-0.14,-1.54) and (-1.19,-1.34) ..
				(-1.94,-0.59)
				.. controls (-2.74, 0.16) and (-2.39, 0.46) ..
				(-2.09, 0.81);
				
				% ---- 3. Mid-surface omega (dashed green) ----
				\draw[midsurface]
				(-1.90, 0.70)
				.. controls (-1.60, 1.55) and (-0.70, 1.95) ..
				( 0.10, 2.00)
				.. controls ( 0.90, 2.05) and ( 1.55, 1.65) ..
				( 1.75, 1.05)
				.. controls ( 1.95, 0.45) and ( 1.65,-0.55) ..
				( 0.85,-1.10)
				.. controls ( 0.05,-1.65) and (-1.00,-1.45) ..
				(-1.75,-0.70)
				.. controls (-2.55, 0.05) and (-2.20, 0.35) ..
				(-1.90, 0.70);
				\node[green!50!black, font=\footnotesize] at (-0.15, 0.25) {$\omega$};
				
				% ---- 4. Top face Gamma_h^+ ----
				\fill[top face]
				(-1.71, 0.59)
				.. controls (-1.41, 1.44) and (-0.51, 1.84) ..
				( 0.29, 1.89)
				.. controls ( 1.09, 1.94) and ( 1.74, 1.54) ..
				( 1.94, 0.94)
				.. controls ( 2.14, 0.34) and ( 1.84,-0.66) ..
				( 1.04,-1.21)
				.. controls ( 0.24,-1.76) and (-0.81,-1.56) ..
				(-1.56,-0.81)
				.. controls (-2.36,-0.06) and (-2.01, 0.24) ..
				(-1.71, 0.59);
				
				% Omega_h interior label
				\node at (0.55,-0.40) {$\Omega_h$};
				
				% ---- Axes ----
				\draw[axis line] (-3.3,-2.2) -- (-3.3, 2.6) node[above] {$x_3$};
				\draw[axis line] (-3.3,-2.2) -- ( 3.0,-2.2) node[right] {$x'$};
				
				% ---- Thickness markers ----
				\draw[gray!60] (-3.5,-0.22) -- (-3.1,-0.22);
				\draw[gray!60] (-3.5, 0.22) -- (-3.1, 0.22);
				\node[left, gray!55!black, font=\footnotesize] at (-3.55,-0.22)
				{$+\frac{h}{2}$};
				\node[left, gray!55!black, font=\footnotesize] at (-3.55, 0.22)
				{$-\frac{h}{2}$};
				
				% ---- Gamma_h^w brace ----
				\draw[decorate,
				decoration={brace, amplitude=5pt, raise=4pt},
				gray!60!black]
				(-4.1, -0.22) -- (-4.1,0.22)
				node[midway, right=-25pt, font=\footnotesize] {$\Gamma_h^{w}$};
				
				% ---- Leader labels ----
				\draw[leader] (1.94, 0.94) -- (3.1, 1.55);
				\node[anchor=west, font=\footnotesize] at (3.1, 1.75)
				{$\Gamma_h^{+}$};
				\node[anchor=west, font=\scriptsize] at (3.1, 1.45)
				{$(\nu = +e_3)$};
				
				\draw[leader] (1.56, 1.16) -- (3.1, 2.55);
				\node[anchor=west, font=\footnotesize] at (3.1, 2.75)
				{$\Gamma_h^{-}$};
				\node[anchor=west, font=\scriptsize] at (3.1, 2.45)
				{$(\nu = -e_3)$};
				
				\draw[leader] (1.80,-0.05) -- (3.1, 0.60);
				\node[anchor=west, font=\footnotesize] at (3.1, 0.80)
				{$\Gamma_h^{\mathrm{lat}}$};
				\node[anchor=west, font=\scriptsize] at (3.1, 0.50)
				{(no BC)};
				
			\end{scope}
			
			% ============================================================
			%  DIVIDING LINE
			% ============================================================
			\draw[gray!30, densely dashed, line width=0.4pt]
			(1.2,-3.0) -- (1.2, 2.9);
			\begin{scope}[xshift=4.8cm, yshift=0.2cm]
				
				\def\s{0.48}% cell size
				
				% -------------------------------------------------------
				%  STEP 1: Draw background (all cells in bounding box) as
				%  white first, then paint interior teal and boundary amber.
				% -------------------------------------------------------
				
				% ------ boundary-layer cells (amber) drawn FIRST ------
				% All cells that overlap the blob (interior will be painted over).
				\foreach \ci/\cj in {
					-1/4, 0/4,
					-3/3, -2/3, -1/3, 0/3, 1/3, 2/3,
					-4/2, -3/2, -2/2, -1/2, 0/2, 1/2, 2/2, 3/2,
					-4/1, -3/1, -2/1, -1/1, 0/1, 1/1, 2/1, 3/1,
					-5/0, -4/0, -3/0, -2/0, -1/0, 0/0, 1/0, 2/0, 3/0,
					-5/-1, -4/-1, -3/-1, -2/-1, -1/-1, 0/-1, 1/-1, 2/-1, 3/-1,
					-4/-2, -3/-2, -2/-2, -1/-2, 0/-2, 1/-2, 2/-2, 3/-2,
					-3/-3, -2/-3, -1/-3, 0/-3, 1/-3, 2/-3,
					0/-4
				}{%
					\fill[bnd cell]
					({\ci*\s},{\cj*\s}) rectangle ({(\ci+1)*\s},{(\cj+1)*\s});
				}
				
				% ------ interior cells (teal) painted ON TOP ------
				% Solid rectangular block of cells fully inside the blob,
				% plus two extra cells on the left that are fully interior.
				\foreach \ci/\cj in {
					-3/2, -2/2, -1/2, 0/2, 1/2, 2/2,
					-1/3, 0/3,
					-3/1, -2/1, -1/1, 0/1, 1/1, 2/1,
					-4/0, -3/0, -2/0, -1/0, 0/0, 1/0, 2/0,
					-4/-1, -3/-1, -2/-1, -1/-1, 0/-1, 1/-1, 2/-1,
					-3/-2, -2/-2, -1/-2, 0/-2, 1/-2
				}{%
					\fill[int cell]
					({\ci*\s},{\cj*\s}) rectangle ({(\ci+1)*\s},{(\cj+1)*\s});
				}
				
				% ------ grid lines ------
				\draw[gray!38, line width=0.28pt, xstep=\s, ystep=\s]
				({-5.5*\s},{-4.5*\s}) grid ({4.5*\s},{5.0*\s});
				
				% ------ blob boundary (same as mid-surface from Part A) ------
				\draw[gray!70!black, line width=1.3pt]
				(-1.90, 0.70)
				.. controls (-1.60, 1.55) and (-0.70, 1.95) ..
				( 0.10, 2.00)
				.. controls ( 0.90, 2.05) and ( 1.55, 1.65) ..
				( 1.75, 1.05)
				.. controls ( 1.95, 0.45) and ( 1.65,-0.55) ..
				( 0.85,-1.10)
				.. controls ( 0.05,-1.65) and (-1.00,-1.45) ..
				(-1.75,-0.70)
				.. controls (-2.55, 0.05) and (-2.20, 0.35) ..
				(-1.90, 0.70);
				
				% omega label
				\node[font=\small] at (1.55, 1.75) {$\omega$};
				
				% ------ highlighted unit cell Y' ------
				\fill[hi cell]
				({0.0*\s},{0.0*\s}) rectangle ({1.0*\s},{1.0*\s});
				\node[blue!65!black, font=\footnotesize]
				at ({0.5*\s},{0.5*\s}) {};
				
				% ------ epsilon dimension annotation ------
				\draw[dim line]
				({-4.0*\s},{-4.2*\s}) -- ({-3.0*\s},{-4.2*\s})
				node[midway, below=2pt, font=\footnotesize] {$\varepsilon$};
				
				% ------ Lambda_eps brace on the right ------
				\draw[decorate,
				decoration={brace, amplitude=4pt, raise=3pt},
				orange!70!black, line width=0.7pt]
				({4.0*\s},{0.5*\s}) -- ({4.0*\s},{-0.5*\s})
				node[midway, right=8pt, orange!65!black, font=\footnotesize]
				{$\Lambda_\varepsilon$};
				
				% ------ omega_hat_eps label inside teal region ------
				\node[teal!60!black, font=\footnotesize]
				at ({-1.5*\s},{-0.5*\s}) {$\widehat{\omega}_\varepsilon$};
				
			\end{scope}
			
			% ============================================================
			%  LEGEND
			% ============================================================
			\begin{scope}[yshift=-3.2cm]
				% teal swatch
				\fill[int cell] (-4.0,-0.22) rectangle (-3.5, 0.22);
				\draw[gray!45, line width=0.25pt] (-4.0,-0.22) rectangle (-3.5, 0.22);
				\node[anchor=west, font=\footnotesize] at (-3.38, 0.00)
				{interior cell $(\varepsilon\xi+\varepsilon Y'\subset\omega)$,\;
					forming $\widehat{\omega}_\varepsilon$};
				
				% orange swatch
				\fill[bnd cell] (2.2,-0.22) rectangle (2.7, 0.22);
				\draw[gray!45, line width=0.25pt] (2.2,-0.22) rectangle (2.7, 0.22);
				\node[anchor=west, font=\footnotesize] at (2.80, 0.00)
				{cells containing the boundary layer $\Lambda_\varepsilon$};
			\end{scope}
			
		\end{tikzpicture}
\caption{%
	\textbf{Left:} The thin film $\Omega_h=\omega\times hI$.
	Weak anchoring is imposed on
	$\Gamma_h^{w}=\Gamma_h^{+}\cup\Gamma_h^{-}$, with outward normals
	$e_3$ and $-e_3$, respectively. The lateral boundary
	$\Gamma_h^{\mathrm{lat}}$ is hatched, and the dashed green curve
	represents the mid-surface $\omega$.
	\textbf{Right:} The $\varepsilon$-periodic decomposition of $\omega$.
	The teal region is $\widehat{\omega}_\varepsilon$, the union of the
	complete cells indexed by $\Xi_\varepsilon$, while the amber region is the cells containing
	the boundary strip $\Lambda_\varepsilon$. The cell outlined in blue is
	a representative $\varepsilon$-cell.}
		\label{fig:geometry}
	\end{figure}
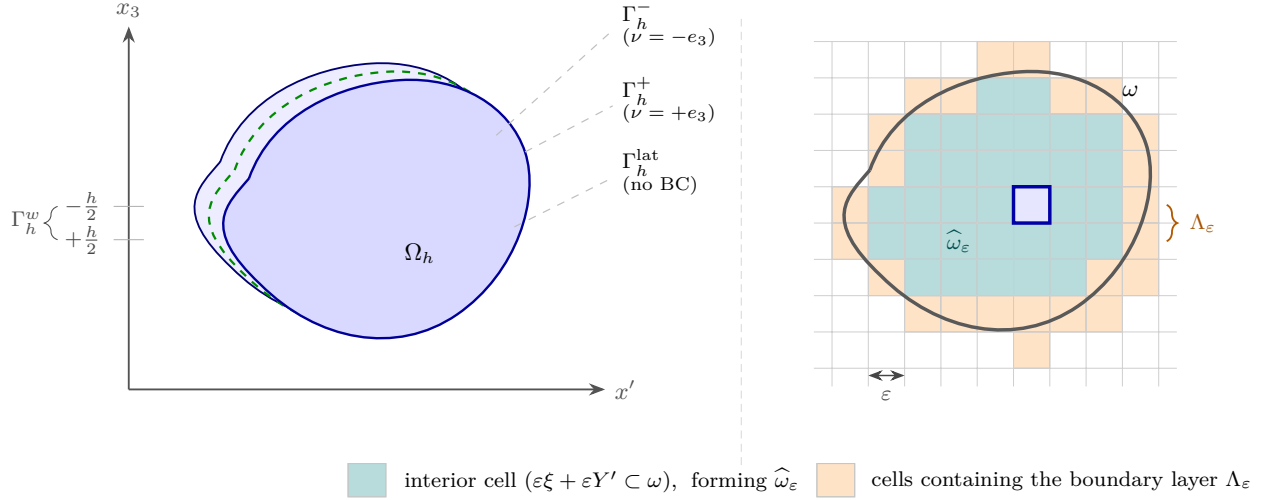
%	Now we introduce the periodicity structure of the thin heterogeneous nematic film; this framework is used again in Section~\ref{sec:unfolding} when we define the unfolding operators. 
%	We start by introducing the in-plane {unit cell}
%	$\Yp := (0,1)^{2}\subset\R^{2}$ and write a generic fast in-plane
%	variable as $y^{\prime}\in\Yp$. In the periodic setting $x'\in\R^2$ can be decomposed a.e.~as
%	$$x'=\left(\e\left[{x'\over \e}\right]+\e\left\{{x'\over \e}\right\}\right),\quad\text{where $\left[{x'\over \e}\right]\in \Z^2$, $\left\{{x'\over \e}\right\}\in Y'$}.$$
%	Set 
%	$$\Xi_\e=\left\{\xi\in\Z^2\,|\,\e\xi+\e Y'\subset \o\right\},\quad \wh\o_\e=\mathrm{interior}\left\{\bigcup_{\xi\Xi_\e}(\e\xi+\e \wo {Y'})\right\},\quad \Lambda_\e=\left(\o\setminus\wo{\wh \o_\e}\right),$$
%	where the boundary layer set $\Lambda_\e$ contains the parts of the cells intersecting the boundary $\partial\o$. Observe that $|\Lambda_\e|\leq C\e\to 0$ as $\e\to0$, since $\o$ is a bounded domain with Lipschitz boundary.
%	We also set $\wh \O_{h,\e}=\wh \o_\e\X hI$.
	
	Let $\Yp:=(0,1)^2$ be the in-plane unit cell. For almost every
	$x'\in\mathbb R^2$, we write
	\[
	x'
	=
	\varepsilon\left[\frac{x'}{\varepsilon}\right]
	+
	\varepsilon\left\{\frac{x'}{\varepsilon}\right\},
	\qquad
	\left[\frac{x'}{\varepsilon}\right]\in\mathbb Z^2,
	\quad
	\left\{\frac{x'}{\varepsilon}\right\}\in\Yp.
	\]
	Define the indices of the cells contained in $\omega$ by
	\[
	\Xi_\varepsilon
	:=
	\left\{
	\xi\in\mathbb Z^2:
	\varepsilon\xi+\varepsilon\Yp\subset\omega
	\right\},\quad\text{and set}\quad
	\widehat\omega_\varepsilon
	:=
	\operatorname{int}
	\left(
	\bigcup_{\xi\in\Xi_\varepsilon}
	\bigl(\varepsilon\xi+\varepsilon\overline{\Yp}\bigr)
	\right),
	\qquad
	\Lambda_\varepsilon
	:=
	\omega\setminus\overline{\widehat\omega_\varepsilon}.
	\]
	The boundary strip $\Lambda_\varepsilon$ consists of the parts of the cells
	intersecting $\partial\omega$. Since $\omega$ is a bounded Lipschitz domain, $|\Lambda_\varepsilon|
	\leq C\varepsilon$
	for all sufficiently small $\varepsilon>0$. Finally, define
	\[
	\widehat\Omega_{h,\varepsilon}
	:=
	\widehat\omega_\varepsilon\times hI.
	\]

\subsection{The Landau--de~Gennes energy on $\Ohmh$}
\label{sec:energy}
%=====================================================================

The {Landau--de~Gennes free energy} of the heterogeneous nematic
film is the functional
$\Fhe : H^{1}(\Ohmh;\Szero)\to\R$ defined by
\begin{multline}
	\label{eq:energy}
	\Fhe(Q)=E^{h,\e}_{\mathrm{el}}(Q)+E^{h,\e}_{\mathrm{bulk}}(Q)
	+E^{h,\e}_{\mathrm{surf}}(Q)\\
	=\underbrace{\int_{\O_h}f^{h,\e}_e(\nabla Q)\,dx
	}_{\displaystyle\text{(elastic energy)}}
	+
	\underbrace{
		\int_{\Ohmh}\!f^{h,\e}_b(Q)\,dx
	}_{\displaystyle\text{(bulk potential)}}
	+
	\underbrace{
		\int_{\Gamma_h^w}f^{h,\e}_s(Q)\,\dd\Hc^2
	}_{\displaystyle\text{(surface anchoring)}}.
\end{multline}
\textbf{Elastic term $E^{h,\e}_{\mathrm{el}}$:}
The elastic energy density is
\[
f^{h,\e}_e(\nabla Q)=
\frac{1}{2}\bA\!\Bigl(\frac{x^{\prime}}{\eps},\frac{x_3}{h}\Bigr)
\nabla Q : \nabla Q,
\quad\forall\,Q\in H^1(\O_h;\Szero),
\]
using Einstein summation in the compact notation
$\bA(y)\,P:P := A_{ijkl}^{\alpha\beta}(y)\,P_{ij\alpha}\,P_{kl\beta}$.

The heterogeneous elastic tensor
$\bA : Y=\Yp\times\Iint\to\mathcal{L}(\Tgrad,\Tgrad)$
encodes the spatially varying elastic response of the composite.
Its dependence on $x$ is through the two-scale argument
$(x^{\prime}/\eps,\,x_3/h)$: the in-plane variable $x^{\prime}/\eps$
oscillates rapidly at the microstructure scale $\eps$, while $x_3/h$
is the normalised thickness coordinate.

\begin{assumption}[Elastic tensor $\bA$]
	\label{asm:elastic}
	The tensor
	$\bA = (A_{ijkl}^{\alpha\beta}) : Y
	\to \mathcal{L}(\Tgrad,\Tgrad)$
	satisfies:
	\begin{enumerate}[label=\textnormal{(A\arabic*)},leftmargin=3em]
		\item\label{A1}
		\textbf{Measurability and in-plane $\Yp$-periodicity.}
		For a.e.\ $y_3\in\Iint$, the map
		$y^{\prime}\mapsto A_{ijkl}^{\alpha\beta}(y^{\prime},y_3)$ is
		$\Yp$-periodic and bounded. Globally,
		\[
		\bA \;\in\;
		L^{\infty}\!\bigl(\Iint;\,
		L^{\infty}_{\mathrm{per}}(\Yp;\mathcal{L}(\Tgrad,\Tgrad))\bigr).
		\]
		No periodicity in $y_3$ is assumed.
		
		\item\label{A2}
		\textbf{Minor and major symmetry.}
		For a.e.\ $(y^{\prime},y_3)\in Y=\Yp\times\Iint$:
		\[
		A_{ijkl}^{\alpha\beta}(y^{\prime},y_3)
		\;=\; A_{klij}^{\beta\alpha}(y^{\prime},y_3),\quad A_{ijkl}^{\alpha\beta}(y^{\prime},y_3)=A_{jikl}^{\alpha\beta}(y^{\prime},y_3),\quad A_{ijkl}^{\alpha\beta}(y^{\prime},y_3)=A_{ijlk}^{\alpha\beta}(y^{\prime},y_3) .
		\]
		\item\label{A3}
		\textbf{Uniform ellipticity on $\Tgrad$.}
		There exist constants $0<\lambda\leq\Lambda<\infty$ such that
		for a.e.\ $(y^{\prime},y_3)\in\Yp\times\Iint$ and every
		$P\in\Tgrad$:
		\[
		\lambda\,\abs{P}^{2}
		\;\leq\;
		A_{ijkl}^{\alpha\beta}(y^{\prime},y_3)\,
		P_{ij\alpha}\,P_{kl\beta}
		\;\leq\;
		\Lambda\,\abs{P}^{2}.
		\]
	\end{enumerate}
\end{assumption}

\begin{remark}[The isotropic LdG elastic energy]
	\label{rem:L1L2L3}
	The classical isotropic LdG elastic density \cite{Lon+87,CZ20}
	\[
	f_{e}^{\mathrm{LdG}}(\nabla Q)
	= \frac{L_1}{2}\,\partial_k Q_{ij}\,\partial_k Q_{ij}
	+ \frac{L_2}{2}\,\partial_j Q_{ij}\,\partial_k Q_{ik}
	+ \frac{L_3}{2}\,\partial_k Q_{ij}\,\partial_j Q_{ik}
	\]
	is a special case with constant tensor components
	$A_{ijkl}^{\alpha\beta}
	= L_1\,\delta_{\alpha\beta}\delta_{ik}\delta_{jl}
	+ L_2\,\delta_{i\alpha}\delta_{\beta l}\delta_{jk}
	+ L_3\,\delta_{i\beta}\delta_{\alpha k}\delta_{jl}$.
	Assumption~\ref{asm:elastic}\ref{A3} is equivalent to
	$L_1>0$, $-L_1<L_3<2L_1$, and
	$-\frac{3}{5}L_1-\frac{1}{10}L_3<L_2$; see \cite{Lon+87}.
\end{remark}

\smallskip
%------------------------------------------------------------------
\textbf{Bulk term $E^{h,\e}_{\mathrm{bulk}}$:}
The LdG bulk potential $f^{h,\e}_b:\Szero\to\R$ models the
isotropic-to-nematic phase transition. We take the standard quartic
polynomial
\begin{equation}
	\label{eq:fb}
	f^{h,\e}_b(Q)=f_b(Q) \;:=\, a\,\tr(Q^{2}) - b\,\tr(Q^{3})
	+ c\,\bigl(\tr(Q^{2})\bigr)^{2},
	\qquad a,b\in\R,\quad c>0.
\end{equation}
The coefficient $a$ is proportional to temperature;
$b>0$ controls the first-order character of the transition; and $c>0$
ensures the potential is bounded below.

\begin{assumption}[Bulk potential $f^{h,\e}_b$]
	\label{asm:bulk}
	The bulk potential is given by \cref{eq:fb} with $a,b\in\R$ and $c>0$.
\end{assumption}
All coercivity and growth properties follow from the above assumption
and are recorded in Lemma~\ref{lem:fB}.

%------------------------------------------------------------------
\textbf{Surface anchoring:}
The surface energy is formulated in a general set-valued framework that
accommodates the principal physically relevant anchoring geometries,
including spatially heterogeneous preferred states on each face
independently. We recall 
$\mathbb{S}^2:=\{x\in \R^3\,:\, \|x\|_2=1\}$. 
Let
\[
\mathcal{A} : \partial Y \times \mathbb{S}^2
\;\longrightarrow\; 2^{\Szero}
\]
be a set-valued map assigning to each pair
$(y,\nu)\in\partial Y\times\mathbb{S}^2$ a non-empty closed set
$\mathcal{A}(y,\nu)\subset\Szero$ of preferred boundary states,
where $y=(y',y_3)\in\partial Y=(Y'\times\partial I)\cup(\partial Y'\times\bar{I})$
is the fast boundary variable.

\begin{definition}[Rescaled anchoring data]
	\label{def:anchoring-data}
	Let $W:\partial Y\times\mathbb{S}^2\to[0,\infty)$ be $Y'$-periodic in $y'$,
	Borel measurable, and essentially bounded.
	For $h,\varepsilon>0$ and anchoring exponent $\gamma\geq 0$, define
	$W^{h,\varepsilon}:\partial\Omega_h\times\mathbb{S}^2\to[0,\infty)$ by
	\[
	W^{h,\varepsilon}(x,\nu)
	:=
	\begin{cases}
		h^\gamma\,W\!\Bigl(\dfrac{x'}{\varepsilon},\;\dfrac{x_3}{h},\;\nu\Bigr)
		& x\in\Gamma_h^{\pm}, \\[8pt]
		0 & x\in\Gamma_h^{\mathrm{lat}},
	\end{cases}
	\]
	and the rescaled anchoring set
	$\mathcal{A}^{h,\varepsilon}:\Gamma_h^{w}\times\mathbb{S}^2\to 2^{\Szero}$ by
	\[
	\mathcal{A}^{h,\varepsilon}(x,\nu)
	:= \mathcal{A}\!\Bigl(\frac{x'}{\varepsilon},\;\frac{x_3}{h},\;\nu\Bigr),
	\qquad x\in\Gamma_h^{w},\quad\nu\in\mathbb{S}^2.
	\]
\end{definition}

For $x\in\Gamma_h^w=\Gamma_h^+\cup\Gamma_h^-$, the rescaled spatial argument
satisfies $(x'/\varepsilon,\,x_3/h)=(x'/\varepsilon,\,\pm\tfrac{1}{2})
\in Y'\times\partial I\subset\partial Y$, so $\mathcal{A}^{h,\varepsilon}(x,\nu)$
is well-defined. The zero value of $W^{h,\varepsilon}$ on $\Gamma_h^{\mathrm{lat}}$
renders the value of $\mathcal{A}$ on $\partial Y'\times\bar{I}$ irrelevant.

The surface energy density is
\[
f_s^{h,\varepsilon}(Q,x,\nu)
:= W^{h,\varepsilon}(x,\nu)\,
\operatorname{dist}^2\!\bigl(Q,\,\mathcal{A}^{h,\varepsilon}(x,\nu)\bigr),
\qquad
Q\in\Szero,\quad x\in\Gamma_h^{w},\quad\nu\in\mathbb{S}^2,
\]
where
$\operatorname{dist}(Q,\mathcal{A}^{h,\varepsilon}(x,\nu))
:=\inf_{A\in\mathcal{A}^{h,\varepsilon}(x,\nu)}|Q-A|$.
The total surface energy is
\[
E_{\mathrm{surf}}^{h,\varepsilon}(Q)
:= \int_{\Gamma_h^{w}}
f_s^{h,\varepsilon}\!\bigl(Q(x),x,\nu(x)\bigr)\,d\mathcal{H}^2(x),
\]
where $\nu(x)=+e_3$ on $\Gamma_h^+$ and $\nu(x)=-e_3$ on $\Gamma_h^-$.
Splitting $\Gamma_h^w=\Gamma_h^+\cup\Gamma_h^-$:
\begin{multline}\label{eq:Esurf-split}
	E_{\mathrm{surf}}^{h,\varepsilon}(Q)
	=
	\int_{\Gamma_h^{+}}
	W^{h,\varepsilon}(x,e_3)\,
	\operatorname{dist}^2\!\bigl(Q,\mathcal{A}^{h,\varepsilon}(x,e_3)\bigr)
	\,d\mathcal{H}^2 \\
	+\;
	\int_{\Gamma_h^{-}}
	W^{h,\varepsilon}(x,-e_3)\,
	\operatorname{dist}^2\!\bigl(Q,\mathcal{A}^{h,\varepsilon}(x,-e_3)\bigr)
	\,d\mathcal{H}^2.
\end{multline}
The two integrands involve
$\mathcal{A}(x'/\varepsilon,+\tfrac{1}{2},e_3)$ and
$\mathcal{A}(x'/\varepsilon,-\tfrac{1}{2},-e_3)$ respectively,
which need not coincide.

\begin{assumption}[Anchoring data]\label{asm:anchoring}
	The set-valued map
	$\mathcal{A}:\partial Y\times\mathbb{S}^2\to 2^{\Szero}$ satisfies:
	\begin{enumerate}[label=\textup{(S\arabic*)},ref=\textup{S\arabic*}]
		
		\item\label{SS1}
		\textbf{Non-emptiness, closedness, and measurability.}
		For every $\nu\in\mathbb{S}^2$ and a.e.\ $y\in\partial Y$, the set
		$\mathcal{A}(y,\nu)$ is non-empty and closed in $\Szero$.
		For every $Q\in\Szero$ and every $\nu\in\mathbb{S}^2$, the map
		$y\mapsto\operatorname{dist}(Q,\mathcal{A}(y,\nu))$ is Borel measurable
		on $\partial Y$.
		
		\item\label{SS2}
		\textbf{Uniform local boundedness.}
		There exists $C_{\mathcal{A}}>0$ such that for every $\nu\in\mathbb{S}^2$
		and a.e.\ $y\in\partial Y$, there exists $A_0\in\mathcal{A}(y,\nu)$
		with $|A_0|\leq C_{\mathcal{A}}$.
		
	\end{enumerate}
\end{assumption}

\begin{remark}[Anchoring framework and comparison with \cite{GMS15}]
	\label{rem:anchoring-framework-comments}
	Two assumptions commonly imposed in the literature are unnecessary in the
	present setting. First, no continuity of
	$\nu\mapsto\operatorname{dist}(Q,\mathcal{A}(y,\nu))$ is required, since the
	surface energy is supported on
	$\Gamma_h^w=\Gamma_h^+\cup\Gamma_h^-$, where the outer unit normal takes only
	the values $\pm e_3$. Second, no symmetry between the two faces is imposed:
	the preferred sets and anchoring weights may be prescribed independently on
	$\Gamma_h^+$ and $\Gamma_h^-$.
	
	The homogeneous weak anchoring potential considered in \cite{GMS15} has a
	zero set of the form
	\[
	\{Q\in\mathbb{S}_0:(Q-Q_0)\nu=0\},
	\]
	and is therefore covered, at the level of its preferred set, by the affine
	constraint in~\ref{E3}. Its full quadratic density, however, assigns
	generally different weights to the normal and tangential components of
	$(Q-Q_0)\nu$ and hence is not, in general, a scalar multiple of
	$\operatorname{dist}^2(Q,\mathcal{A}(\nu))$ with respect to the Frobenius
	norm. Thus, except for the corresponding special choice of coefficients, the
	present framework reproduces the zero set of the anchoring potential in
	\cite{GMS15}, but not its full surface density.
\end{remark}

%------------------------------------------------------------------
\textbf{Additional structure for the anchoring-dominated regime.}
The following assumption is imposed in addition to
Assumption~\ref{asm:anchoring} exclusively in the regime $0\leq\gamma<1$,
and is not required elsewhere. It captures the situation in which the
anchoring geometry is uniform across the in-plane microstructure, so that
the effective constraint set is determined solely by the face normals.

\begin{assumption}[Compatible anchoring]\label{asm:compatible}
	The map $\mathcal{A}(y,\nu)$ is independent of $y'$, so that
	$\mathcal{A}(\cdot,\nu)$ depends only on $y_3$ and $\nu$. The
	effective constraint set
	\begin{equation}\label{eq:Astar}
		\mathscr{A}
		:= \mathcal{A}\!\left(+\tfrac{1}{2},+e_3\right)
		\cap \mathcal{A}\!\left(-\tfrac{1}{2},-e_3\right)
		\subset \Szero
	\end{equation}
	is non-empty and closed, and there exists $Q^*\in H^1(\omega;\mathscr{A})$.
	The anchoring weight satisfies the uniform lower bound
	\[
	\operatorname{ess\,inf}_{y'\in Y'}
	W\!\left(y',\pm\tfrac{1}{2},\pm e_3\right) \geq w_0 > 0.
	\]
	Moreover, $\mathscr{A}$ satisfies one of the following geometric conditions:
	\begin{enumerate}[label=\textup{(\roman*)}]
		\item $\mathscr{A}$ is a closed convex subset of $\Szero$; or
		\item $\mathscr{A}$ is a compact smooth submanifold without boundary
		of $\Szero$; or
		\item For fixed constants $0<s_0<s_1<1$,
		\begin{equation}\label{eq:csmwb+}
			\mathscr{A}
			=
			\Bigl\{s\Bigl(P-\tfrac{1}{3}I\Bigr):
			P\in\mathbb{P},\ s\in[s_0,s_1]\Bigr\},
			\qquad
			\mathbb{P}
			:=
			\{P\in\mathbb{R}^{3\times3}_{\mathrm{sym}}:
			P^2=P,\ \operatorname{tr}P=1\},
		\end{equation}
		where $I$ is the identity matrix.
		Then $\mathbb{P}$ is naturally diffeomorphic to the space of unoriented directions
		$\mathbb{S}^2/\{\pm 1\}\cong\mathbb{RP}^2$, and $\mathscr{A}$ is a
		compact smooth submanifold with boundary of $\Szero$, diffeomorphic
		to $[s_0,s_1]\times\mathbb{RP}^2$. Observe that $P\in\mathbb{P}$ implies $\mathrm{rank}(P)=1$.
	\end{enumerate}
	Furthermore, in the regime $\rho=0$, the scales additionally
satisfy
\begin{equation}\label{eq:subcritical-scale}
	\frac{\varepsilon^2}{h}
	\longrightarrow 0
	\qquad\text{as }(h,\varepsilon)\to(0,0).
\end{equation}
\end{assumption}
\begin{remark}
A further restriction on the relative rate of $h$ and $\e$, of the type \cref{eq:subcritical-scale}, is also required in the regime $h/\e\to0$ in \cite{CGOC_rigid,cherdantsev2015bending,velvcic2015derivation} for simultaneous homogenisation and dimension reduction in nonlinear elasticity, where it arises in the bending and constrained von Kármán regimes, whereas here it is inherited from the anchoring strength $h^\gamma$.
\end{remark}
Closedness of $\mathscr{A}$ follows immediately from~\eqref{eq:Astar}:
each of $\mathcal{A}(\pm\tfrac{1}{2},\pm e_3)$ is closed in $\Szero$
by Assumption~\ref{asm:anchoring}\ref{SS1}, and $\mathscr{A}$ is
their intersection. 		Density of smooth maps in $H^1(\omega;\mathscr{A})$ follows from
Lemma~\ref{lem:density-Astar}.

\begin{proposition}[Representative anchoring models]
	\label{prop:compatible-examples}
	The following examples satisfy Assumption~\ref{asm:anchoring}. Suppose,
	in addition, that the preferred sets are independent of $y'$ on each face
	and that the anchoring weight $W$ satisfies the positive lower bound in
	Assumption~\ref{asm:compatible}. Then the effective constraint set
	$\mathscr{A}$ is as stated below. Assumption~\ref{asm:compatible} holds
	whenever the indicated compatibility condition between the two faces is
	satisfied.
	
	\begin{enumerate}[label=\textup{(E\arabic*)},ref=\textup{E\arabic*}]
		
		\item\label{E1}
		\textbf{Prescribed $Q$-tensor anchoring.}
		Let $Q_0:\partial Y\times\mathbb{S}^2\to\Szero$ be $Y'$-periodic
		in $y'$, essentially bounded, and measurable in $y$ for each fixed
		$\nu$. Set $\mathcal{A}(y,\nu):=\{Q_0(y,\nu)\}$.
		The corresponding surface density is
		\[
		f_s^{h,\varepsilon}(Q,x,\nu)
		=
		W^{h,\varepsilon}(x,\nu)
		\bigl|Q-Q_0(x'/\varepsilon,x_3/h,\nu)\bigr|^2.
		\]
		Assumption~\ref{asm:anchoring} holds with
		$A_0(y,\nu):=Q_0(y,\nu)$ and
		$C_{\mathcal{A}}:=\|Q_0\|_{L^\infty}$.
		
		Under $y'$-independence, define
		\[
		Q_0^+:=Q_0\Bigl(\frac12,e_3\Bigr),
		\qquad
		Q_0^-:=Q_0\Bigl(-\frac12,-e_3\Bigr).
		\]
		Then $\mathscr{A}=\{Q_0^+\}\cap\{Q_0^-\}$.
		Hence $\mathscr{A}\neq\varnothing$ if and only if
		$Q_0^+=Q_0^-=:Q^*$, in which case
		$\mathscr{A}=\{Q^*\}$ and
		Assumption~\ref{asm:compatible}(i) holds.
		In the anchoring-dominated regime $0\leq\gamma<1$, the admissible
		limit is $Q\equiv Q^*$ and its energy is $|\omega|f_b(Q^*)$.
		A standard homogeneous example is $Q_0(\nu)
		=
		s_+\Bigl(\nu\otimes\nu-\frac13I\Bigr)$,
		for which $Q_0(e_3)=Q_0(-e_3)$; see
		\cite{RP69,CZ20}.
		
		\item\label{E2}
		\textbf{Vanishing normal action.}
		Set $\mathcal{A}(y,\nu)
		:=
		\{Q\in\Szero:Q\nu=0\}$.
		This is a closed linear subspace of $\Szero$ of codimension~$3$,
		independent of $y$. Assumption~\ref{asm:anchoring} holds with
		$A_0:=0$. If $P_{\mathcal{A}(\nu)}$ denotes the orthogonal
		projection in $\Szero$ onto $\mathcal{A}(\nu)$ and $\Pi_\nu:=I-P_{\mathcal{A}(\nu)}$,
		then $\operatorname{dist}^2
		\bigl(Q,\mathcal{A}(y,\nu)\bigr)
		=
		|\Pi_\nu Q|^2$.
		Since $\mathcal{A}(e_3)=\mathcal{A}(-e_3)$, $\mathscr{A}
		=
		\{Q\in\Szero:Qe_3=0\}$.
		This is a non-empty closed linear subspace, and
		Assumption~\ref{asm:compatible}(i) holds.
		
		\item\label{E3}
		\textbf{Prescribed normal action.}
		Let $Q_0$ be as in~\ref{E1} and set
		\[
		\mathcal{A}(y,\nu)
		:=
		\{Q\in\Szero:(Q-Q_0(y,\nu))\nu=0\}.
		\]
		This is an affine translate of the subspace in~\ref{E2}.
		Assumption~\ref{asm:anchoring} holds with
		$A_0(y,\nu):=Q_0(y,\nu)$ and
		$C_{\mathcal{A}}:=\|Q_0\|_{L^\infty}$. Moreover, $\operatorname{dist}
		\bigl(Q,\mathcal{A}(y,\nu)\bigr)
		=
		\bigl|\Pi_\nu\bigl(Q-Q_0(y,\nu)\bigr)\bigr|$.
		Under $y'$-independence, let $Q_0^\pm$ be defined as in~\ref{E1}.
		Then
		\[
		\mathscr{A}
		=
		\{Q\in\Szero:Qe_3=Q_0^+e_3\}
		\cap
		\{Q\in\Szero:Qe_3=Q_0^-e_3\}.
		\]
		Consequently, $\mathscr{A}\neq\varnothing$ if and only if $Q_0^+e_3=Q_0^-e_3$.
		Under this condition, $\mathscr{A}
		=
		\{Q\in\Szero:Qe_3=Q_0^+e_3\}$,
		which is a closed affine subspace, and
		Assumption~\ref{asm:compatible}(i) holds. In particular,
		the constant map $Q^*\equiv Q_0^+$ belongs to
		$H^1(\omega;\mathscr{A})$.
		
		\item\label{E4}
		\textbf{Fixed-order degenerate planar anchoring.}
		For a fixed $s_0\in\mathbb{R}\setminus\{0\}$, set
		\[
		\mathcal{A}(y,\nu)
		:=
		\left\{
		s_0\left(t\otimes t-\frac13I\right):
		t\in\mathbb{S}^2,\ t\cdot\nu=0
		\right\}.
		\]
		This set is independent of $y$ and is a compact smooth
		submanifold of $\Szero$ diffeomorphic to $\mathbb{RP}^1$.
		Assumption~\ref{asm:anchoring} holds with $C_{\mathcal{A}}
		=
		|s_0|\sqrt{\frac23}$.
		Since $\mathcal{A}(e_3)=\mathcal{A}(-e_3)$,
		\[
		\mathscr{A}
		=
		\left\{
		s_0\left(t\otimes t-\frac13I\right):
		t\in\mathbb{S}^2,\ t\cdot e_3=0
		\right\}.
		\]
		Thus $\mathscr{A}$ is a non-empty compact smooth submanifold of
		$\Szero$, diffeomorphic to $\mathbb{RP}^1$, and
		Assumption~\ref{asm:compatible}(ii) holds.
	\end{enumerate}
\end{proposition}

\begin{proof}
	In~\ref{E1} and~\ref{E3}, the measurability required in
	Assumption~\ref{asm:anchoring} follows from the measurability of $Q_0$;
	in~\ref{E2} and~\ref{E4}, it is immediate because the preferred sets
	are independent of $y$. The stated choices of $A_0$ give the required
	uniformly bounded selections.
	
	Under $y'$-independence, $\mathscr{A}$ is the intersection of the
	preferred sets on the two faces. The formulas above follow directly
	from the identities
	\[
	\mathcal{A}(e_3)=\mathcal{A}(-e_3)
	\quad\text{in \ref{E2} and \ref{E4}},
	\]
	and
	\[
	(Q-Q_0^-)(-e_3)=0
	\quad\Longleftrightarrow\quad
	Qe_3=Q_0^-e_3
	\quad\text{in \ref{E3}}.
	\]
	The effective sets are closed and have the geometric properties stated
	in each case. Whenever $\mathscr{A}\neq\varnothing$, any constant map
	with value in $\mathscr{A}$ belongs to $H^1(\omega;\mathscr{A})$.
	Together with the separately imposed lower bound on $W$, this verifies
	Assumption~\ref{asm:compatible}.
\end{proof}

We end this subsection by defining the three-dimensional minimisation problem 
\begin{equation}
	\label{eq:min-problem}
	\mhe \;:=\; \inf_{Q\in H^{1}(\Ohmh;\Szero)}\Fhe(Q),
\end{equation}
on the thin nematic film.

\subsection{Statement of the main results}
\label{sec:main-results}

We state the principal results of this paper. Precise definitions of all
objects appearing below are given in the sections indicated; the reader
may treat forward references as pointers to the relevant constructions.

We begin by defining the set of admissible limit $Q$-tensors.
The admissible space of limit fields is
\[
\A^\gamma_\rho 
:=H^1(\o;\mathscr{E}_\gamma)\X\C_\rho,\quad\text{where}\quad \mathscr{E}_\gamma=\left\{\begin{aligned}
	&\Szero,\quad &&\text{for $\gamma\geq1$},\\
	&\mathscr{A} ,\quad 	&&\text{for $0\leq \gamma<1$},
\end{aligned}\right.
\]
and the microscopic $Q$-tensor spaces $\C_\rho$ are defined 
regime by regime as follows.

\textit{Regime $\varrho = 0$.}
\begin{equation*}
	(\widehat{Q}_0,\,\widehat{R}) \in \C_0 
	:= L^2\!\bigl(\omega;\,H^1_{\mathrm{per},0}(Y';\,\Szero)\bigr)
	\times
	L^2\!\bigl(\omega \times Y';\,H^1_\sharp(I;\,\Szero)\bigr).
\end{equation*}
Observe that $\widehat{Q}_0$ is  $y_3$-independent.

\noindent\textit{Regime $\varrho = \infty$.}
\begin{equation*}
	(\widehat{Q}_\infty,\,\widehat{R}) \in\C_\infty 
	:=  L^2\!\bigl(\omega \times I;\,H^1_{\mathrm{per},0}(Y';\,\Szero)\bigr)
	\times
	L^2\!\bigl(\omega;\,H^1_\sharp(I;\,\Szero)\bigr).
\end{equation*}
Observe that $\widehat{R}$ is $y'$-independent. 

\noindent\textit{Regime $\varrho \in (0,\infty)$.}
\begin{multline*}
	\C_\rho 
	:= \Bigl\{
	(\widehat{Q}_\rho,\,\widehat{R}) 
	\in L^2\!\bigl(\omega;\,H^1_{\mathrm{per},0}(Y;\,\Szero)\bigr)
	\times
	L^2\!\bigl(\omega;\,H^1_{\mathrm{per},\sharp}(Y;\,\Szero)\bigr)
	:\\
	\quad\exists\;
	\widehat{Q} \in L^2\!\bigl(\omega;\,H^1_{\mathrm{per},0}(Y';\,\Szero)\bigr)
	\ \text{ such that }\ 
	\widehat{Q}_\rho = \widehat{Q} + \varrho\,\widehat{R}
	\Bigr\}.
\end{multline*}
Observe that the in-plane corrector has the coupled form
$\widehat{Q}_\rho = \widehat{Q} + \varrho\,\widehat{R}$, linking the
in-plane and transverse microscopic degrees of freedom.

Each space $\C_\rho$ is a closed subspace of the ambient Hilbert 
product space. 	Hence 
$\A^\gamma_\rho$ is 
a closed subset of a Hilbert space in each regime. Moreover for $\gamma\geq 1$, it is a Hilbert space.

\textbf{Throughout this paper, we assume that the limit $\rho$ in \cref{eq:rho} exists as  $(h,\e)\to(0,0)$.}

\begin{definition}
	\label{def:hom-conv}
	Let $Q_{h,\varepsilon}\in H^1(\Omega_h;\Szero)$. We say that
	$Q_{h,\varepsilon}$ converges to the triplet
	$(\Qc,\wh Q_\rho,\wh R)\in\A^\gamma_\rho$, and write
	\[
	Q_{h,\varepsilon}\rightharpoonup^{\mathrm{Qhom}}(\Qc,\wh Q_\rho,\wh R)
	\]
	if and only if 
	\[
	\begin{aligned}
		\Peh(Q_{h,\e}) \;&\to\; \Qc\;&&\text{strongly in $L^p(\o\X Y;\Szero)$ for $2\leq p<4$, and weakly in $L^4(\o\X Y;\Szero)$},\\
		\Te(Q^\pm_{h,\varepsilon}) &\to \Qc\;&&\text{strongly in }L^2(\omega\times Y';\Szero),\\
		\Peh(\nabla Q_{h,\e})&\rightharpoonup \mathcal G_\rho\;&&\text{weakly in }L^2(\omega\times Y;\Szero\otimes\R^3),
	\end{aligned}
	\]
	where $Q^\pm_{h,\e}$ denote the traces on the top and bottom, $\Pi_{h,\e}$, $\Te$ denote the rescaling unfolding and classical unfolding operators, respectively, defined in  Section \ref{sec:unfolding}, and $\mathcal G_\rho$ is defined from $(\Qc,\wh Q_\rho,\wh R)\in \A^\gamma_\rho$ as in Lemma \ref{prop:admissible-correctors}.  
\end{definition}

\paragraph{Sequential $\Gamma$-convergence theorem}

\begin{theorem}
	\label{thm:gamma}
	Under Assumptions~\ref{asm:elastic}, \ref{asm:bulk}, and~\ref{asm:anchoring}
	for $\gamma\geq 1$, and additionally Assumption~\ref{asm:compatible} for
	$0\leq\gamma<1$, and
	$\rho\in[0,\infty]$, the rescaled
	microscopic energies $h^{-1}\GF_{h,\varepsilon}$ sequentially
	$\Gamma$-converge, with respect to the convergence
	$\rightharpoonup^{\mathrm{Qhom}}$ of Definition~\ref{def:hom-conv}, to the
	two-scale limit energy $\GF^\gamma_\rho$ defined by
	\begin{equation}\label{eq:limE}
		\GF^\gamma_\rho(\Qc,\wh Q_\rho,\wh R)=E_{\mathrm{el}}^\rho(\Qc,\wh Q_\rho,\wh R)+E_{\mathrm{bulk}}(\Qc)+E^\gamma_{\mathrm{surf}}(\Qc),\quad 
	\end{equation}
	where 
	\begin{itemize}
		\item The two-scale elastic energy is given by
		\begin{equation*}
			E_{\mathrm{el}}^\rho(\Qc,\wh Q_\rho,\wh R)={1\over 2}\int_{\o\X Y}\bA \Gc_\rho :\Gc_\rho\,dydx'
		\end{equation*}
		with $\Gc_\rho$ defined from $(\Qc,\wh Q_\rho,\wh R)$ as in Lemma \ref{prop:admissible-correctors}. 
		\item The limit bulk energy is given by
		$$	E_{\mathrm{bulk}}(\Qc)
		=
		\int_\omega f_b(\Qc)\,dx',$$
		with $f_b$ given in \cref{eq:fb}.
		\item The limit surface energy is given by
		$$ E^\gamma_{\mathrm{surf}}(\Qc):= \left\{
		\begin{aligned}
			&			
			\int_\omega \bar{f}_s(\Qc(x'))\,dx',\quad && \text{for $\gamma=1$},\\
			& 0,\quad &&\text{for $\gamma\neq1$},
		\end{aligned}\right.$$
		where
		\begin{equation}
			\bar{f}^\pm_s(\Qc) 
			:= \int_{Y'} W\!\left(y', \pm\frac{1}{2}, \pm e_3\right)
			\mathrm{dist}^2\!\left(\Qc,\, \Ac\!\left(y', \pm\frac{1}{2}, \pm e_3\right)\right)dy',
			\qquad \Qc \in \Szero,
			\label{eq:hom-surf-density-pm}
		\end{equation}
		and their sum
		\begin{equation}
			\bar{f}_s(\Qc) := \bar{f}^+_s(\Qc) + \bar{f}^-_s(\Qc), 
			\qquad \Qc \in \Szero.
			\label{eq:hom-surf-density}
		\end{equation}
	\end{itemize}
	More precisely:

	\noindent\textbf{\emph{(i) Compactness.}}
	Let $Q_{h,\varepsilon}\in H^1(\Omega_h;\Szero)$ satisfy
	\[
	\limsup_{(h,\varepsilon)\to(0,0)}
	\frac1h\GF_{h,\varepsilon}(Q_{h,\varepsilon})
	<+\infty .
	\]
	Then there exists a subsequence, not relabelled, and a triplet $(\Qc,\wh Q_\rho,\wh R)\in\A^\gamma_\rho$
	such that
	\[
	Q_{h,\varepsilon}\rightharpoonup^{\mathrm{Qhom}}(\Qc,\wh Q_\rho,\wh R) .
	\]

	\noindent\textbf{\emph{(ii) $\Gamma$-liminf inequality.}}
	If
	\[
	Q_{h,\varepsilon}\rightharpoonup^{\mathrm{Qhom}}(\Qc,\wh Q_\rho,\wh R) ,
	\]
	then
	\begin{equation}
		\label{eq:liminf}
		\GF^\gamma_\rho(\Qc,\wh Q_\rho,\wh R)
		\leq
		\liminf_{(h,\varepsilon)\to(0,0)}
		\frac1h\GF_{h,\varepsilon}(Q_{h,\varepsilon}) .
	\end{equation}

	\noindent\textbf{\emph{(iii) $\Gamma$-limsup inequality.}}
	For every $(\Qc^\ast,\wh Q^\ast_\rho,\wh R^\ast)\in \A^\gamma_\rho$, there exists a recovery sequence
	$Q^*_{h,\varepsilon}\in H^1(\Omega_h;\Szero)$ such that
	\[
	Q^*_{h,\varepsilon}\rightharpoonup^{\mathrm{Qhom}}(\Qc^\ast,\wh Q^\ast_\rho,\wh R^\ast),\quad\text{and}\quad
	\lim_{(h,\varepsilon)\to(0,0)}
	\frac1h\GF_{h,\varepsilon}(Q^*_{h,\varepsilon})
	=
	\GF^\gamma_\rho(\Qc^\ast,\wh Q^\ast_\rho,\wh R^\ast).
	\]

	\noindent\textbf{\emph{(iv) Convergence of minima and minimisers.}}
	The rescaled minima converge:
	\begin{equation}
		\label{eq:min-convergence}
		\lim_{(h,\varepsilon)\to(0,0)}
		\frac{\Gm_{h,\varepsilon}}{h}
		=
		\min_{(\Qc^\dagger,\wh Q^\dagger_\rho,\wh R^\dagger)\in\A^\gamma_\rho}
		\GF^\gamma_\rho(\Qc^\dagger,\wh Q^\dagger_\rho,\wh R^\dagger)=\Gm^\gamma_\rho.
	\end{equation}
	Moreover, every $\rightharpoonup^{\mathrm{Qhom}}$-cluster point of a sequence
	of minimisers of $\GF_{h,\varepsilon}$ is a minimiser of
	$\GF^\gamma_\rho$.
\end{theorem}

\begin{remark}[Sequential $\Gamma$-convergence]
	\label{rem:seq-gamma}
	The notion of $\Gamma$-convergence used in Theorem~\ref{thm:gamma} is the 
	\emph{sequential} one, which is standard when the underlying convergence 
	(e.g., Definition~\ref{def:hom-conv}) is defined by sequential limits and the 
	relevant topologies (weak $H^1$, two-scale) are not metrizable. 
	A family of functionals is said to sequentially $\Gamma$-converge if it satisfies 
	the compactness property, the liminf inequality, and the existence of a recovery 
	sequence (see~\cite[Chapters~7--8]{DalMaso}; for a context similar to ours, 
	see~\cite{NV13}). 
\end{remark}
\begin{theorem}
	\label{cor:strong-conv-minimisers}
	Let $\{Q_{h,\varepsilon}\}_{h,\e}\subset H^1(\Omega_h;\Szero)$ be a sequence such that
	\[
	Q_{h,\varepsilon}\rightharpoonup^{\mathrm{Qhom}}(\Qc,\wh Q_\rho,\wh R),\quad \text{and}\quad
	\lim_{(h,\varepsilon)\to(0,0)}
	\frac1h\GF_{h,\varepsilon}(Q_{h,\varepsilon})
	=
	\GF^\gamma_\rho(\Qc,\wh Q_\rho,\wh R),
	\]
	with $(\Qc,\hQ_\rho,\hR)\in\A^\gamma_\rho$. Then the
	individual energy contributions converge to their corresponding limiting
	contributions:
	\begin{align}
		\lim_{(h,\varepsilon)\to(0,0)}
		\frac1h E^{h,\varepsilon}_{\el}(Q_{h,\varepsilon})
		&=
		E^\rho_{\el}(\Qc,\hQ_\rho,\hR),
		\label{eq:elastic-conv-min}\\
		\lim_{(h,\varepsilon)\to(0,0)}
		\frac1h E^{h,\e}_{\bulk}(Q_{h,\varepsilon})
		&=
		\Ebulk(\Qc),
		\label{eq:bulk-conv-min}\\
		\lim_{(h,\varepsilon)\to(0,0)}
		\frac1h E^{h,\varepsilon}_{\surf}(Q_{h,\varepsilon})
		&=
		E^\gamma_\surf(\Qc).
		\label{eq:surf-conv-min}
	\end{align}
	Moreover, if $\Gc_\rho$ denotes the limiting two-scale gradient associated
	with $(\Qc,\hQ_\rho,\hR)$ as in Lemma \ref{prop:admissible-correctors}, then
	\begin{equation}
		\label{eq:strong-grad-min}
		\Pi_{h,\varepsilon}(\nabla Q_{h,\varepsilon})
		\to
		\Gc_\rho
		\qquad\text{strongly in }L^2(\omega\times Y;\Szero\otimes\mathbb R^3).
	\end{equation}
	Finally,
	\begin{equation}
		\label{eq:strong-L4-min}
		\Pi_{h,\varepsilon}Q_{h,\varepsilon}
		\to
		\Qc
		\qquad\text{strongly in }L^4(\omega\times Y;\Szero).
	\end{equation}
\end{theorem}

\paragraph{Homogenised energy}

\begin{theorem}[Homogenised energy]
	\label{thm:hom-energy}
	Under the hypotheses of Theorem~\ref{thm:gamma} with
	$\gamma\geq 0$ and $\rho\in[0,\infty]$, the limiting homogenised energy is
	\begin{equation}
		\label{eq:hom-energy}
		\GF^{\mathrm{hom},\gamma}_\rho(\Qc)
		:=
		\frac12
		\int_\omega
		\bA^{\mathrm{hom}}_\rho\,\nabla'\Qc:\nabla'\Qc\,dx'
		+
		\int_\omega f_b(\Qc)\,dx'
		+
		E^\gamma_\surf(\Qc),
	\end{equation}
	where $\bA^{\mathrm{hom}}_\rho$ is given by Definition~\ref{def:hom-tensor} and
	$E^\gamma_\surf(\Qc)$ is given by Theorem~\ref{thm:gamma}. Moreover, if
	$(\Qc,\hQ_\rho,\hR)\in\A^\gamma_\rho$ is a two-scale minimiser obtained in
	Theorem~\ref{thm:gamma}, then
	\begin{equation}
		\label{eq:limit-min}
		\lim_{(h,\varepsilon)\to(0,0)}
		\frac{\Gm_{h,\varepsilon}}{h}
		=
		\GF^\gamma_\rho(\Qc,\hQ_\rho,\hR)
		=
		\GF^{\mathrm{hom},\gamma}_\rho(\Qc)=
			\min_{\Qc^\dagger\in H^1(\omega;\mathscr{E}_\gamma)}
			\GF^{\mathrm{hom},\gamma}_\rho(\Qc^\dagger).
	\end{equation}
	In particular, $\Qc$ is a minimiser of the homogenised energy
	$\GF^{\mathrm{hom},\gamma}_\rho$, and the minimum in \cref{eq:limit-min} is attained.
\end{theorem}

\begin{remark}
	\label{rem:R1-hierarchy}
	In regime R2, the in-plane corrector has the coupled form
	$\wh Q_\rho=\wh Q+\rho\wh R$. Thus, for every $\rho\in(0,\infty)$, in-plane
	microscopic relaxation and transverse relaxation interact, and this interaction
	is responsible for the explicit dependence of $\bA^{\mathrm{hom}}_\rho$ on the
	ratio $\rho=h/\varepsilon$.
	
	Regime R1 corresponds to the hierarchy $\varepsilon\ll h$. In this case,
	in-plane homogenisation is performed at each fixed thickness level $y_3$, while
	the thickness corrector describes the remaining transverse relaxation. Because
	$\bA$ may couple in-plane and transverse gradients, $\bA^{\mathrm{hom}}_\infty$
	need not be the thickness average of slice-wise homogenised tensors.
\end{remark}

\begin{theorem}
	\label{thm:hom-tensor-continuity}
	For every $M \in \Szero \otimes \mathbb{R}^2$, the map
	\[
	[0,\infty] \ni \rho \;\longmapsto\; \bA^{\mathrm{hom}}_\rho M:M
	\]
	is continuous.
\end{theorem}
	The continuity of $\rho \mapsto \bA^{\mathrm{hom}}_\rho$ at the
	endpoints $\rho = 0$ and $\rho = \infty$ is not immediate, since the
	corrector spaces change qualitatively across regimes.

\paragraph{Oseen--Frank and Ericksen energies from LdG}

The following results concern the anchoring-dominated regime
$0\leq\gamma<1$, where the anchoring imposes a hard constraint on the
macroscopic field. We show that under specific anchoring constraints the
homogenised energy takes the form of an Oseen--Frank or Ericksen energy.

\begin{theorem}[Homogenised anisotropic Ericksen energy]
	\label{thm:Er}
	Let the hypotheses of Theorem~\ref{thm:gamma} hold with
	$0\leq\gamma<1$. Let $\mathscr A$ be given by \cref{eq:csmwb+}, and
	suppose that the anchoring sets satisfy $\mathcal A(\nu)=\mathscr A$
	independently of $y$ and $\nu$ on both faces $y_3=\pm\frac12$.
	Then Assumption~\ref{asm:compatible}(iii) holds.
	
	Moreover, every $\Qc\in H^1(\omega;\mathscr A)$ admits the unique
	representation
	\[
	\Qc=s\left(P-\frac13I\right)
	\quad\text{a.e. in }\omega,\quad\text{with}\quad
	s\in H^1(\omega;[s_0,s_1]),
	\qquad
	P\in H^1(\omega;\mathbb P).
	\]
	Equivalently, $P$ represents an unoriented director field
	$[\bn]\in H^1(\omega;\mathbb S^2/\{\pm1\})$.
	
	The homogenised energy $\GF^{\hom,\gamma}_\rho$ restricted to
	$H^1(\omega;\mathscr A)$ is equivalently written as the homogenised
	anisotropic Ericksen energy
	\begin{equation}\label{eq:Er}
		\begin{aligned}
			\GF^{\hom}_{\mathrm{Er},\rho}(s,P)
			:=
			{1\over 2}\int_\omega
			\bA^{\hom}_\rho \Mc(s,P):\Mc(s,P)\,dx'
			+
			\int_\omega f^{\mathrm{Er}}_b(s)\,dx'=\GF^{\hom,\gamma}_\rho(\Qc),
		\end{aligned}
	\end{equation}
	where
	\[
	\Mc(s,P)_\alpha
	:=
	\partial_\alpha s\left(P-\frac13I\right)
	+
	s\,\partial_\alpha P,
	\qquad \alpha=1,2,\quad\text{and}\quad
	f^{\mathrm{Er}}_b(s)
	:=
	f_b\left(s\left(P-\frac13I\right)\right).
	\]
	The definition of $f^{\mathrm{Er}}_b$ is independent of $P\in\mathbb P$.
	Consequently, minimisers of $\GF^{\hom}_{\mathrm{Er},\rho}$ over $H^1(\omega;[s_0,s_1])\times H^1(\omega;\mathbb P)$
	are identified, through the map
	\[
	(s,P)\mapsto s\left(P-\frac13I\right),
	\]
	with minimisers of $\GF^{\hom,\gamma}_\rho$ over
	$H^1(\omega;\mathscr A)$.
\end{theorem}

\begin{theorem}[Homogenised anisotropic Oseen--Frank energy]
	\label{thm:OF}
	Let the hypotheses of Theorem~\ref{thm:gamma} hold with
	$0\leq\gamma<1$, and let $s_0\in\mathbb R\setminus\{0\}$ be fixed.
	Suppose that the anchoring sets are given by
	\[
	\mathcal A(\nu)
	=
	\mathscr A
	=
	\left\{
	s_0\left(P-\frac13 I\right):
	P\in\mathbb P
	\right\}.
	\]
	independently of $y$ and $\nu$ on both faces $y_3=\pm\frac12$.
	Then $\mathscr A$ is a compact smooth submanifold of $\Szero$,
	diffeomorphic to $\mathbb{RP}^2$, and
	Assumption~\ref{asm:compatible}(ii) holds.
	
	Moreover, every $\Qc\in H^1(\omega;\mathscr A)$ admits the unique
	representation
	\[
	\Qc=s_0\left(P-\frac13 I\right)
	\quad\text{a.e. in }\omega,\quad\text{with}\quad
	P\in H^1(\omega;\mathbb P).
	\]
	Equivalently, $P$ represents an unoriented director field
	$[\bn]\in H^1(\omega;\mathbb S^2/\{\pm1\})$.
	
	The homogenised energy $\GF^{\hom,\gamma}_\rho$ restricted to
	$H^1(\omega;\mathscr A)$ is equivalently written as the homogenised
	anisotropic Oseen--Frank energy
	\begin{equation}\label{eq:OF}
		\GF^{\hom}_{\mathrm{OF},\rho}(P)
		:=
		{s_0^2\over 2}\int_\omega
		\bA^{\hom}_\rho\nabla'P:\nabla'P\,dx'
		+
		c_b|\omega|=\GF^{\hom,\gamma}_\rho(\Qc),
	\end{equation}
	where
	\[
	c_b
	:=
	f_b\left(s_0\left(P-\frac13 I\right)\right).
	\]
	The constant $c_b$ is independent of $P\in\mathbb P$. 
	Consequently, minimisers of $\GF^{\hom}_{\mathrm{OF},\rho}$ over
	$H^1(\omega;\mathbb P)$
	are identified, through the map
	\[
	P\mapsto s_0\left(P-\frac13I\right),
	\]
	with minimisers of $\GF^{\hom,\gamma}_\rho$ over
	$H^1(\omega;\mathscr A)$.
\end{theorem}

\begin{remark}[Orientable director representation]
	If the line field $P\in H^1(\omega;\mathbb P)$ is orientable, i.e.~if
	$P=\bn\otimes\bn$ a.e.~in $\omega$ for some
	$\bn\in H^1(\omega;\mathbb S^2)$, then
	\[
	\Qc=s_0\left(\bn\otimes\bn-\frac13 I\right),
	\qquad
	\nabla'P=\nabla'\bn\otimes\bn+\bn\otimes\nabla'\bn .
	\]
	Accordingly, the homogenised Oseen--Frank energy may be written as
	\[
	\GF^{\hom}_{\mathrm{OF},\rho}(\bn)
	=
	{s_0^2\over 2}\int_\omega
	\bA^{\hom}_\rho
	\bigl(\nabla'\bn\otimes\bn+\bn\otimes\nabla'\bn\bigr)
	:
	\bigl(\nabla'\bn\otimes\bn+\bn\otimes\nabla'\bn\bigr)
	\,dx'
	+
	c_b|\omega|.
	\]
	This expression is invariant under $\bn\sim-\bn$. In general, a global
	orientation need not exist, so the intrinsic formulation is in terms of
	$P\in H^1(\omega;\mathbb P)$.
\end{remark}
	
	\section{Preliminaries: Unfolding operators and thickness decomposition}
	\label{sec:pre}
	
	\subsection{Rescaling and unfolding operators}
	\label{sec:unfolding}
	
	The passage to the limit is carried out on fixed reference domains by means of two operators from periodic homogenisation: the unfolding operator $\Te$ on the mid-surface $\o$, and its thin-film counterpart, the rescaling unfolding operator $\Peh$ on the three-dimensional thin domain $\O_h$. The key idea is that these operators rewrite oscillatory quantities depending on the fast variables $x'/\e$ and $x_3/h$ as functions of an independent microscopic variable
	\[
	y=(y',y_3)\in Y:=Y'\times I,
	\qquad
	Y':=(0,1)^2,\quad I:=\left(-\frac12,\frac12\right),
	\]
	while keeping the slow in-plane variable $x'\in\o$.
	As a consequence, integrals over the varying thin domains $\O_h$ are transformed into integrals over the fixed set $\o\times Y$, which is the natural framework for simultaneous homogenisation and dimension reduction.
	
	Recall the decomposition of $\o$ into complete $\e$-cells and the boundary remainder:
	\[
	\Xi_\e=\{\xi\in\mathbb Z^2:\ \e\xi+\e Y'\subset \o\},
	\qquad
	\wh{\o}_\e:=\operatorname{int}\Big(\bigcup_{\xi\in\Xi_\e}(\e\xi+\e Y')\Big),
	\qquad
	\Lambda_\e:=\o\setminus\wh{\o}_\e,
	\]
	and set $\wh \Omega_{h,\e}:=\wh{\o}_\e\times hI$.
	Thus $\wh{\o}_\e$ is the union of the complete periodicity cells contained in $\o$, while $\Lambda_\e$ is the thin boundary strip made of the incomplete cells.
	
	\begin{definition}[Unfolding and rescaling--unfolding operators]
		For every measurable function $\psi$ on $\O_h$, the rescaling--unfolding operator $\Peh$ is defined by
		\begin{equation}\label{eq:Peh-def}
			\Peh(\psi)(x',y)\doteq
			\begin{cases}
				\psi\!\left(\e\left[\dfrac{x'}{\e}\right]+\e y',\,hy_3\right)
				& \text{for a.e. }(x',y)\in \wh{\o}_\e\times Y,\\[1ex]
				0
				& \text{for a.e. }(x',y)\in \Lambda_\e\times Y.
			\end{cases}
		\end{equation}
		For every measurable function $\phi$ on $\o$, the unfolding operator $\Te$ is defined by
		\begin{equation}\label{eq:Te-def}
			\Te(\phi)(x',y')\doteq
			\begin{cases}
				\phi\!\left(\e\left[\dfrac{x'}{\e}\right]+\e y'\right)
				& \text{for a.e. }(x',y')\in \wh{\o}_\e\times Y',\\[1ex]
				0
				& \text{for a.e. }(x',y')\in \Lambda_\e\times Y'.
			\end{cases}
		\end{equation}
	\end{definition}
%	The operator $\Te$ is the standard periodic unfolding on $\o$, whereas $\Peh$ is adapted to the thin geometry: it unfolds the in-plane oscillations at scale $\e$ and simultaneously rescales the thickness variable $x_3\in hI$ to the fixed interval $I$.
	
	\begin{proposition}[Basic properties of $\Te$ and $\Peh$]\label{prop:unfolding-basic}
		The following statements are standard; see, for instance,
		\cite[Proposition 1.12, Proposition 1.35, Theorem 1.36, Corollary 1.37, Proposition 1.39]{CDG}.
		
		\begin{enumerate}
			\item[(i)] \textbf{Linearity and continuity for $p>1$. }
			The operator $\Te$ is linear and continuous from $L^p(\o)$ into $L^p(\o\times Y')$, and $\Peh$ is linear and continuous from $L^p(\O_h)$ into $L^p(\o\times Y)$.
			More precisely, for every $\phi\in L^p(\o)$ and every $\psi\in L^p(\O_h)$,
			\begin{equation}\label{eq:L2-bound}
				\|\Te(\phi)\|_{L^p(\o\times Y')}
				\le \|\phi\|_{L^p(\o)},
				\|\Peh(\psi)\|_{L^p(\o\times Y)}
				\le \frac{1}{ h^{1/p}}\,\|\psi\|_{L^p(\O_h)}.
			\end{equation}
			
			\item[(ii)] \textbf{Exact integral identities on the complete cells.}
			For every $\phi\in L^1(\o)$,
			\begin{equation}\label{eq:Te-integral-exact}
				\int_{\o\times Y'} \Te(\phi)(x',y')\,dx'dy'
				=
				\int_{\wh{\o}_\e}\phi(x')\,dx'.
			\end{equation}
			For every $\psi\in L^1(\O_h)$,
			\begin{equation}\label{eq:Peh-integral-exact}
				\int_{\o\times Y}\Peh(\psi)(x',y)\,dx'dy
				=
				\frac{1}{h}\int_{\wh\Omega_{h,\e}}\psi(x)\,dx.
			\end{equation}
			Consequently,
			\begin{equation}\label{eq:Te-integral-remainder}
				\left|
				\int_{\o}\phi(x')\,dx'
				-
				\int_{\o\times Y'} \Te(\phi)(x',y')\,dx'dy'
				\right|
				\le
				\int_{\Lambda_\e}|\phi(x')|\,dx',
			\end{equation}
			and
			\begin{equation}\label{eq:Peh-integral-remainder}
				\left|
				\frac{1}{h}\int_{\O_h}\psi(x)\,dx
				-
				\int_{\o\times Y}\Peh(\psi)(x',y)\,dx'dy
				\right|
				\le
				\frac{1}{h}\int_{\Lambda_\e\times hI} |\psi(x)|\,dx.
			\end{equation}
			In particular, since $|\Lambda_\e|\to 0$ as $\e\to0$, the unfolded integrals are asymptotically equivalent to the original ones.
			
			\item[(iii)] \textbf{Exact $L^p$-norm identities on the complete cells.}
			For every $1\le p<\infty$ and every $\phi\in L^p(\o)$,
			\begin{equation}\label{eq:Te-Lp-exact}
				\|\Te(\phi)\|_{L^p(\o\times Y')}^p
				=
				\|\phi\|_{L^p(\wh{\o}_\e)}^p
				\le
				\|\phi\|_{L^p(\o)}^p.
			\end{equation}
			For every $\psi\in L^p(\O_h)$,
			\begin{equation}\label{eq:Peh-Lp-exact}
				\|\Peh(\psi)\|_{L^p(\o\times Y)}^p
				=
				\frac{1}{h}\|\psi\|_{L^p(\wh\Omega_{h,\e})}^p
				\le
				\frac{1}{h}\|\psi\|_{L^p(\O_h)}^p.
			\end{equation}
			
			\item[(iv)] \textbf{Compatibility with products.}
			If the products are integrable, then
			\begin{equation}\label{eq:Te-product}
				\Te(\phi_1\phi_2)=\Te(\phi_1)\Te(\phi_2)
				\qquad\text{a.e. in }\o\times Y',\quad
				\Peh(\psi_1\psi_2)=\Peh(\psi_1)\Peh(\psi_2)
				\qquad\text{a.e. in }\o\times Y.
			\end{equation}
			
			\item[(v)] \textbf{Derivative identities.}
			Let $\psi\in H^1(\O_h)$. Then
			\begin{equation}\label{eq:Peh-grad-yprime}
				\nabla_{y'}\Peh(\psi)
				=
				\e\,\Peh(\nabla_{x'}\psi),\quad 
				\partial_{y_3}\Peh(\psi)
				=
				h\,\Peh(\partial_{x_3}\psi)
				\qquad\text{a.e. in }\o\times Y.
			\end{equation}
			Equivalently,
			\begin{equation}\label{eq:Peh-grad-inverse}
				\Peh(\nabla_{x'}\psi)=\frac1\e\nabla_{y'}\Peh(\psi),
				\qquad
				\Peh(\partial_{x_3}\psi)=\frac1h\partial_{y_3}\Peh(\psi).
			\end{equation}
			Similarly, for every $\phi\in H^1(\o)$,
			\begin{equation}\label{eq:Te-grad}
				\nabla_{y'}\Te(\phi)
				=
				\e\,\Te(\nabla_{x'}\phi)
				\qquad\text{a.e. in }\o\times Y'.
			\end{equation}
		\end{enumerate}
	\end{proposition}

\begin{remark}[Interpretation in the present thin-film LdG setting]
	\label{rem:unfolding-interpretation}
%	The operator $\Peh$ transforms functions on the thin domain
%	$\O_h=\o\times hI$ into functions on the fixed domain $\o\times Y$, with
%	slow variable $x'\in\o$ and microscopic variable
%	\[
%	y=\left(y',y_3\right)
%	=
%	\left(\left\{\frac{x'}{\e}\right\},\frac{x_3}{h}\right)\in Y .
%	\]
%	Thus the in-plane oscillations and the thickness rescaling are treated
%	simultaneously. 
	For a $\phi$ is defined only on $\o$, then
	\[
	\Peh(\phi)(x',y)=\Te(\phi)(x',y')
	\qquad\text{for a.e. }(x',y)\in\o\times Y,
	\]
	which is the form used for the mean fields obtained from the thickness
	decomposition. Moreover, since $f_b$ is a polynomial in $Q$ with $f_b(0)=0$, the product
	property gives
	\begin{equation}\label{eq:Peh-bulk-potential}
	\Peh(f_b(Q))=f_b(\Peh(Q))
	\qquad\text{a.e. in }\o\times Y.
	\end{equation}
	Similarly, if $g\in L^\infty(Y)$ is $Y'$-periodic in $y'$ and
	$g_{h,\e}(x):=g(x'/\e,x_3/h)$, then
	\[
	\Peh(g_{h,\e})(x',y)=g(y)
	\qquad\text{for a.e. }(x',y)\in\wh{\o}_\e\times Y.
	\]
	In particular, the oscillatory elastic tensor
	$\bA(x'/\e,x_3/h)$ becomes $\bA(y)$ after unfolding. Finally, for
	$\psi\in H^1(\O_h)$ and traces $\psi^\pm(x'):=\psi(x',\pm h/2)$,
	\[
	\Peh(\psi)(x',y',\pm\tfrac12)=\Te(\psi^\pm)(x',y')
	\qquad\text{for a.e. }(x',y')\in\wh{\o}_\e\times Y',
	\]
	which links the bulk unfolding with the surface anchoring terms on
	$\Gamma_h^\pm$.
\end{remark}
	
	To formulate the compactness statements in a concise way, we use the cell-average operators
	\begin{equation}\label{eq:cell-average}
		\Mc_{Y'}(g)(x'):=\int_{Y'} g(x',y')\,dy',
		\qquad
		\Mc_Y(G)(x'):=\int_Y G(x',y)\,dy.
	\end{equation}
	Since $|Y'|=|Y|=1$, these are both integrals and averages.
	We also write
	\[
	H^1_{per,0}(Y'):=\left\{\zeta\in H^1_{per}(Y'):\ \int_{Y'}\zeta(y')\,dy'=0\right\}.
	\]
	
	The next lemma is the compactness input that will be used later, in particular for quantities defined on the mid-surface $\o$.
	
	\begin{lemma}\label{A11}
		The following statements are standard consequences of the unfolding compactness theory; see
		\cite[Theorem 1.36, Corollary 1.37, Proposition 1.39]{CDG}.
		
		\begin{enumerate}
			\item Let $\{w_{h,\e}\}_{h,\e}$ be a sequence weakly convergent in $L^2(\o)$, with weak limit $w\in L^2(\o)$.
			Then there exists $\wh{w}\in L^2(\o\times Y')$ with
			\[
			\Mc_{Y'}(\wh{w})=0
			\qquad\text{a.e. in }\o,
			\]
			such that, up to a subsequence,
			\begin{equation}\label{eq:Te-L2-compact}
				\Te(w_{h,\e})\rightharpoonup w+\wh{w}
				\qquad\text{weakly in }L^2(\o\times Y').
			\end{equation}
			
			\item Let $\{v_{h,\e}\}_{h,\e}$ be a sequence weakly convergent in $H^1(\o)$, with weak limit $v\in H^1(\o)$.
			Then there exists
			\[
			\wh{v}\in L^2(\o;H^1_{per,0}(Y'))
			\]
			such that, up to a subsequence,
			\begin{equation}\label{eq:Te-H1-strong}
				\Te(v_{h,\e})\to v
				\qquad\text{strongly in }L^2(\o\times Y'),
			\end{equation}
			and
			\begin{equation}\label{eq:Te-grad-compact}
				\Te(\nabla v_{h,\e})
				\rightharpoonup
				\nabla v+\nabla_{y'}\wh{v}
				\qquad\text{weakly in }L^2(\o\times Y').
			\end{equation}
		\end{enumerate}
	\end{lemma}	
	
	\subsection{The decomposition of $Q$-tensor on thin domains}\label{sec:decomposition}
	
	\begin{definition}[Thickness decomposition]
		\label{def:decomp}
		For $Q \in H^1(\Omega_h;\Szero)$ define the mean field
		$\mathcal{Q} \in H^1(\omega;\Szero)$ and the remainder
		$\mathcal{R} \in H^1(\Omega_h;\Szero)$ by
		\[
		\mathcal{Q}(x') := \frac{1}{h}\int_{hI} Q(x',x_3)\,\mathrm{d}x_3,
		\qquad
		\mathcal{R}(x',x_3) := Q(x',x_3) - \mathcal{Q}(x').
		\]
	\end{definition}

	\begin{proposition}[Thickness decomposition estimates]
		\label{prop:thickness-decomposition}
		Let $Q\in H^1(\Omega_h;\Szero)$, and let
		$\mathcal Q$ and $\mathcal R$ be defined as in
		Definition~\ref{def:decomp}. Then, for a.e.\ $x'\in\omega$,
		\begin{equation}
			\label{Ortho01}
			\int_{hI}\mathcal R(x',x_3)\,\mathrm dx_3=0.
		\end{equation}
		Moreover,
		\begin{align}
			h\|\mathcal Q\|_{L^2(\omega)}^2
			+\|\mathcal R\|_{L^2(\Omega_h)}^2
			&=
			\|Q\|_{L^2(\Omega_h)}^2,
			\label{Es01a}\\
			h\|\nabla'\mathcal Q\|_{L^2(\omega)}^2
			+\|\nabla'\mathcal R\|_{L^2(\Omega_h)}^2
			&=
			\|\nabla'Q\|_{L^2(\Omega_h)}^2.
			\label{Es01b}
		\end{align}
		There exists $C>0$, independent of $h$ and $Q$, such that
		\begin{align}
			\|\mathcal R\|_{L^2(\Omega_h)}
			&\leq
			Ch\|\partial_3\mathcal R\|_{L^2(\Omega_h)}
			=
			Ch\|\partial_3Q\|_{L^2(\Omega_h)},
			\label{Es04}\\
			\|Q-\mathcal Q\|_{L^2(\Omega_h)}
			+\|\mathcal R\|_{L^2(\Omega_h)}
			&\leq
			Ch\|\partial_3Q\|_{L^2(\Omega_h)}.
			\label{Es03}
		\end{align}
	\end{proposition}
	
	\begin{proof}
		The zero-average property~\eqref{Ortho01} follows from the definition of
		$\mathcal Q$. Consequently, $\int_{\Omega_h}\mathcal Q:\mathcal R\,\mathrm dx=0$.
		Moreover,
		\[
		\int_{hI}\nabla'\mathcal R(x',x_3)\,\mathrm dx_3
		=
		\nabla'\int_{hI}\mathcal R(x',x_3)\,\mathrm dx_3
		=0
		\]
		for a.e.\ $x'\in\omega$. Applying these orthogonality relations to
		$Q=\mathcal Q+\mathcal R$ and
		$\nabla'Q=\nabla'\mathcal Q+\nabla'\mathcal R$ gives
		\eqref{Es01a} and~\eqref{Es01b}.
		
		For a.e.\ $x'\in\omega$, the Poincaré--Wirtinger inequality on $hI$ and
		\eqref{Ortho01} yield
		\[
		\|\mathcal R(x',\cdot)\|_{L^2(hI)}
		\leq
		Ch\|\partial_3\mathcal R(x',\cdot)\|_{L^2(hI)}.
		\]
		Integration over $\omega$, together with
		$\partial_3\mathcal R=\partial_3Q$, proves~\eqref{Es04}.
		Since $Q-\mathcal Q=\mathcal R$, estimate~\eqref{Es03} follows.
	\end{proof}
	
	\begin{lemma}[Trace estimates]
		\label{lem:trace-estimates}
		Let $Q\in H^1(\Omega_h;\Szero)$, and define
		\[
		Q^\pm(x'):=Q(x',\pm h/2),
		\qquad
		\mathcal R^\pm(x'):=\mathcal R(x',\pm h/2),
		\]
		where $\mathcal R$ is given by Definition~\ref{def:decomp}. Then
		$Q^\pm,\mathcal R^\pm\in L^2(\omega;\Szero)$, and there exists
		$C>0$, independent of $h$ and $Q$, such that
		\begin{align}
			\|\mathcal R^\pm\|_{L^2(\omega)}^2
			&\leq
			Ch\|\partial_3Q\|_{L^2(\Omega_h)}^2,
			\label{eq:Rtrace}\\
			\|Q^\pm\|_{L^2(\omega)}^2
			&\leq
			C\left(
			\frac1h\|Q\|_{L^2(\Omega_h)}^2
			+h\|\partial_3Q\|_{L^2(\Omega_h)}^2
			\right).
			\label{eq:Qtrace}
		\end{align}
	\end{lemma}
	
	\begin{proof}
		Rescaling the one-dimensional trace inequality from $I$ to $hI$ gives
		\begin{equation}
			\label{eq:gen-trace}
			\|v^\pm\|_{L^2(\omega)}^2
			\leq
			C\left(
			\frac1h\|v\|_{L^2(\Omega_h)}^2
			+h\|\partial_3v\|_{L^2(\Omega_h)}^2
			\right)
		\end{equation}
		for every $v\in H^1(\Omega_h;\Szero)$. Taking $v=Q$ proves
		\eqref{eq:Qtrace}. Taking $v=\mathcal R$ and using
		Proposition~\ref{prop:thickness-decomposition}, namely
		\[
		\|\mathcal R\|_{L^2(\Omega_h)}
		\leq
		Ch\|\partial_3Q\|_{L^2(\Omega_h)},
		\qquad
		\partial_3\mathcal R=\partial_3Q,
		\]
		gives~\eqref{eq:Rtrace}.
	\end{proof}
	%%%%%%%%%%%%%%%%%%%%%%%%%%%%%%%%%%%%%%%%%%%%%%%%%%%%%%%%%%%%%%%%%%%%%%%%
	\section{Asymptotic analysis of the $Q$-tensors}
	\label{sec:asymptotic}
	
\subsection{Uniform energy  bounds}
\label{sec:bounds}
In this section, we derive a few properties of each term of the total LdG energy functional needed to establish that $\mhe$ is of order $h$, bounded below and above, uniformly in $h$ and $\varepsilon$.

The \textbf{elastic energy} satisfies the two-sided bound using  Assumption~\ref{asm:elastic}\ref{A3}, 
\begin{equation}
	\label{eq:elastic-coercive}
	\frac{\lambda}{2}\int_{\Ohmh}\!\abs{\nabla Q}^{2}\,dx
	\;\leq\; E^{h,\e}_{\mathrm{el}}(Q)
	\;\leq\; \frac{\Lambda}{2}\int_{\Ohmh}\!\abs{\nabla Q}^{2}\,dx
	\qquad\forall\, Q\in H^{1}(\Ohmh;\Szero).
\end{equation}
In particular, $E^{h,\e}_{\mathrm{el}}(Q)\geq 0$ and the lower bound is
coercive in $\nabla Q$. The following two lemmas record the growth
properties of the \textbf{bulk} and \textbf{surface} terms needed throughout.

\begin{lemma}[Properties of $f_b$]
	\label{lem:fB}
	Under Assumption~\ref{asm:bulk}, there exist constants
	$c_1,c_2,c_3>0$ depending only on $a,b,c$ such that for all
	$Q\in\Szero$:
	\begin{equation}
		\label{B3}
		f_b(Q) \;\geq\; -c_2,\qquad
		f_b(Q) \;\geq\; c_1\abs{Q}^{4} - c_2,\qquad
		\abs{f_b(Q)} \;\leq\; c_3\bigl(1+\abs{Q}^{4}\bigr).
	\end{equation}
\end{lemma}

\begin{proof}
	For all $Q\in\Szero$, the definition of the Frobenius norm gives
	\begin{equation}
		\label{eq:trQ3}
		\tr(Q^{2}) = \abs{Q}^{2},\qquad
		\abs{\tr(Q^{3})} \leq \abs{Q}^{3}.
	\end{equation}
	Using \cref{eq:trQ3},
	\[
	f_b(Q) = a\abs{Q}^{2} - b\,\tr(Q^3) + c\abs{Q}^{4}.
	\]
	By \cref{eq:trQ3} and Young's inequality with exponents $(4,4/3)$:
	\[
	\abs{b}\abs{Q}^{3}
	\;\leq\; \frac{c}{2}\abs{Q}^{4} + K(b,c),
	\qquad
	K(b,c) := \frac{3}{4}\Bigl(\frac{2\abs{b}}{c}\Bigr)^{4/3} > 0,
	\]
	which gives the lower bound
	\begin{equation}
		\label{eq:fB-lower}
		f_b(Q)
		\;\geq\; a\abs{Q}^{2} + \frac{c}{2}\abs{Q}^{4} - K(b,c).
	\end{equation}
	The polynomial $s\mapsto as^2+\frac{c}{2}s^4$ is bounded below,
	giving $f_b(Q)\geq -c_2$ for $c_2>0$ depending on $a,b,c$.
	This proves \cref{B3}$_1$.
	
	Set $R:=\sqrt{2\abs{a}/c}$. For $\abs{Q}\geq R$,
	$a\abs{Q}^2\geq -\frac{c}{4}\abs{Q}^4$, so
	$f_b(Q)\geq\frac{c}{4}\abs{Q}^{4} - K(b,c)$,
	giving \cref{B3}$_2$ with $c_1:=c/4$.
	
	For the upper bound, Young's inequality with exponents $(4/3,4)$ gives
	$\abs{Q}^3\leq\tfrac{3}{4}\abs{Q}^4+\tfrac{1}{4}\leq\tfrac{3}{4}(1+\abs{Q}^4)$,
	and $(\abs{Q}^2-\tfrac{1}{2})^2\geq 0$ gives
	$\abs{Q}^2\leq\abs{Q}^4+\tfrac{1}{4}\leq 1+\abs{Q}^4$. Therefore
	\[
	f_b(Q)
	\;\leq\;
	\abs{a}\abs{Q}^{2} + \abs{b}\abs{Q}^{3} + c\abs{Q}^{4}
	\;\leq\;
	\Bigl(\abs{a} + \frac{3\abs{b}}{4} + c\Bigr)\bigl(1+\abs{Q}^{4}\bigr),
	\]
	proving \cref{B3}$_3$ with $c_3:=\abs{a}+\frac{3\abs{b}}{4}+c$.
\end{proof}

\begin{lemma}[Properties of $E_{\mathrm{surf}}^{h,\varepsilon}$]
	\label{lem:surf-properties}
	Under Assumption~\ref{asm:anchoring} and
	Definition~\ref{def:anchoring-data}, the following hold for each fixed
	$h,\varepsilon>0$.
	\begin{enumerate}[label=\textup{(\alph*)}]
		\item\label{surf-normal-integrand}
		\textup{(Normal integrand.)}
		$f_s^{h,\varepsilon}$ is a normal integrand on
		$\Szero\times\Gamma_h^{w}\times\mathbb{S}^2$: for every fixed
		$Q\in\Szero$, the map $x\mapsto f_s^{h,\varepsilon}(Q,x,\nu(x))$
		is $\mathcal{H}^2$-measurable on $\Gamma_h^{w}$; and for
		$\mathcal{H}^2$-a.e.\ $x\in\Gamma_h^{w}$, the map
		$Q\mapsto f_s^{h,\varepsilon}(Q,x,\nu(x))$ is continuous on $\Szero$.
		
		\item\label{surf-growth}
		\textup{(Quadratic growth.)}
		There exists $C_s>0$ depending only on $C_{\mathcal{A}}$ and
		$\|W\|_{L^\infty}$ such that
		\[
		0 \;\leq\; f_s^{h,\varepsilon}(Q,x,\nu)
		\;\leq\; C_s\,h^\gamma\,(1+|Q|^2)
		\]
		for all $Q\in\Szero$, $x\in\Gamma_h^{w}$, $\nu\in\mathbb{S}^2$.
		
		\item\label{surf-meas-finite}
		\textup{(Finiteness.)}
		For every $Q\in H^1(\Omega_h;\Szero)$,
		$E_{\mathrm{surf}}^{h,\varepsilon}(Q)<\infty$.
	\end{enumerate}
\end{lemma}

\begin{proof}
	On each face $\Gamma_h^\pm$ the outer unit normal is the constant
	$\pm e_3$, so it suffices to work with $\nu=\pm e_3$ fixed throughout.
		
	\textbf{Measurability in $x$.}
	Fix $Q\in\Szero$. By~\ref{SS1}, the map
	$y\mapsto\operatorname{dist}(Q,\mathcal{A}(y,\pm e_3))$ is Borel
	measurable on $\partial Y$ for every fixed $Q$. The map
	$x'\mapsto(x'/\varepsilon,\pm\tfrac{1}{2})$ is Borel measurable from
	$\omega$ to $\partial Y$, so
	$x'\mapsto\operatorname{dist}^2(Q,\mathcal{A}^{h,\varepsilon}(x,\pm e_3))$
	is Borel measurable. Since $W^{h,\varepsilon}(\cdot,\pm e_3)$ is Borel
	measurable by definition,
	$x\mapsto f_s^{h,\varepsilon}(Q,x,\pm e_3)$ is
	$\mathcal{H}^2$-measurable on $\Gamma_h^\pm$.
	
	\textbf{Continuity in $Q$.}
	Fix $x\in\Gamma_h^\pm$. By~\ref{SS1},
	$\mathcal{A}^{h,\varepsilon}(x,\pm e_3)
	=\mathcal{A}(x'/\varepsilon,\pm\tfrac{1}{2},\pm e_3)$
	is a non-empty closed subset of $\Szero$. The map
	$Q\mapsto\operatorname{dist}^2(Q,\mathcal{A}^{h,\varepsilon}(x,\pm e_3))$
	is continuous on the finite-dimensional space $\Szero$ (it is
	$1$-Lipschitz before squaring), and $W^{h,\varepsilon}(x,\pm e_3)\geq 0$
	is a finite constant for fixed $x$. Hence
	$Q\mapsto f_s^{h,\varepsilon}(Q,x,\pm e_3)$ is continuous for every
	$x\in\Gamma_h^\pm$, completing the normal-integrand verification. This completes 	Part~\ref{surf-normal-integrand}. 
	
	For Part~\ref{surf-growth}, we fix $Q\in\Szero$ and $x\in\Gamma_h^w$. By~\ref{SS2}, there exists
	$A_0\in\mathcal{A}^{h,\varepsilon}(x,\nu)$ with $|A_0|\leq C_{\mathcal{A}}$.
	Then
	\[
	\operatorname{dist}^2\!\bigl(Q,\mathcal{A}^{h,\varepsilon}(x,\nu)\bigr)
	\;\leq\; |Q-A_0|^2
	\;\leq\; 2|Q|^2+2C_{\mathcal{A}}^2
	\;\leq\; 2(1+C_{\mathcal{A}}^2)(1+|Q|^2).
	\]
	Multiplying by
	$W^{h,\varepsilon}(x,\nu)\leq h^\gamma\|W\|_{L^\infty}$
	gives the upper bound with
	$C_s:=2(1+C_{\mathcal{A}}^2)\|W\|_{L^\infty}$,
	independent of $h$ and $\varepsilon$. The lower bound is immediate
	from $W^{h,\varepsilon}\geq 0$ and $\operatorname{dist}^2\geq 0$.
	
	The map $x\mapsto f_s^{h,\varepsilon}(Q(x),x,\nu(x))$ is
	$\mathcal{H}^2$-measurable on $\Gamma_h^w$ by
	Part~\ref{surf-normal-integrand} and the Carathéodory measurability
	criterion, since the trace $x\mapsto Q(x)$ is a measurable
	$\Szero$-valued map on $\Gamma_h^w$. Finiteness then follows from
	the quadratic bound in Part~\ref{surf-growth} together with the $L^2$
	estimate on the trace of $Q\in H^1(\Omega_h;\Szero)$. This completes 	Part~\ref{surf-meas-finite} and the proof.
\end{proof}

\begin{proposition}[Uniform bounds on $\mhe$]
	\label{prop:inf-bounds}
	Under Assumptions~\ref{asm:elastic}, \ref{asm:bulk},
	and~\ref{asm:anchoring}, for all $\gamma\geq 0$ there exist constants
	$C_\ell,C_u>0$ independent of $h$ and $\varepsilon$ such that
	\begin{equation}
		\label{eq:inf-bounds}
		-C_\ell\,h \;\leq\; \mhe \;\leq\; C_u\,h^\gamma.
	\end{equation}
	If in addition Assumption~\ref{asm:compatible} holds \textup{(}the case
	$0\leq\gamma<1$\textup{)}, then the upper bound reads $C_u\,h$.
\end{proposition}
\begin{proof}
	\textbf{Lower bound.}
	Let $Q\in H^{1}(\Ohmh;\Szero)$ be arbitrary.
	By \cref{eq:elastic-coercive}, $E_{\mathrm{el}}(Q)\geq 0$.
	By Assumption~\ref{asm:anchoring}-\ref{SS1} with $W^{h,\varepsilon}\geq 0$,
	$E_{\mathrm{surf}}(Q)\geq 0$.
	By \cref{B3}$_1$ of Lemma~\ref{lem:fB}, $f_b(Q)\geq -c_2$ pointwise.
	Therefore
	\[
	\Fhe(Q)
	\;\geq\; \int_{\Ohmh}\!(-c_2)\,dx
	= -c_2\,\abs{\Ohmh}
	= -c_2\,h\,\abs{\omg}.
	\]
	Taking the infimum gives $\mhe\geq -c_2\,h\,\abs{\omg}=:-C_\ell h$.
	
	\textbf{Upper bound for $\gamma\geq 1$.}
	We test with the admissible field $Q\equiv 0\in H^{1}(\Ohmh;\Szero)$.
	The elastic term vanishes since $\nabla Q=0$, and $f_b(0)=0$
	from \cref{eq:fb}, so the bulk term also vanishes. By
	Lemma~\ref{lem:surf-properties}\ref{surf-growth} and
	$\mathcal{H}^2(\Gamma_h^w)=2|\omega|$,
	\[
	\Fhe(0) = E_{\mathrm{surf}}(0)\leq 2C_s|\omega|h^\gamma.
	\]
	Hence $\mhe\leq\Fhe(0)\leq C_u h^\gamma$.
	
	\textbf{Upper bound for $0\leq\gamma<1$.}
	Under Assumption~\ref{asm:compatible}, let $Q^*\in H^1(\omega;\mathscr{A})$
	be the field provided there, and test with the $x_3$-independent extension
	$Q(x',x_3):=Q^*(x')$. Since $\partial_3 Q^*=0$, \cref{eq:elastic-coercive}
	gives
	\[
	E^{h,\e}_{\mathrm{el}}(Q^*)
	\leq\frac{\Lambda}{2}\|\nabla Q^*\|^2_{L^2(\Ohmh)}
	=\frac{\Lambda h}{2}\|\nabla'Q^*\|^2_{L^2(\omega)}\leq Ch.
	\]
	By \cref{B3}$_3$, the bulk term satisfies
	\[
	E^{h,\e}_{\mathrm{bulk}}(Q^*)
	\leq c_3\int_{\Ohmh}\bigl(1+|Q^*(x')|^4\bigr)\,dx
	= c_3 h\int_{\omega}\bigl(1+|Q^*(x')|^4\bigr)\,dx'\leq Ch.
	\]
	Since $Q^*(x')\in\mathscr{A}$ for a.e.\ $x'\in\omega$, we have
	$\operatorname{dist}(Q^*(x'),\mathcal{A}(x'/\varepsilon,\pm\tfrac{1}{2},
	\pm e_3))=0$ pointwise, so $E^{h,\varepsilon}_{\mathrm{surf}}(Q^*)=0$.
	Hence $\mhe\leq\Fhe(Q^*)\leq C_u h$.
\end{proof}

\begin{theorem}[Existence of minimisers and uniform energy bounds]
	\label{prop:existence}
		Under Assumptions~\ref{asm:elastic}, \ref{asm:bulk}, and~\ref{asm:anchoring}
	for $\gamma\geq 1$, and additionally Assumption~\ref{asm:compatible} for
	$0\leq\gamma<1$ and every fixed
	$h,\varepsilon>0$ there exists a minimiser
	$Q_{h,\varepsilon}\in H^1(\Ohmh;\Szero)$ of $\Fhe$. Moreover the
	minimiser satisfies the uniform energy bound
	\begin{equation}
		\label{eq:energy-bound}
		\|Q_{h,\varepsilon}\|^{2}_{H^1(\Ohmh)}
		+\|Q_{h,\varepsilon}\|^{4}_{L^{4}(\Ohmh)}
		\leq Ch,
	\end{equation}
	with $C$ independent of $\varepsilon$ and $h$.
	
	Furthermore, any sequence
	$\{Q_{h,\varepsilon}\}\subset H^1(\Ohmh;\Szero)$ satisfying
	\begin{equation}
		\label{eq:ub}
		\Fhe(Q_{h,\varepsilon})\leq Ch
	\end{equation}
	with $C$ independent of $\varepsilon$ and $h$ also
	satisfies~\cref{eq:energy-bound}.
\end{theorem}

\begin{proof}
	We first show that any sequence satisfying~\cref{eq:ub} satisfies~\cref{eq:energy-bound}.
	Assume $\{Q_{h,\varepsilon}\}$ satisfies $\Fhe(Q_{h,\varepsilon})\leq Ch$.
	Using the coercivity of the elastic energy~\cref{eq:elastic-coercive},
	the growth estimate~\cref{B3}$_2$ for $E_{\mathrm{bulk}}$, and the
	non-negativity of $E_{\mathrm{surf}}^{h,\varepsilon}$ from
	Lemma~\ref{lem:surf-properties}\ref{surf-growth}, we obtain
	\begin{align*}
		\frac{\lambda}{2}\|\nabla Q_{h,\varepsilon}\|^{2}_{L^{2}(\Ohmh)}
		+c_1\|Q_{h,\varepsilon}\|^{4}_{L^{4}(\Ohmh)}
		&\leq E_{\mathrm{el}}(Q_{h,\varepsilon})
		+\int_{\Ohmh}\!\bigl(f_b(Q_{h,\varepsilon})+c_2\bigr)\,dx\\
		&\leq\Fhe(Q_{h,\varepsilon})+c_2|\Ohmh|
		\leq Ch+c_2|\omega|h\leq Ch.
	\end{align*}
	Since $\|Q_{h,\varepsilon}\|^2_{L^2(\Ohmh)}
	\leq Ch^{1/2}\|Q_{h,\varepsilon}\|^2_{L^4(\Ohmh)}$,
	we immediately deduce~\cref{eq:energy-bound}. This proves the second
	claim and supplies the compactness needed for the existence part.
	
	Now we prove existence of a minimiser for fixed $h,\varepsilon$.
	Proposition~\ref{prop:inf-bounds} yields
	$-\infty<-C_\ell h\leq\mhe\leq C_u h<+\infty$.
	Choose a minimising sequence $\{Q_n\}_n\subset H^{1}(\Ohmh;\Szero)$ such
	that
	\[
	\liminf_{n\to\infty}\Fhe(Q_n)=\mhe,\qquad
	\Fhe(Q_n)\leq C_u h^{\min(\gamma,1)}\leq Ch
	\]
	for all sufficiently large $n$. The coercivity argument above gives the
	uniform bound
	\begin{equation}
		\label{eq:min-seq-bound}
		\|Q_n\|^{2}_{H^1(\Ohmh)}+\|Q_n\|^{4}_{L^{4}(\Ohmh)}\leq Ch,
	\end{equation}
	with $C$ independent of $n$, $\varepsilon$, and $h$.
	
	Since $H^1(\Ohmh;\Szero)$ is reflexive, we extract a subsequence
	(not relabelled) such that
	\[
	Q_n\rightharpoonup Q_{h,\varepsilon}\quad
	\text{weakly in }H^1(\Ohmh;\Szero).
	\]
	
	\textbf{Elastic term.}
	The map $Q\mapsto E^{h,\varepsilon}_{\mathrm{el}}(Q)
	=\frac{1}{2}\int_{\Ohmh}A(\cdot)\,\nabla Q:\nabla Q\,dx$
	is convex and continuous in $\nabla Q$, hence weakly lower
	semicontinuous in $H^1(\Ohmh;\Szero)$. Consequently,
	\[
	E^{h,\varepsilon}_{\mathrm{el}}(Q_{h,\varepsilon})
	\leq\liminf_{n\to\infty}E^{h,\varepsilon}_{\mathrm{el}}(Q_n).
	\]
	
	\textbf{Bulk term.}
	By the Rellich--Kondrachov theorem on the bounded Lipschitz domain
	$\Ohmh\subset\mathbb{R}^3$, the embedding
	$H^1(\Ohmh;\Szero)\hookrightarrow\hookrightarrow L^4(\Ohmh;\Szero)$
	is compact, so $Q_n\to Q_{h,\varepsilon}$ strongly in $L^4(\Ohmh;\Szero)$.
	Passing to a subsequence (not relabelled), we may assume
	$Q_n(x)\to Q_{h,\varepsilon}(x)$ for a.e.\ $x\in\Ohmh$. By continuity
	of $f_b$,
	\[
	f_b(Q_n(x))\to f_b(Q_{h,\varepsilon}(x))\quad\text{for a.e.\ }x\in\Ohmh.
	\]
	The growth bound \cref{B3}$_3$ gives $|f_b(Q_n)|\leq c_3(1+|Q_n|^4)$.
	Since $Q_n\to Q_{h,\varepsilon}$ strongly in $L^4(\Ohmh;\Szero)$, the
	sequence $\{|Q_n|^4\}$ is uniformly integrable in $L^1(\Ohmh)$, and
	hence so is $\{f_b(Q_n)\}$. Vitali's convergence theorem therefore yields
	\[
	E_{\mathrm{bulk}}^{h,\varepsilon}(Q_n)
	\to E_{\mathrm{bulk}}^{h,\varepsilon}(Q_{h,\varepsilon}).
	\]
	Since the limit is the same along every subsequence, the convergence
	holds for the original sequence.     
	
	\textbf{Surface term.}
	The trace operator $H^1(\Ohmh;\Szero)\to L^2(\Gamma_h^w;\Szero)$
	is compact, so
	$Q_n|_{\Gamma_h^w}\to Q_{h,\varepsilon}|_{\Gamma_h^w}$
	strongly in $L^2(\Gamma_h^w;\Szero)$.
	Since $\mathcal{A}^{h,\varepsilon}(x,\nu)$ is closed for every $x$
	by~Assumption \ref{asm:anchoring}~\ref{SS1}, the map
	$Q\mapsto\operatorname{dist}(Q,\mathcal{A}^{h,\varepsilon}(x,\nu))$
	is $1$-Lipschitz on $\Szero$, uniformly in $x$. Therefore
	\begin{align*}
		\bigl|f_s^{h,\varepsilon}(Q_n,x,\nu)
		-f_s^{h,\varepsilon}(Q_{h,\varepsilon},x,\nu)\bigr|
		&\leq W^{h,\varepsilon}(x,\nu)
		\bigl|\operatorname{dist}^2(Q_n,\mathcal{A}^{h,\varepsilon})
		-\operatorname{dist}^2(Q_{h,\varepsilon},\mathcal{A}^{h,\varepsilon})\bigr|\\
		&\leq h^\gamma\|W\|_{L^\infty}
		|Q_n-Q_{h,\varepsilon}|
		\bigl(|Q_n|+|Q_{h,\varepsilon}|+2C_{\mathcal{A}}\bigr),
	\end{align*}
	where we used $|a^2-b^2|=|a-b||a+b|$ and the $1$-Lipschitz bound on
	$\operatorname{dist}$. Integrating and applying the Cauchy--Schwarz
	inequality on $\Gamma_h^w$ gives
	\[
	\bigl|E_{\mathrm{surf}}^{h,\varepsilon}(Q_n)
	-E_{\mathrm{surf}}^{h,\varepsilon}(Q_{h,\varepsilon})\bigr|
	\leq
	h^\gamma\|W\|_{L^\infty}
	\|Q_n-Q_{h,\varepsilon}\|_{L^2(\Gamma_h^w)}
	\Bigl(\|Q_n\|_{L^2(\Gamma_h^w)}
	+\|Q_{h,\varepsilon}\|_{L^2(\Gamma_h^w)}
	+2C_{\mathcal{A}}\mathcal{H}^2(\Gamma_h^w)^{1/2}\Bigr).
	\]
	The first factor tends to zero by strong $L^2$ convergence of the
	traces; the second factor is bounded since $\{Q_n\}$ is bounded in
	$H^1(\Ohmh;\Szero)$. Hence
	$E_{\mathrm{surf}}^{h,\varepsilon}(Q_n)\to
	E_{\mathrm{surf}}^{h,\varepsilon}(Q_{h,\varepsilon})$.
		
	Combining the three terms,
	\[
	\Fhe(Q_{h,\varepsilon})
	\leq\liminf_{n\to\infty}\Fhe(Q_n)=\mhe
	\leq\Fhe(Q_{h,\varepsilon}),
	\]
	so $Q_{h,\varepsilon}$ is a minimiser. Finally, because the minimiser
	is the weak limit of a sequence satisfying~\cref{eq:min-seq-bound} and
	the norms are weakly lower semicontinuous, $Q_{h,\varepsilon}$
	inherits~\cref{eq:energy-bound}. 
%	(This also follows directly from the
%	first part of the proof, since
%	$\Fhe(Q_{h,\varepsilon})=\mhe\leq Ch$.)
\end{proof}

	%%%%%%%%%%%%%%%%%%%%%%%%%%%%%%%%
	
	\subsection{Compactness of the sequence $Q$-tensors}
	\label{sec:compactness}
	
	The goal of this section is to extract, from every energy bounded sequence, the weak limit objects
	that serve as trial fields in the sequential $\Gamma$-limits.  The argument proceeds
	in two stages: we first derive uniform bounds on the components of the thickness decomposition of
	Section~\ref{sec:decomposition}, and then apply the unfolding compactness of Lemma~\ref{A11} to
	pass to two-scale limits.
	
	%\subsection{Two‑scale limit $Q$-tensors}
	Let $\{Q_{h,\varepsilon}\}_{h,\varepsilon}\subset H^1(\Omega_h;\Szero)$ be a sequence satisfying the
	uniform energy bound~\cref{eq:energy-bound} of Theorem~\ref{prop:existence}. Then, using
	Proposition~\ref{prop:thickness-decomposition} gives the existence of $\Qc_{h,\e}\in H^1(\o;\Szero)$
	and $\Rc_{h,\e}\in H^1(\O_h;\Szero)$ such that
	\[
	Q_{h,\e}(x)=\Qc_{h,\e}(x')+\Rc_{h,\e}(x),\quad (x',x_3)\in \O_h.
	\]
	Moreover, we have the following estimates:
	\begin{equation}\label{es:01}
		\begin{aligned}
			\|\Qc_{h,\e}\|_{H^1(\o)}\leq C,\quad
			\|Q_{h,\e}-\Qc_{h,\e}\|_{L^2(\O_h)} &\le Ch^{3/2},\\
			\|\Rc_{h,\e}\|_{L^2(\O_h)}+h\|\nabla \Rc_{h,\e}\|_{L^2(\O_h)} &\leq Ch^{3/2}.
		\end{aligned}
	\end{equation}
	The above estimates together with the properties~\cref{eq:L2-bound}, \cref{eq:Peh-grad-yprime}
	of the rescaling unfolding operator give
	\begin{equation}\label{es:02}
		\begin{aligned}
			\frac{1}{h}\|\Peh(\Rc_{h,\e})\|_{L^2(\o\X Y)}
			+\frac{1}{h}\|\partial_{y_3}\Peh(\Rc_{h,\e})\|_{L^2(\o\X Y)} &\le C,\\
			\frac{1}{h}\|\nabla_{y'}\Peh(\Rc_{h,\e})\|_{L^2(\o\X Y)} &\le C\frac{\e}{h}.
		\end{aligned}
	\end{equation}
	The constants are independent of $\e$ and $h$.
	
	\begin{lemma}\label{lem:71}
		There exist $\Qc\in H^1(\o;\Szero)$ and $\wh Q\in L^2(\o;H^1_{\pe,0}(Y';\Szero))$ such that,
		up to a subsequence,
		\begin{equation}\label{con:01}
			\begin{aligned}
				\Qc_{h,\e} &\rightharpoonup \Qc, \quad &&\text{weakly in $H^1(\o;\Szero)$},\\
				\Te(\Qc_{h,\e}) &\to \Qc, \quad &&\text{strongly in $L^2(\o\X Y';\Szero)$},\\
				\Te(\nabla'\Qc_{h,\e}) &\rightharpoonup \nabla'\Qc+\nabla_{y'}\wh Q, \quad
				&&\text{weakly in $L^2(\o\X Y';\Szero\otimes\mathbb{R}^2)$}.
			\end{aligned}
		\end{equation}
		Moreover, there exists $\wh R\in L^2(\o\X Y';H^1(I;\Szero))$ satisfying
		\begin{equation}\label{es:ort}
			\int_{I} \wh R(\cdot,y_3)\,\dd y_3 = 0,
			\quad \text{a.e.\ in $\omega\X Y'$},
		\end{equation}
		such that
		\begin{equation}\label{con:03}
			\frac{1}{h}\Peh(\Rc_{h,\e}) \;\rightharpoonup\; \wh R,
			\qquad\text{weakly in $L^2(\o\X Y';H^1(I;\Szero))$}.
		\end{equation}
		For the in‑plane gradient we distinguish three regimes depending on
		$\rho$ from \cref{eq:rho}.
		\begin{enumerate}
			\item[(R3)] For $\rho=0$: $\wh R\in L^2(\o\X Y';H^1(I;\Szero))$ and
			\begin{equation}\label{es:03}
				\Peh(\nabla'\Rc_{h,\e}) \rightharpoonup 0
				\quad\text{weakly in $L^2(\o\X Y;\Szero\otimes\R^2)$}.
			\end{equation}
			\item[(R2)] For $\rho\in(0,\infty)$: $\wh R\in L^2(\o;H_{\pe}^1(Y;\Szero))$ and
			\begin{equation}\label{es:04}
				\Peh(\nabla'\Rc_{h,\e}) \rightharpoonup \rho \nabla_{y'}\wh R
				\quad\text{weakly in $L^2(\o\X Y;\Szero\otimes\R^2)$}.
			\end{equation}
			\item[(R1)] For $\rho=\infty$: $\wh R\in L^2(\o;H^1(I;\Szero))$,
			and there exists $\wh R_1\in L^2(\o\X I;H^1_{\pe,0}(Y';\Szero))$ satisfying~\cref{es:ort}
			such that
			\begin{equation}\label{es:05}
				\Peh(\nabla'\Rc_{h,\e}) \rightharpoonup \nabla_{y'}\wh R_1
				\quad\text{weakly in $L^2(\o\X Y;\Szero\otimes\R^2)$}.
			\end{equation}
		\end{enumerate}
		Finally, we have the following convergences of the unfolded $Q$-tensor sequence:
		\begin{align}
			\label{con:79}
			\Peh(Q_{h,\e}) \;&\to\; \Qc,
			\quad\text{strongly in $L^2(\o\X Y;\Szero)$},\\
			\label{con55}
			\Pi_{h,\e}(Q_{h,\e})&\to \Qc
			\qquad\text{strongly in }L^p(\omega\times Y;\Szero)
			\quad\text{for every }2< p<4.\\
			\label{con56}
			\Pi_{h,\e}(Q_{h,\e})&\rightharpoonup \Qc
			\qquad\text{weakly in }L^4(\omega\times Y;\Szero).
		\end{align}
	\end{lemma}
	
	\begin{proof}
		The macroscopic field convergence~\cref{con:01}$_1$ is a direct consequence of
		\cref{es:01}$_{1}$ and the Banach-Alaoglu theorem.  Combining this with the properties of the
		periodic unfolding operator, Lemma~\ref{A11} gives~\cref{con:01}$_{2,3}$.

		Set $\wh{R}_{h,\varepsilon} := \frac{1}{h}\Peh(\Rc_{h,\varepsilon})$.
		From~\cref{es:02}$_{1,2}$ the sequences $\{\wh{R}_{h,\varepsilon}\}$ and
		$\{\partial_{y_3}\wh{R}_{h,\varepsilon}\}$ are bounded in $L^2(\omega\times Y;\Szero)$,
		hence~\cref{con:03} holds for some $\wh R\in L^2(\o\X Y';H^1(I;\Szero))$.
		The orthogonality~\cref{Ortho01} translates to $\int_I \wh{R}_{h,\varepsilon}\,\dd y_3 = 0$ a.e.
		and passes to the weak limit, giving~\cref{es:ort}.
		
		The in‑plane gradient satisfies
		\begin{equation}\label{eq:grad-rel}
			\Peh(\nabla'\Rc_{h,\e})
			= \frac1\e\nabla_{y'}\Peh(\Rc_{h,\e})
			= \frac{h}{\e}\nabla_{y'}\wh{R}_{h,\e}.
		\end{equation}
		By~\cref{es:02}$_{2}$ we have $\|\nabla_{y'}\wh{R}_{h,\e}\|_{L^2}
		= \frac1h\|\nabla_{y'}\Peh(\Rc_{h,\e})\|_{L^2} \le C\frac{\e}{h} $.
		For $\rho\in(0,\infty)$ this implies $\nabla_{y'}\wh{R}_{h,\e}$ is bounded, so together with
		the bounds on $\wh{R}_{h,\e}$ and $\partial_{y_3}\wh{R}_{h,\e}$ we obtain
		$\wh{R}_{h,\e} \rightharpoonup \wh R$ weakly in $L^2(\o;H^1_{\pe}(Y;\Szero))$; in particular
		$\nabla_{y'}\wh{R}_{h,\e} \rightharpoonup \nabla_{y'}\wh R$ weakly.  Multiplying by $\rho$
		yields~\cref{es:04}.
		
		For $\rho = 0$ (i.e.\ $h/\e\to 0$), we use an integration by parts.
		For every smooth $Y$-periodic test function $\psi\in H^1(\o\X Y;\Szero\otimes\R^2)$,
		\[
		\int_{\o\X Y} \Peh(\nabla'\Rc_{h,\e})\cdot\psi
		= -\frac{h}{\e}\int_{\o\X Y} \wh{R}_{h,\e}\,\operatorname{div}_{y'}\psi
		\le \frac{h}{\e} \|\wh{R}_{h,\e}\|_{L^2}\|\psi\|_{H^1} \le C\frac{h}{\e}\to 0.
		\]
		Hence $\Peh(\nabla'\Rc_{h,\e})\rightharpoonup 0$, which is~\cref{es:03}.
		
		For $\rho = \infty$ (i.e.\ $\e/h\to 0$), the bound $\|\nabla_{y'}\wh{R}_{h,\e}\|_{L^2}
		\le C\frac{\e}{h}\to 0$ forces the weak limit $\wh R$ to be independent of $y'$.
		
		Now set $U_{h,\e} := \frac1\e\Peh(\Rc_{h,\e})$.
		Because $\nabla_{y'}U_{h,\e} = \Peh(\nabla'\Rc_{h,\e})$ is bounded in
		$L^2(\o\X Y;\Szero\otimes\R^2)$ by~\cref{es:02}, we can extract a weak limit for the
		gradients.  Let $\widetilde U_{h,\e}$ be the function obtained from $U_{h,\e}$ by subtracting
		its $Y'$-average for a.e.\ $(x',y_3)$.  Then $\widetilde U_{h,\e}$ has zero mean on $Y'$ and
		$\nabla_{y'}\widetilde U_{h,\e} = \nabla_{y'}U_{h,\e}$.  The Poincaré-Wirtinger inequality on
		$Y'$ gives
		$\|\widetilde U_{h,\e}\|_{L^2(\o\X Y)}\le C\|\nabla_{y'}U_{h,\e}\|_{L^2(\o\X Y)}\le C$.
		Thus $\widetilde U_{h,\e}$ is bounded in $L^2(\o\X I;H^1_{\per,0}(Y';\Szero))$.  Up to a
		subsequence, $\widetilde U_{h,\e}\rightharpoonup \wh R_1$ weakly in that space.
		Consequently $\Peh(\nabla'\Rc_{h,\e}) = \nabla_{y'}U_{h,\e} = \nabla_{y'}\widetilde U_{h,\e}
		\rightharpoonup \nabla_{y'}\wh R_1$, which is~\cref{es:05}.  The orthogonality
		$\int_I \wh R_1\,\dd y_3 = 0$ follows from the same property for $\frac1\e\Peh(\Rc_{h,\e})$.
		
		Finally, the strong convergence~\cref{con:79} follows from
		$\|Q_{h,\e}-\Qc_{h,\e}\|_{L^2(\O_h)}\le Ch^{3/2}$ and the strong convergence
		$\Te(\Qc_{h,\e})\to\Qc$, because
		$\Peh(Q_{h,\e}) = \Te(\Qc_{h,\e}) + \Peh(\Rc_{h,\e})$ and
		$\|\Peh(\Rc_{h,\e})\|_{L^2(\o\X Y)} \leq  Ch^{-1/2}\|\Rc_{h,\e}\|_{L^2(\O_h)}\le Ch\to 0$.
		
				By the coercive growth \cref{B3} of the Landau--de Gennes bulk potential and the
		energy bound \cref{eq:energy-bound}, the sequence $\{\Pi_{h,\e}(Q_{h,\e})\}$ is bounded in
		$L^4(\omega\times Y;\Szero)$ by \cref{eq:L2-bound}. Since we have the strong convergence \cref{con:79}, using
		interpolation gives the strong convergences \cref{con55} and the weak convergence \cref{con56}.

This completes the proof.
	\end{proof}
	
	Now we identify the two‑-scale limit of the full gradient.
	
	\begin{lemma}[Admissible corrector spaces for the full rescaled gradient]
		\label{prop:admissible-correctors}
		Assume the conclusions of Lemma~\ref{lem:71} hold.  Then, up to the same subsequence,
		\begin{equation}
			\label{con:80}
			\Peh(\nabla Q_{h,\e})\rightharpoonup \mathcal G_\rho=(G_\rho',G_3)
			\qquad\text{weakly in }L^2(\omega\times Y;\Szero\otimes\R^3),
		\end{equation}
		and the limit pair admits the following representation in each regime.
		
		\textbf{(R3) The regime \(\rho=0\).}
		There exists $\wh Q_0\in L^2(\omega;H^1_{\rm per,0}(Y';\Szero))$ such that
		\[
		G_0'=\nabla'\Qc+\nabla_{y'}\wh Q_0,
		\qquad
		G_3=\partial_{y_3}\wh R,
		\]
		and one may take $\wh Q_0(x',y')=\wh Q(x',y')$.
		
		\textbf{(R2) The regime \(\rho\in(0,\infty)\).}
		There exists $\wh Q_\rho\in L^2(\omega;H^1_{\rm per,0}(Y;\Szero))$ such that
		\[
		G'_\rho=\nabla'\Qc+\nabla_{y'}\wh Q_\rho,
		\qquad
		G_3=\partial_{y_3}\wh R,
		\]
		and one may take $\wh Q_\rho=\wh Q+\rho\,\wh R$.
		
		\textbf{(R1) The regime \(\rho=\infty\).}
		There exists $\wh Q_\infty\in L^2(\omega\times I;H^1_{\rm per,0}(Y';\Szero))$ such that
		\[
		G_\infty'=\nabla'\Qc+\nabla_{y'}\wh Q_\infty,
		\qquad
		G_3=\partial_{y_3}\wh R,
		\]
		and one may take $\wh Q_\infty=\wh Q+\wh R_1$.
	\end{lemma}
	
	\begin{proof}
		By the thickness decomposition,
		$Q_{h,\e}(x',x_3)=\Qc_{h,\e}(x')+\Rc_{h,\e}(x',x_3)$,
		hence after applying $\Peh$,
		\[
		\Peh(Q_{h,\e})(x',y)
		= \Te(\Qc_{h,\e})(x',y')
		+ \Peh(\Rc_{h,\e})(x',y).
		\]
		Differentiating and using the scaling properties of $\Peh$,
		\begin{align*}
			\frac1\e\nabla_{y'}\Peh(Q_{h,\e})
			= \Te(\nabla'\Qc_{h,\e}) + \Peh(\nabla'\Rc_{h,\e}),\quad
			\frac1h\partial_{y_3}\Peh(Q_{h,\e})= \partial_{y_3}\Bigl(\frac1h\Peh(\Rc_{h,\e})\Bigr).
		\end{align*}
		The first term converges to $\nabla'\Qc+\nabla_{y'}\wh Q$ by~\cref{con:01}$_3$.
		The second term converges to $\partial_{y_3}\wh R$ by~\cref{con:03}.
		For the in‑plane remainder gradient we insert the limits from
		\cref{es:03}--\cref{es:05} according to the regime.
		Collecting everything yields exactly the stated formulas, with the identifications
		$\wh Q_0:=\wh Q$, $\wh Q_\rho:=\wh Q+\rho\wh R$, $\wh Q_\infty:=\wh Q+\wh R_1$.
	\end{proof}
	
	\paragraph{Traces on the top and bottom faces}
	Denote the traces on the top and bottom faces by
	$$Q^\pm_{h,\varepsilon}(x') := Q_{h,\varepsilon}(x', \pm h/2).$$
	Then, using the estimates \cref{eq:Qtrace} and the uniform bound \cref{eq:energy-bound}, there exists a constant $C > 0$, independent of $h$ and $\varepsilon$, such that
	\begin{equation}
		\|Q^\pm_{h,\varepsilon}\|^2_{L^2(\omega;\Szero)} \le C.
		\label{eq:trace-bound}
	\end{equation}
	Moreover, the remainder contribution satisfies
	\begin{equation}
		\|\Rc^\pm_{h,\varepsilon}\|^2_{L^2(\omega;\Szero)} \le Ch^2,
		\label{eq:trace-remainder-vanish}
	\end{equation}
	where $\Rc^\pm_{h,\e}(x')=\Rc_{h,\e}(x',\pm h/2)$ for $x'\in\o$. 
	
	Finally, we have the strong convergence for the trace $Q$-tensor sequence.
	\begin{lemma}
		\label{lem:trace-bound}
		For the same subsequence as in Lemma \ref{lem:71}, we have the strong convergence
		\begin{equation}\label{conN}
			\Te(Q^\pm_{h,\varepsilon}) \to \Qc
			\quad\text{strongly in }L^2(\omega\times Y';\Szero).
		\end{equation}
	\end{lemma}
	
	\begin{proof}	
		For the strong convergence of the unfolded traces we use the decomposition Definition~\ref{def:decomp} again.
		Because $\Te$ is a continuous operator from $L^2(\omega)$ to $L^2(\omega\times Y')$ (up to a
		scale factor that does not affect strong convergence),
		\[
		\Te(Q^\pm_{h,\varepsilon}) = \Te(\Qc_{h,\varepsilon}) + \Te\bigl(\Rc^\pm_{h,\varepsilon}\bigr).
		\]
		From~\cref{con:01}$_2$ we have $\Te(\Qc_{h,\varepsilon})\to\Qc$ strongly in
		$L^2(\omega\times Y')$.
		The remainder term satisfies
		$$\|\Te(\Rc^\pm_{h,\varepsilon})\|_{L^2(\omega\times Y')}
		\leq \|\Rc^\pm_{h,\varepsilon}\|_{L^2(\omega)}\le Ch\to 0,$$
		by~\cref{eq:trace-remainder-vanish}. Thus the sum converges strongly to $\Qc$,
		which is exactly~\cref{conN}.
	\end{proof}
	
	\subsection{Two-scale limit LdG energy and recovery sequence}
	
	\label{sec:limitspace}
	In this section, we derive the two-scale limit energy, using the liminf inequality, establish the existence of a minimiser, and end with the construction of the recovery sequence. 
	
	\begin{proposition}\label{prop:liminf}
		Let $\{Q_{h,\e}\}_{h,\e}\subset H^1(\O_h;\Szero)$ be the sequence satisfying the energy bound \cref{eq:ub}. Then there exists $(\Qc,\wh Q_\rho,\wh R)\in\A^\gamma_\rho$ such that (up to a subsequence for $\gamma\geq 0$ and $\rho\in[0,\infty]$)
		\begin{equation}\label{limE}
			\GF^\gamma_\rho(\Qc,\wh Q_\rho,\wh R)\leq\liminf_{(h,\e)\to(0,0)}{\Fhe(Q_{h,\e})\over h},
		\end{equation}
		where $\Gc_\rho$ is defined using $(\Qc,\wh Q_\rho,\wh R)\in\A^\gamma_\rho$ as in Lemma \ref{lem:71}--\ref{prop:admissible-correctors}.
	\end{proposition}
	\begin{proof}
		Since $\{Q_{h,\e}\}_{h,\e}\subset H^1(\O_h;\Szero)$ satisfy the energy bound \cref{eq:ub}, using Theorem \ref{prop:existence}, we have that the sequence of Q-tensors satisfies the energy estimates \cref{eq:energy-bound}. Then, using the decomposition \ref{def:decomp}, we have the estimates \cref{es:01}--\cref{es:02} and \cref{eq:trace-bound}--\cref{eq:trace-remainder-vanish}. So, we obtain convergences as in Lemmas \ref{lem:71}--\ref{lem:trace-bound}. So, there exists $(\Qc,\wh Q_\rho,\wh R)\in \A^\gamma_\rho$, with the convergences \cref{con:79}, \cref{con:80} and \cref{conN}.
		
		It remains to prove the right-hand inequality \cref{limE}. We treat the three contributions to $\GF^\gamma_\rho$ separately.
		
		\textbf{Elastic term.}
		Split the physical domain into interior and boundary strip,
		\[
		\frac{2}{h}E^{h,\varepsilon}_{\mathrm{el}}(Q_{h,\varepsilon})
		\;=\; \underbrace{\frac{1}{h}\int_{\wh{\Omega}_{h,\varepsilon}}\!\!
			\bA\!\left(\frac{x'}{\varepsilon},\frac{x_3}{h}\right)\nabla Q_{h,\varepsilon} : \nabla Q_{h,\varepsilon}\,dx}_{=:I^{\mathrm{el,int}}_{h,\varepsilon}}
		\;+\; \underbrace{\frac{1}{h}\int_{\Omega_h\setminus\wh{\Omega}_{h,\varepsilon}}\!\!
			\bA\!\left(\frac{x'}{\varepsilon},\frac{x_3}{h}\right)\nabla Q_{h,\varepsilon} : \nabla Q_{h,\varepsilon}\,dx}_{=:I^{\mathrm{el,strip}}_{h,\varepsilon}}.
		\]
		On the complete cells, the $L^p$-identity~\cref{eq:Peh-Lp-exact} with $p=2$ 
		together with the unfolding of the coefficient~\cref{eq:Peh-integral-exact} gives
		\[
		I^{\mathrm{el,int}}_{h,\varepsilon}
		\;=\; \int_{\omega\times Y} \bA(y)\,\Pi_{h,\varepsilon}(\nabla Q_{h,\varepsilon}) 
		: \Pi_{h,\varepsilon}(\nabla Q_{h,\varepsilon})\,dy\,dx'.
		\]
		By the uniform ellipticity of $\bA$(see Assumption \ref{asm:elastic}\ref{A3}), we have
		\begin{equation}\label{eq:el-strip-liminf}
			I^{\mathrm{el,strip}}_{h,\varepsilon} \;\geq\; 0.
		\end{equation}
%		Note that the strip term is handled by the non-negativity of the integrand alone; 
%		no equi-integrability of $|\nabla Q_{h,\varepsilon}|^2$ on the strip is required, 
%		which is what would have been needed for a full limit.
		
		By Lemma~\ref{prop:admissible-correctors}, 
		$\Pi_{h,\varepsilon}(\nabla Q_{h,\varepsilon})\rightharpoonup \Gc_\rho$ weakly 
		in $L^2(\omega\times Y;\Szero\otimes\mathbb{R}^3)$, with $\Gc_\rho$ defined from 
		$(\Qc,\wh Q_\rho,\wh R)\in\A^\gamma_\rho$. The functional 
		$F\mapsto\int_{\omega\times Y}\bA(y)\,F:F\,dy\,dx'$ is convex and strongly continuous 
		on $L^2(\omega\times Y;\Szero\otimes\mathbb{R}^3)$, hence weakly lower semicontinuous. 
		Therefore
		\begin{equation}\label{eq:el-int-liminf}
			\int_{\omega\times Y} \bA(y)\,\Gc_\rho:\Gc_\rho\,dy\,dx'
			\;\leq\; \liminf_{(h,\varepsilon)\to(0,0)} I^{\mathrm{el,int}}_{h,\varepsilon}.
		\end{equation}
		Combining~\cref{eq:el-strip-liminf} and~\cref{eq:el-int-liminf},
		\begin{multline}\label{eq:liminfE}
			E^\rho_{\mathrm{el}}(\Qc,\wh{Q}_\rho,\wh{R})
			\;=\;{1\over 2} \int_{\omega\times Y} \bA(y)\,\Gc_\rho:\Gc_\rho\,dy\,dx'
			\;\leq\; {1\over 2}\liminf_{(h,\varepsilon)\to(0,0)} I^{\mathrm{el,int}}_{h,\varepsilon}
			\;+\;{1\over 2} \liminf_{(h,\varepsilon)\to(0,0)} I^{\mathrm{el,strip}}_{h,\varepsilon}\\
			\;\leq\; \liminf_{(h,\varepsilon)\to(0,0)}\frac{1}{h}E^{h,\varepsilon}_{\mathrm{el}}(Q_{h,\varepsilon}).
		\end{multline}
		This completes the elastic term.
	
		\textbf{Bulk term.}
		Split the physical domain into interior and boundary strip,
		\[
		\frac{1}{h}\int_{\Omega_h} f_b(Q_{h,\varepsilon})\,dx
		\;=\; \underbrace{\frac{1}{h}\int_{\widehat\Omega_{h,\varepsilon}}\!\! f_b(Q_{h,\varepsilon})\,dx}_{=:I_{h,\varepsilon}^{\mathrm{int}}}
		\;+\; \underbrace{\frac{1}{h}\int_{\Lambda_\varepsilon\times hI}\!\! f_b(Q_{h,\varepsilon})\,dx}_{=:I_{h,\varepsilon}^{\mathrm{strip}}}.
		\]
		By the unfolding identity~\cref{eq:Peh-integral-remainder} (in its exact form on
		$\widehat\Omega_{h,\varepsilon}$) together with the polynomial
		property~\cref{eq:Peh-bulk-potential},
		\[
		I_{h,\varepsilon}^{\mathrm{int}}
		\;=\; \int_{\omega\times Y}\Pi_{h,\varepsilon}\bigl(f_b(Q_{h,\varepsilon})\bigr)\,dx'\,dy
		\;=\; \int_{\omega\times Y} f_b\bigl(\Pi_{h,\varepsilon}(Q_{h,\varepsilon})(x',y)\bigr)\,dx'\,dy.
		\]
		For the strip term, the lower bound $f_b(Q)\geq -c_2$ from~\cref{B3} gives
		\begin{equation}\label{eq:strip-liminf}
			I_{h,\varepsilon}^{\mathrm{strip}}
			\;\geq\; -c_2\,\frac{|\Lambda_\varepsilon\times hI|}{h}
			\;=\; -c_2\,|\Lambda_\varepsilon| \;\longrightarrow\; 0.
		\end{equation}
		It remains to prove the lower bound for the unfolded interior term. 
		
		Using the form  \cref{eq:fb} of the Landau--de Gennes bulk potential,
		\[
		f_b(Q)
		=
		a |Q|^2
		-
		b \tr(Q^3)
		+
		c |Q|^4,
		\qquad a,b\in\R,\quad c>0,
		\]
		we treat the three terms separately. The strong $L^2$ convergence gives
		\[
		\int_{\omega\times Y}|\Pi_{h,\e}(Q_{h,\e})|^2\,dx'\,dy
		\to
		\int_{\omega\times Y}|\Qc|^2\,dx'\,dy .
		\]
		The strong $L^3$ convergence gives
		\[
		\int_{\omega\times Y}\tr(\Pi_{h,\e}(Q_{h,\e})^3)\,dx'\,dy
		\to
		\int_{\omega\times Y}\tr(\Qc^3)\,dx'\,dy .
		\]
		Finally, by weak lower semicontinuity of the convex functional
		$U\mapsto\int_{\omega\times Y}|U|^4\,dx'\,dy$ on $L^4$,
		\[
		\int_{\omega\times Y}|\Qc|^4\,dx'\,dy
		\leq
		\liminf_{(h,\varepsilon)\to(0,0)}
		\int_{\omega\times Y}|\Pi_{h,\e}(Q_{h,\e})|^4\,dx'\,dy .
		\]
		Combining these three convergences yields
		\[
		\int_{\omega\times Y} f_b(\Qc)\,dx'\,dy
		\leq
		\liminf_{(h,\varepsilon)\to(0,0)}
		\int_{\omega\times Y} f_b(\Pi_{h,\e}(Q_{h,\e}))\,dx'\,dy .
		\]
		Since $\Qc=\Qc(x')$ is independent of $y$ and $|Y|=1$,
		\[
		\int_{\omega\times Y} f_b(\Qc)\,dx'\,dy
		=
		\int_\omega f_b(\Qc)\,dx' .
		\]
		Combining this with the boundary-strip lower bound \cref{eq:strip-liminf} gives
		\begin{equation}\label{eq:liminfS}
				\int_\omega f_b(\Qc)\,dx'
			\leq
			\liminf_{(h,\varepsilon)\to(0,0)}
			\frac1h E^{h,\varepsilon}_{\mathrm{bulk}}(Q_{h,\varepsilon}).
		\end{equation}
		This completes the bulk term.

	\noindent\textbf{Surface term for $\gamma\geq1$.}
	For $\sigma\in\{-1,1\}$, set
	\[
	Q^\sigma_{h,\varepsilon}(x')
	:=
	Q_{h,\varepsilon}\left(x',\frac{\sigma h}{2}\right),
	\qquad
	W_\sigma(y')
	:=
	W\left(y',\frac{\sigma}{2},\sigma e_3\right),\quad\text{and}\quad
	\mathcal A_\sigma(y')
	:=
	\mathcal A\left(y',\frac{\sigma}{2},\sigma e_3\right).
	\]
	
	Suppose first that $\gamma=1$. Then
	\begin{equation}
		\label{eq:surface-gamma-one}
		\frac1hE_{\mathrm{surf}}^{h,\varepsilon}
		(Q_{h,\varepsilon})
		=
		\sum_{\sigma\in\{-1,1\}}
		\int_\omega
		W_\sigma\left(\frac{x'}{\varepsilon}\right)
		\operatorname{dist}^2
		\left(
		Q^\sigma_{h,\varepsilon}(x'),
		\mathcal A_\sigma\left(\frac{x'}{\varepsilon}\right)
		\right)
		\,\mathrm dx'.
	\end{equation}
	For each $\sigma\in\{-1,1\}$, decompose the corresponding integral as
	\[
	J^\sigma_{\mathrm{int}}
	+
	J^\sigma_{\mathrm{strip}},
	\]
	where the two terms are obtained by integrating over
	$\widehat\omega_\varepsilon$ and $\Lambda_\varepsilon$, respectively.
	By the unfolding and product identities,
	\[
	J^\sigma_{\mathrm{int}}
	=
	\int_{\omega\times Y'}
	W_\sigma(y')
	\operatorname{dist}^2
	\left(
	\mathcal T_\varepsilon
	(Q^\sigma_{h,\varepsilon})(x',y'),
	\mathcal A_\sigma(y')
	\right)
	\,\mathrm dx'\,\mathrm dy'.
	\]
	Moreover, $J^\sigma_{\mathrm{strip}}\geq0$.
	
	By~\cref{conN},
	\[
	\mathcal T_\varepsilon(Q^\sigma_{h,\varepsilon})
	\longrightarrow
	\Qc
	\qquad\text{strongly in }
	L^2(\omega\times Y';\Szero).
	\]
	Since the distance from a nonempty closed set is $1$-Lipschitz and
	$W_\sigma\in L^\infty(Y')$, it follows that
	\[
	\sqrt{W_\sigma}\,
	\operatorname{dist}
	\left(
	\mathcal T_\varepsilon(Q^\sigma_{h,\varepsilon}),
	\mathcal A_\sigma
	\right)
	\longrightarrow
	\sqrt{W_\sigma}\,
	\operatorname{dist}(\Qc,\mathcal A_\sigma),\quad\text{	strongly in $L^2(\omega\times Y')$}.
	\]
Hence
	\begin{equation}
		\label{eq:surf-int-liminf}
		\lim_{(h,\varepsilon)\to(0,0)}
		J^\sigma_{\mathrm{int}}
		=
		\int_{\omega\times Y'}
		W_\sigma(y')
		\operatorname{dist}^2
		\bigl(\Qc(x'),\mathcal A_\sigma(y')\bigr)
		\,\mathrm dx'\,\mathrm dy'.
	\end{equation}
	Combining this convergence with the non-negativity of the boundary-strip
	terms gives
	\begin{equation}
		\label{Sur01}
		\liminf_{(h,\varepsilon)\to(0,0)}
		\frac1hE_{\mathrm{surf}}^{h,\varepsilon}
		(Q_{h,\varepsilon})
		\geq
		\sum_{\sigma\in\{-1,1\}}
		\int_{\omega\times Y'}
		W_\sigma(y')
		\operatorname{dist}^2
		\bigl(\Qc(x'),\mathcal A_\sigma(y')\bigr)
		\,\mathrm dx'\,\mathrm dy'.
	\end{equation}
	
	Define
	\begin{equation}
		\label{eq:hom-surf-density-pm}
		\overline f_s^\sigma(q)
		:=
		\int_{Y'}
		W_\sigma(y')
		\operatorname{dist}^2
		\bigl(q,\mathcal A_\sigma(y')\bigr)
		\,\mathrm dy',
		\qquad q\in\Szero,
	\end{equation}
	and
	\[
	\overline f_s(q)
	:=
	\sum_{\sigma\in\{-1,1\}}
	\overline f_s^\sigma(q).
	\]
	The quadratic growth estimate in
	Lemma~\ref{lem:surf-properties}\ref{surf-growth} implies
	\begin{equation}
		\label{eq:LimG}
		0\leq\overline f_s(q)
		\leq
		2C_s(1+|q|^2)
		\qquad\text{for all }q\in\Szero,
	\end{equation}
	where $|Y'|=1$. Therefore,~\eqref{Sur01} becomes
	\begin{equation}
		\label{eq:surf-limit-compact}
		\liminf_{(h,\varepsilon)\to(0,0)}
		\frac1hE_{\mathrm{surf}}^{h,\varepsilon}
		(Q_{h,\varepsilon})
		\geq
		\int_\omega\overline f_s(\Qc(x'))\,\mathrm dx'.
	\end{equation}
	
	Suppose now that $\gamma>1$. With the preceding notation,
	\[
	\frac1hE_{\mathrm{surf}}^{h,\varepsilon}
	(Q_{h,\varepsilon})
	=
	h^{\gamma-1}
	\sum_{\sigma\in\{-1,1\}}
	\int_\omega
	W_\sigma\left(\frac{x'}{\varepsilon}\right)
	\operatorname{dist}^2
	\left(
	Q^\sigma_{h,\varepsilon},
	\mathcal A_\sigma\left(\frac{x'}{\varepsilon}\right)
	\right)
	\,\mathrm dx'.
	\]
	By the quadratic growth of the surface density and the uniform trace
	bounds,
	\[
	\int_\omega
	W_\sigma\left(\frac{x'}{\varepsilon}\right)
	\operatorname{dist}^2
	\left(
	Q^\sigma_{h,\varepsilon},
	\mathcal A_\sigma\left(\frac{x'}{\varepsilon}\right)
	\right)
	\,\mathrm dx'
	\leq
	C\left(
	1+\|Q^\sigma_{h,\varepsilon}\|_{L^2(\omega)}^2
	\right)
	\leq C.
	\]
	Since $h^{\gamma-1}\to0$, we conclude that
	\begin{equation}
		\label{eq:surf-vanish}
		\frac1hE_{\mathrm{surf}}^{h,\varepsilon}
		(Q_{h,\varepsilon})
		\longrightarrow0.
	\end{equation}
		
\noindent\textbf{Surface term for $0\leq\gamma<1$.}
By the non-negativity of the elastic energy, the lower bound
$f_b\geq-c_2$, and the uniform energy bound,
\[
0
\leq
E_{\mathrm{surf}}^{h,\varepsilon}(Q_{h,\varepsilon})
\leq
F_{h,\varepsilon}(Q_{h,\varepsilon})
+c_2|\Omega_h|
\leq
Ch.
\]
Using the notation introduced above, this gives
\[
h^\gamma
\sum_{\sigma\in\{-1,1\}}
\int_\omega
W_\sigma\left(\frac{x'}{\varepsilon}\right)
\operatorname{dist}^2
\left(
Q^\sigma_{h,\varepsilon}(x'),
\mathcal A_\sigma\left(\frac{x'}{\varepsilon}\right)
\right)
\,\mathrm dx'
\leq
Ch.
\]
Since $W_\sigma\geq w_0>0$ a.e.\ by
Assumption~\ref{asm:compatible}, we obtain
\begin{equation}
	\label{eq:strong-anchoring-estimate}
	\sum_{\sigma\in\{-1,1\}}
	\int_\omega
	\operatorname{dist}^2
	\left(
	Q^\sigma_{h,\varepsilon}(x'),
	\mathcal A_\sigma\left(\frac{x'}{\varepsilon}\right)
	\right)
	\,\mathrm dx'
	\leq
	Ch^{1-\gamma}
	\longrightarrow0.
\end{equation}
Let $\chi_\varepsilon
:=
\mathbf{1}_{\widehat\omega_\varepsilon}$.
Applying the unfolding identity to the restriction of
\eqref{eq:strong-anchoring-estimate} to the complete cells yields
\begin{equation}
	\label{eq:unfolded-strong-anchoring}
	\sum_{\sigma\in\{-1,1\}}
	\int_{\omega\times Y'}
	\chi_\varepsilon(x')
	\operatorname{dist}^2
	\left(
	\mathcal T_\varepsilon
	(Q^\sigma_{h,\varepsilon})(x',y'),
	\mathcal A_\sigma(y')
	\right)
	\,\mathrm dx'\,\mathrm dy'
	\leq
	Ch^{1-\gamma}.
\end{equation}
Moreover, $\chi_\varepsilon\longrightarrow1
\quad\text{in }L^p(\omega)$ for every $1\leq p<\infty$,
because $|\Lambda_\varepsilon|\to0$.

By Lemma~\ref{lem:trace-bound},
\[
\mathcal T_\varepsilon(Q^\sigma_{h,\varepsilon})
\longrightarrow
\Qc
\qquad\text{strongly in }
L^2(\omega\times Y';\Szero).
\]
Set
\[
d^\sigma_{h,\varepsilon}(x',y')
:=
\operatorname{dist}
\left(
\mathcal T_\varepsilon(Q^\sigma_{h,\varepsilon})(x',y'),
\mathcal A_\sigma(y')
\right),\quad\text{and}\quad
d^\sigma(x',y')
:=
\operatorname{dist}
\bigl(\Qc(x'),\mathcal A_\sigma(y')\bigr).
\]
Since the distance from a nonempty closed set is $1$-Lipschitz,
\[
d^\sigma_{h,\varepsilon}
\longrightarrow
d^\sigma
\qquad\text{strongly in }L^2(\omega\times Y'),\quad\text{consequently},\quad
\chi_\varepsilon d^\sigma_{h,\varepsilon}
\longrightarrow
d^\sigma
\qquad\text{strongly in }L^2(\omega\times Y').
\]
On the other hand, \eqref{eq:unfolded-strong-anchoring} gives
\[
\|\chi_\varepsilon
d^\sigma_{h,\varepsilon}\|_{L^2(\omega\times Y')}^2
\leq
Ch^{1-\gamma}
\longrightarrow0.
\]
It follows that $\operatorname{dist}
\bigl(\Qc(x'),\mathcal A_\sigma(y')\bigr)
=0$
for a.e.\ $(x',y')\in\omega\times Y'$ and each
$\sigma\in\{-1,1\}$. Thus,
\[
\Qc(x')\in\mathcal A_\sigma(y')
\quad\text{for a.e.\ }y'\in Y'
\quad\text{and each }\sigma\in\{-1,1\},\quad\text{for a.e.\ $x'\in\omega$}.
\]
By the definition of $\mathscr A$, this means $\Qc(x')\in\mathscr A$ for a.e.~$x'\in\omega$.
Together with the previously established regularity
$\Qc\in H^1(\omega;\Szero)$, we conclude that $\Qc\in H^1(\omega;\mathscr A)$.

Combining the elastic, bulk, and surface estimates proves~\cref{limE}
and completes the proof.
	\end{proof}

\begin{theorem}\label{thm:existence-limit}
	For every $\gamma\geq0$ and $\rho\in[0,\infty]$, the two-scale
	limit energy~\cref{eq:limE} admits a minimiser in
	$\A^\gamma_\rho$.
\end{theorem}

\begin{proof}
	Fix $\gamma\geq0$ and $\rho\in[0,\infty]$, and set
	\begin{equation}\label{eq:limmin}
		\Gm^\gamma_\rho
		:=
		\inf_{(\Qc,\wh Q_\rho,\wh R)\in\A^\gamma_\rho}
		\GF^\gamma_\rho(\Qc,\wh Q_\rho,\wh R).
	\end{equation}
	By the compatibility assumptions \cref{asm:compatible}, $\A^\gamma_\rho$ contains a
	finite-energy admissible triple. Hence
	$\Gm^\gamma_\rho<+\infty$.
	
	Let $(\Qc,\wh Q_\rho,\wh R)\in\A^\gamma_\rho$. The uniform
	ellipticity~\cref{eq:elastic-coercive} and
	Lemma~\ref{lem:coercivity} give
	\begin{align*}
		E^\rho_{\mathrm{el}}(\Qc,\wh Q_\rho,\wh R)
		&\geq
		\lambda\|\Gc_\rho\|_{L^2(\omega\times Y)}^2\geq
		\lambda c_\rho
		\Bigl(
		\|\nabla'\Qc\|_{L^2(\omega)}^2
		+\|\wh Q_\rho\|_
		{L^2(\omega\times I;H^1(Y'))}^2
		+\|\wh R\|_
		{L^2(\omega\times Y';H^1(I))}^2
		\Bigr),
	\end{align*}
	where $\Gc_\rho$ is the generalized gradient associated with
	$(\Qc,\wh Q_\rho,\wh R)$. Moreover, the bulk lower bound
	\cref{B3} yields $E_{\mathrm{bulk}}(\Qc)
	\geq
	c_1\|\Qc\|_{L^4(\omega)}^4-c_2|\omega|$,
	while the surface contribution is non-negative. Therefore
	\begin{multline}
		\GF^\gamma_\rho(\Qc,\wh Q_\rho,\wh R)
		\geq{}
		\lambda c_\rho\|\nabla'\Qc\|_{L^2(\omega)}^2
		+c_1\|\Qc\|_{L^4(\omega)}^4
		-c_2|\omega|\\
		+
		\lambda c_\rho
		\Bigl(
		\|\wh Q_\rho\|_
		{L^2(\omega\times I;H^1(Y'))}^2
		+\|\wh R\|_
		{L^2(\omega\times Y';H^1(I))}^2
		\Bigr).
		\label{eq:limit-coercivity}
	\end{multline}
	In particular, $-\infty<\Gm^\gamma_\rho<+\infty$.	
	Let $\{(\Qc_n,\wh Q_{\rho,n},\wh R_n)\}_n
	\subset\A^\gamma_\rho$
	be a minimizing sequence such that
	\[
	\GF^\gamma_\rho(\Qc_n,\wh Q_{\rho,n},\wh R_n)
	\longrightarrow
	\Gm^\gamma_\rho.
	\]
	Estimate~\cref{eq:limit-coercivity}, together with the boundedness
	of $\omega$, implies
	\begin{equation}\label{EsLim}
		\|\Qc_n\|_{H^1(\omega)}^2
		+\|\wh Q_{\rho,n}\|_
		{L^2(\omega\times I;H^1(Y'))}^2
		+\|\wh R_n\|_
		{L^2(\omega\times Y';H^1(I))}^2
		\leq C,
	\end{equation}
	where $C$ is independent of $n$ but may depend on $\rho$.
	
By the definition of $\mathscr E_\gamma$, the admissible class
$\A^\gamma_\rho$ is weakly sequentially closed in the product Hilbert space $H^1(\omega;\Szero)\times\C_\rho$.
Therefore, the uniform bound~\cref{EsLim} and weak sequential compactness
give, after passing to a subsequence,
\[
\Qc_n\rightharpoonup\Qc^\ast
\quad\text{in }H^1(\omega;\Szero),\quad\text{and}\quad
(\wh Q_{\rho,n},\wh R_n)
\rightharpoonup
(\wh Q^\ast_\rho,\wh R^\ast)
\quad\text{in }\C_\rho,
\]
for some $(\Qc^\ast,\wh Q^\ast_\rho,\wh R^\ast)\in\A^\gamma_\rho$.
In particular, $\Qc^\ast\in H^1(\omega;\mathscr E_\gamma)$.

The generalized gradient depends linearly and continuously on the
admissible triple. Hence
\[
\Gc_{\rho,n}\rightharpoonup\Gc^\ast_\rho
\quad\text{in }
L^2(\omega\times Y;\Szero\otimes\mathbb R^3),
\]
where $\Gc_{\rho,n}$ and $\Gc^\ast_\rho$ are associated with
$(\Qc_n,\wh Q_{\rho,n},\wh R_n)$ and
$(\Qc^\ast,\wh Q^\ast_\rho,\wh R^\ast)$, respectively.
Since the elastic energy is a continuous convex quadratic functional of
the generalized gradient, it is weakly lower semicontinuous. Therefore,
\[
E^\rho_{\mathrm{el}}
(\Qc^\ast,\wh Q^\ast_\rho,\wh R^\ast)
\leq
\liminf_{n\to\infty}
E^\rho_{\mathrm{el}}
(\Qc_n,\wh Q_{\rho,n},\wh R_n).
\]

Since $\omega\subset\mathbb R^2$ is bounded, the embedding $H^1(\omega;\Szero)
\hookrightarrow\hookrightarrow
L^4(\omega;\Szero)$
is compact. Thus, after passing to a further subsequence,
\[
\Qc_n\longrightarrow\Qc^\ast
\quad\text{strongly in }L^4(\omega;\Szero).
\]
By the continuity and quartic growth of $f_b$, the Vitali convergence
theorem gives
\[
E_{\mathrm{bulk}}(\Qc_n)
\longrightarrow
E_{\mathrm{bulk}}(\Qc^\ast).
\]

If $\gamma=1$, the continuity and quadratic growth of $\overline f_s$,
together with the strong convergence
\[
\Qc_n\longrightarrow\Qc^\ast
\quad\text{in }L^2(\omega;\Szero),\quad\text{give}\quad
E^1_{\mathrm{surf}}(\Qc_n)
\longrightarrow
E^1_{\mathrm{surf}}(\Qc^\ast).
\]
If $\gamma>1$, the limit surface energy vanishes identically. If
$0\leq\gamma<1$, the anchoring constraint is incorporated into
$H^1(\omega;\mathscr E_\gamma)$, and the limit surface energy again
vanishes on the admissible class.

Combining these facts, we obtain
\[
\GF^\gamma_\rho
(\Qc^\ast,\wh Q^\ast_\rho,\wh R^\ast)
\leq
\liminf_{n\to\infty}
\GF^\gamma_\rho
(\Qc_n,\wh Q_{\rho,n},\wh R_n)
=
\Gm^\gamma_\rho.
\]
On the other hand, since
$(\Qc^\ast,\wh Q^\ast_\rho,\wh R^\ast)\in\A^\gamma_\rho$, the definition
of $\Gm^\gamma_\rho$ gives
\[
\Gm^\gamma_\rho
\leq
\GF^\gamma_\rho
(\Qc^\ast,\wh Q^\ast_\rho,\wh R^\ast)\implies
\GF^\gamma_\rho
(\Qc^\ast,\wh Q^\ast_\rho,\wh R^\ast)
=
\Gm^\gamma_\rho,
\]
and $(\Qc^\ast,\wh Q^\ast_\rho,\wh R^\ast)$ is a minimiser of
$\GF^\gamma_\rho$ in $\A^\gamma_\rho$.
\end{proof}
	
	\begin{theorem}[Recovery sequence]
		\label{thm:reco}
		Let $(\Qc^*, \widehat{Q}^*_\rho, \widehat{R}^*) \in \A^\gamma_\rho$. There exist a sequence $\{Q^\ast_{h,\e}\}_{h,\e}\subset H^1(\O_h;\Szero)$ such that 
		\begin{equation}
			\label{th:recM}
			\begin{aligned}
				\Peh(Q^\ast_{h,\e}) &\to \Qc^\ast,\quad &&\text{strongly in $L^p(\o\X Y;\Szero)$, for $2\leq p\leq 4$},\\
				\Te({Q^\ast_{h,\varepsilon}}_{|x_3=\pm h/2}) &\to \Qc^\ast 
				\quad&&\text{strongly in }L^2(\omega\times Y';\Szero),\\
				\Pi_{h,\varepsilon}(\nabla Q^\ast_{h,\varepsilon}) 
				&\to \Gc^\ast_\rho 
				\quad &&\text{strongly in } L^2(\omega \times Y; \Szero \otimes \mathbb{R}^3),
			\end{aligned}
		\end{equation}
		where $\Gc^\ast_\rho$ is defined from $(\Qc^\ast,\wh Q^\ast_\rho,\wh R^\ast)\in \A^\gamma_\rho$ as in Lemma \ref{prop:admissible-correctors}. Moreover, we have
		\begin{equation}
			\label{eq:limsupR}
			\limsup_{(h,\e)\to(0,0)}{\mhe\over h}\leq 	\limsup_{(h,\e)\to(0,0)}{\Fhe(Q^\ast_{h,\e})\over h} = \GF^\gamma_\rho(\Qc^\ast,\wh Q^\ast_\rho,\wh R^\ast),
		\end{equation}
		where $F^\gamma_\rho$ is defined in \cref{eq:limE}.	
	\end{theorem}
	\begin{proof}
		It suffices to show \cref{eq:limsupR} for $(\Qc^\ast,\wh Q_\rho^\ast,\wh R^\ast)\in \B^\gamma_\rho$, where
		$$\B^\gamma_\rho=\left\{\begin{aligned}
			&C^\infty(\wo{\o};\Szero)\X\D_\rho,\quad &&\text{for $\gamma\geq1$},\\
			&C^\infty(\wo{\o};\mathscr{A})\X\D_\rho,\quad 	&&\text{for $0\leq \gamma<1$},
		\end{aligned}\right.$$ 
		For each regime, denote by $\D_\rho$ the smooth dense subspace of microscopic $Q$-tensor spaces $\C_\rho$ from the Lemma \ref{lem:dense-subspaces}.
		Since $\B^\gamma_\rho$ is dense 
		in each component of $\A^\gamma_\rho$ as 		$C^\infty(\bar\omega;\mathscr{A})$ (resp.~$C^\infty(\o;\Szero)$) is dense in $H^1(\omega;\mathscr{A})$ (resp.~$H^1(\o;\Szero)$)
		by Lemma~\ref{lem:density-Astar}, for every $\eta > 0$ there exists smooth data
		$(\Qc^\eta, \widehat{Q}^\eta_\rho, \widehat{R}^\eta)$ approximating 
		$(\Qc^*, \widehat{Q}^*_\rho, \widehat{R}^*)$ in $\A^\gamma_\rho$ such that
		\[
		\bigl|\GF^\gamma_\rho(\Qc^\eta, \widehat{Q}^\eta_\rho, \widehat{R}^\eta) 
		- \GF^\gamma_\rho(\Qc^*, \widehat{Q}^*_\rho, \widehat{R}^*)\bigr| \leq \eta.
		\]
		It therefore suffices to prove~\cref{eq:limsupR} for smooth data and conclude by a 
		diagonal argument in $\eta \to 0$. In what follows we drop the superscript $\eta$ 
		and work with fixed smooth $(\Qc^\ast, \widehat{Q}^\ast_\rho, \widehat{R}^\ast)$.

		\noindent\textbf{Step 1: Construction of the recovery sequence.}
		We construct an explicit recovery sequence 
		$\{Q^\ast_{h,\varepsilon}\}_{h,\varepsilon} \subset H^1(\Omega_h; \Szero)$
		in each regime. Recall that $x = (x', x_3) \in \Omega_h = \omega \times hI$.
		
		We treat the cases  $\varrho \in [0,\infty)$ and $\varrho = \infty$ separately.
		
		\noindent\textit{Regime $\varrho \in [0,\infty)$.}
		Here $\widehat{Q}^\ast_\rho = \widehat{Q}^\ast + \rho\widehat{R}^\ast$, 
		where $\widehat{Q}^\ast = \widehat{Q}^\ast(x', y')$ is $y_3$-independent and 
		$\widehat{R}^\ast = \widehat{R}^\ast(x', y', y_3)$ depends on the full cell. Define
		\begin{equation}
			\label{eq:recovery-rho}
			Q^\ast_{h,\varepsilon}(x', x_3) 
			:= \Qc^\ast(x') 
			+ h\,\widehat{R}^\ast\!\left(x', \frac{x'}{\varepsilon}, \frac{x_3}{h}\right) 
			+ \varepsilon\,\widehat{Q}^\ast\!\left(x', \frac{x'}{\varepsilon}\right).
		\end{equation}
		Since $\Qc^\ast \in C^\infty(\wo{\omega}; \Szero)$ or $C^\infty(\o;\mathscr{A})$, $\widehat{Q}^\ast \in C^\infty(\bar\omega; 
		C^\infty_{\mathrm{per}}(Y'; \Szero))$, and 
		$\widehat{R}^\ast \in C^\infty(\bar\omega; C^\infty_{\mathrm{per}}(Y; \Szero))$ 
		with $\int_I \widehat{R}^\ast\,dy_3 = 0$, the field $Q^\ast_{h,\varepsilon}$ is admissible in 
		$H^1(\Omega_h; \Szero)$.
		
		Using the derivative identities~\cref{eq:Peh-grad-yprime} of Proposition~\ref{prop:unfolding-basic}, we have the following expression for the in-plane gradient
		\begin{equation*}
			\Pi_{h,\varepsilon}(\nabla' Q^\ast_{h,\varepsilon})(x', y)
			= \nabla' \Qc^\ast(x')+h\,\nabla'\widehat{R}^\ast\!\left(x', y\right)
			+ {h\over\e}\,\nabla_{y'}\! \widehat{R}^\ast\!(x', y) 
			+ \e\nabla'\wh Q^\ast(x',y)+\,\nabla_{y'} \widehat{Q}^\ast\!(x', y'),
		\end{equation*}
		which give
		\begin{equation}
			\label{eq:inplane-cv}
			\Pi_{h,\varepsilon}(\nabla' Q^\ast_{h,\varepsilon}) 
			\;\longrightarrow\; 
			\nabla' \Qc^\ast + \nabla_{y'}\widehat{Q}^\ast + \rho\,\nabla_{y'}\widehat{R}^\ast 
			\quad \text{strongly in } L^2(\omega \times Y; \Szero \otimes \mathbb{R}^2).
		\end{equation}
		Differentiating $Q^\ast_{h,\varepsilon}$ in $x_3$:
		\[
		\Peh(\partial_{x_3} Q^\ast_{h,\varepsilon})(x', y) 
		= \partial_{y_3}\widehat{R}^\ast(x', y).
		\]
		Since $\widehat{Q}^\ast$ is $y_3$-independent, $\partial_{y_3}\widehat{Q}^\ast = 0$. After unfolding,
		\begin{equation}
			\label{eq:thickness-cv}
			\Pi_{h,\varepsilon}(\partial_{x_3} Q^\ast_{h,\varepsilon})
			= (\partial_{y_3}\widehat{R}^\ast)(x', y)
			\;\longrightarrow\; 
			\partial_{y_3}\widehat{R}^\ast 
			\quad \text{strongly in } L^2(\omega \times Y; \Szero),
		\end{equation}
		where the convergence is in fact an equality for each $(h,\varepsilon)$, since after 
		unfolding the oscillatory argument $(x'/\varepsilon, x_3/h)$ maps exactly to $(y', y_3)$.
		
		Combining~\cref{eq:inplane-cv} and~\cref{eq:thickness-cv},
		\begin{equation}
			\label{eq:strong-grad}
			\Pi_{h,\varepsilon}(\nabla Q^\ast_{h,\varepsilon}) 
			\;\longrightarrow\; 
			(\nabla' \Qc^\ast + \nabla_{y'}\widehat{Q}_\rho^\ast , 		\partial_{y_3}\widehat{R}^\ast )=\Gc^\ast_\rho 
			\quad \text{strongly in } L^2(\omega \times Y; \Szero \otimes \mathbb{R}^3).
		\end{equation}
		
		\textit{Regime $\varrho = \infty$.}
		Here $\widehat{Q}^\ast_\infty \in C^\infty(\bar\omega \times I; C^\infty_{\mathrm{per}}(Y'; \Szero))$ 
		and $\widehat{R}^\ast$ is $y'$-independent. Define
		\begin{equation}
			\label{eq:recovery-infty}
			Q^\ast_{h,\varepsilon}(x', x_3) 
			:= \Qc^\ast(x') 
			+ h\,\widehat{R}^\ast\!\left(x', \frac{x_3}{h}\right) 
			+ \varepsilon\,\widehat{Q}_\infty^\ast\!\left(x', \frac{x'}{\varepsilon}, \frac{x_3}{h}\right),
		\end{equation}
		which give
		$$ 	
		\begin{aligned}
			\Pi_{h,\varepsilon}(\nabla' Q^\ast_{h,\varepsilon})(x', y)
			&= \nabla' \Qc^\ast(x')+h\nabla'\wh R^\ast(x',y_3)+\e\nabla'\widehat{Q}_\infty^\ast\!\left(x', y\right)
			+ \nabla_{y'}\widehat{Q}_\infty^\ast\!(x', y) ,\\
			\Peh(\partial_{x_3} Q^\ast_{h,\varepsilon})(x', y) 
			&= \partial_{y_3}\widehat{R}^\ast(x', y_3)+{\e\over h}\p_{y_3}\wh Q^\ast_\infty(x',y).
		\end{aligned}$$
		Since $\rho=\infty$, we have ${\e\over h}\to 0$, we get the convergences 
		\begin{align*}
			\Pi_{h,\varepsilon}(\nabla' Q^\ast_{h,\varepsilon}) 
			&\;\longrightarrow\; 
			\nabla' \Qc^\ast + \,\nabla_{y'}\widehat{Q}_\infty^\ast 
			\quad \text{strongly in } L^2(\omega \times Y; \Szero \otimes \mathbb{R}^2),\\
			\Pi_{h,\varepsilon}(\partial_{x_3} Q^\ast_{h,\varepsilon})
			&\;\longrightarrow\; 
			\partial_{y_3}\widehat{R}^\ast 
			\quad \text{strongly in } L^2(\omega \times Y; \Szero),
		\end{align*}
		combined
		\begin{equation}\label{Lims1}
			\Pi_{h,\varepsilon}(\nabla Q^\ast_{h,\varepsilon}) 
			\;\longrightarrow\; 
			(\nabla' \Qc^\ast + \nabla_{y'}\widehat{Q}_\infty^\ast , 		\partial_{y_3}\widehat{R}^\ast )=\Gc^\ast_\infty
			\quad \text{strongly in } L^2(\omega \times Y; \Szero \otimes \mathbb{R}^3).
		\end{equation}
		From the above definitions, we directly have the following strong convergences (for $\rho\in[0,\infty]$)
		\begin{equation}\label{Lims2}
			\begin{aligned}
				\Peh(Q^\ast_{h,\e}) &\to \Qc^\ast,\quad \text{strongly in $L^p(\o\X Y;\Szero)$, for $2\leq p\leq 4$},\\
				\Te({Q^\ast_{h,\varepsilon}}_{|x_3=\pm h/2}) &\to \Qc^\ast 
				\quad\text{strongly in }L^2(\omega\times Y';\Szero).
			\end{aligned}
		\end{equation}
		
		\textbf{Step 2: Convergence of the energy.}
		
		\textbf{Elastic term.}
		Split into interior and boundary strip,
		\[
		\frac{1}{h}E^{h,\varepsilon}_{\mathrm{el}}(Q^*_{h,\varepsilon})
		= \frac{1}{2h}\!\int_{\wh\Omega_{h,\varepsilon}}\!\!
		\bA\!\left(\frac{x'}{\varepsilon},\frac{x_3}{h}\right)\nabla Q^*_{h,\varepsilon}:\nabla Q^*_{h,\varepsilon}\,dx
		+ \frac{1}{2h}\!\int_{\Omega_h\setminus\wh\Omega_{h,\varepsilon}}\!\!
		\bA\!\left(\frac{x'}{\varepsilon},\frac{x_3}{h}\right)\nabla Q^*_{h,\varepsilon}:\nabla Q^*_{h,\varepsilon}\,dx.
		\]
		On the complete cells, the $L^2$-identity~\cref{eq:Peh-Lp-exact} and the unfolding 
		of the coefficient~\cref{eq:Peh-integral-exact} give
		\[
		\frac{1}{2h}\!\int_{\wh\Omega_{h,\varepsilon}}\!\!
		\bA\!\left(\frac{x'}{\varepsilon},\frac{x_3}{h}\right)\nabla Q^*_{h,\varepsilon}:\nabla Q^*_{h,\varepsilon}\,dx
		= \frac{1}{2}\!\int_{\omega\times Y}\!\bA(y)\,\Pi_{h,\varepsilon}(\nabla Q^*_{h,\varepsilon})
		:\Pi_{h,\varepsilon}(\nabla Q^*_{h,\varepsilon})\,dy\,dx'.
		\]
		Since $\bA\in L^\infty(Y;\mathcal{L}(\mathbb{T},\mathbb{T}))$ and 
		$\Pi_{h,\varepsilon}(\nabla Q^*_{h,\varepsilon}) \to \Gc^*_\rho$ strongly in 
		$L^2(\omega\times Y;\Szero\otimes\mathbb{R}^3)$ by~\cref{eq:strong-grad}, the 
		integrand on the right converges as
		\[
		\frac{1}{2}\!\int_{\omega\times Y}\!\bA(y)\,\Pi_{h,\varepsilon}(\nabla Q^*_{h,\varepsilon})
		:\Pi_{h,\varepsilon}(\nabla Q^*_{h,\varepsilon})\,dy\,dx'
		\;\longrightarrow\;
		\frac{1}{2}\!\int_{\omega\times Y}\!\bA(y)\,\Gc^*_\rho:\Gc^*_\rho\,dy\,dx'.
		\]
		For the strip remainder, the smoothness of the correctors gives 
		$|\nabla Q^*_{h,\varepsilon}|\leq M_1$ uniformly, so using the upper bound \cref{eq:elastic-coercive}, we obtain
		\[
		\frac{1}{2h}\!\int_{\Omega_h\setminus\wh\Omega_{h,\varepsilon}}\!\!
		\bA\!\left(\frac{x'}{\varepsilon},\frac{x_3}{h}\right)\nabla Q^*_{h,\varepsilon}:\nabla Q^*_{h,\varepsilon}\,dx
		\;\leq\; C\,|\Lambda_\varepsilon|\;\longrightarrow\;0.
		\]
		Combining,
		\[
		\frac{1}{h}E^{h,\varepsilon}_{\mathrm{el}}(Q^*_{h,\varepsilon})
		\;\longrightarrow\;
		\frac{1}{2}\!\int_{\omega\times Y}\!\bA(y)\,\Gc^*_\rho:\Gc^*_\rho\,dy\,dx'
		= E^\rho_{\mathrm{el}}(\Qc^*,\widehat{Q}^*_\rho,\widehat{R}^*).
		\]
		\textbf{Bulk term.}
		The $L^\infty$-bound 
		$\|Q^*_{h,\varepsilon}\|_\infty\leq M$ together with the growth 
		bound~\cref{B3} gives $|f_b(Q^*_{h,\varepsilon})|\leq c_3(1+M^4)$, so
		\[
		\frac{1}{h}\int_{\Omega_h\setminus\wh\Omega_{h,\varepsilon}}|f_b(Q^*_{h,\varepsilon})|\,dx
		\;\leq\; c_3(1+M^4)\,|\Lambda_\varepsilon|\;\longrightarrow\;0.
		\]
		On the interior, by the unfolding identity~\cref{eq:Peh-integral-remainder} 
		and the polynomial property~\cref{eq:Peh-bulk-potential},
		\[
		\frac{1}{h}\int_{\wh\Omega_{h,\varepsilon}}f_b(Q^*_{h,\varepsilon})\,dx
		= \int_{\omega\times Y}f_b\bigl(\Pi_{h,\varepsilon}(Q^*_{h,\varepsilon})\bigr)\,dx'\,dy.
		\]
		By~\cref{Lims2},
		\[
		\Pi_{h,\varepsilon}(Q^*_{h,\varepsilon})\to\Qc^*
		\quad\text{strongly in }L^2(\omega\times Y;\Szero).
		\]
		Moreover, the family is uniformly bounded in \(L^\infty\). Since \(f_b\) is a
		polynomial, it is Lipschitz on bounded subsets of \(\Szero\). Hence there
		exists \(C_M>0\) such that
		\[
		|f_b(P)-f_b(Q)|\le C_M|P-Q|
		\qquad\text{whenever } |P|,|Q|\le M.
		\]
		Therefore
		\[
		\begin{aligned}
			\left\|
			f_b\bigl(\Pi_{h,\varepsilon}(Q^*_{h,\varepsilon})\bigr)
			-
			f_b(\Qc^*)
			\right\|_{L^1(\omega\times Y)}
			&\le
			C_M
			\left\|
			\Pi_{h,\varepsilon}(Q^*_{h,\varepsilon})-\Qc^*
			\right\|_{L^1(\omega\times Y)}
			\\
			&\le
			C_M|\omega\times Y|^{1/2}
			\left\|
			\Pi_{h,\varepsilon}(Q^*_{h,\varepsilon})-\Qc^*
			\right\|_{L^2(\omega\times Y)}
			\to0.
		\end{aligned}
		\]
		Thus
		\[
		\int_{\omega\times Y}
		f_b\bigl(\Pi_{h,\varepsilon}(Q^*_{h,\varepsilon})\bigr)\,dx'\,dy
		\to
		\int_{\omega\times Y}f_b(\Qc^*)\,dx'\,dy.
		\]
		Hence
		\[
		\frac{1}{h}E^{h,\varepsilon}_{\mathrm{bulk}}(Q^*_{h,\varepsilon})
		\;\longrightarrow\;
		\int_{\omega\times Y}f_b(\Qc^*)\,dx'\,dy
		= \int_\omega f_b(\Qc^*(x'))\,dx'
		= E_{\mathrm{bulk}}(\Qc^*).
		\]
\noindent\textbf{Surface term for $\gamma\geq1$.}
For $\sigma\in\{-1,1\}$, set
\[
W_\sigma(y')
:=
W\left(y',\frac{\sigma}{2},\sigma e_3\right),
\qquad
\mathcal A_\sigma(y')
:=
\mathcal A\left(y',\frac{\sigma}{2},\sigma e_3\right),
\]
and denote the traces by
\[
Q^{*,\sigma}_{h,\varepsilon}(x')
:=
Q^\ast_{h,\varepsilon}
\left(x',\frac{\sigma h}{2}\right).
\]
Then
\begin{equation}
	\label{eq:recovery-surface-scaled}
	\frac1h
	E_{\mathrm{surf}}^{h,\varepsilon}
	(Q^\ast_{h,\varepsilon})
	=
	h^{\gamma-1}
	\sum_{\sigma\in\{-1,1\}}
	\int_\omega
	W_\sigma\left(\frac{x'}{\varepsilon}\right)
	\operatorname{dist}^2
	\left(
	Q^{*,\sigma}_{h,\varepsilon}(x'),
	\mathcal A_\sigma\left(\frac{x'}{\varepsilon}\right)
	\right)
	\,\mathrm dx'.
\end{equation}

By the quadratic growth of the surface density and the uniform
$L^\infty$-bound on the traces,
\begin{align*}
	&\int_{\Lambda_\varepsilon}
	W_\sigma\left(\frac{x'}{\varepsilon}\right)
	\operatorname{dist}^2
	\left(
	Q^{*,\sigma}_{h,\varepsilon}(x'),
	\mathcal A_\sigma\left(\frac{x'}{\varepsilon}\right)
	\right)
	\,\mathrm dx'
	\leq
	C_s(1+M^2)|\Lambda_\varepsilon|
	\longrightarrow0.
\end{align*}
Recall from Proposition \ref{prop:liminf}, $\chi_\varepsilon=
\mathbf 1_{\widehat\omega_\varepsilon}$.
The unfolding identity on the complete cells gives
\begin{align*}
	&\int_{\widehat\omega_\varepsilon}
	W_\sigma\left(\frac{x'}{\varepsilon}\right)
	\operatorname{dist}^2
	\left(
	Q^{*,\sigma}_{h,\varepsilon}(x'),
	\mathcal A_\sigma\left(\frac{x'}{\varepsilon}\right)
	\right)
	\,\mathrm dx'
	=
	\int_{\omega\times Y'}
	\chi_\varepsilon(x')
	W_\sigma(y')
	\operatorname{dist}^2
	\left(
	\mathcal T_\varepsilon
	(Q^{*,\sigma}_{h,\varepsilon})(x',y'),
	\mathcal A_\sigma(y')
	\right)
	\,\mathrm dx'\,\mathrm dy'.
\end{align*}
%By~\cref{Lims2},
%\[
%\mathcal T_\varepsilon
%(Q^{*,\sigma}_{h,\varepsilon})
%\longrightarrow
%\Qc^\ast
%\quad\text{strongly in }
%L^2(\omega\times Y';\Szero).
%\]
Since the distance from a nonempty closed set is $1$-Lipschitz, using the strong convergence \cref{Lims2}$_2$, we obtain
\[
\operatorname{dist}
\left(
\mathcal T_\varepsilon
(Q^{*,\sigma}_{h,\varepsilon}),
\mathcal A_\sigma(y')
\right)
\longrightarrow
\operatorname{dist}
\bigl(\Qc^\ast,\mathcal A_\sigma(y')\bigr)
\]
strongly in $L^2(\omega\times Y')$. The boundedness of $W_\sigma$ yield
\begin{align*}
	&\int_{\widehat\omega_\varepsilon}
	W_\sigma\left(\frac{x'}{\varepsilon}\right)
	\operatorname{dist}^2
	\left(
	Q^{*,\sigma}_{h,\varepsilon}(x'),
	\mathcal A_\sigma\left(\frac{x'}{\varepsilon}\right)
	\right)
	\,\mathrm dx'
	\longrightarrow
	\int_{\omega\times Y'}
	W_\sigma(y')
	\operatorname{dist}^2
	\bigl(\Qc^\ast(x'),\mathcal A_\sigma(y')\bigr)
	\,\mathrm dx'\,\mathrm dy'.
\end{align*}
Combining the complete-cell and boundary-strip contributions and
summing over $\sigma\in\{-1,1\}$, we obtain
\begin{equation}
	\label{eq:recovery-surface-unscaled-limit}
	\begin{aligned}
		&\sum_{\sigma\in\{-1,1\}}
		\int_\omega
		W_\sigma\left(\frac{x'}{\varepsilon}\right)
		\operatorname{dist}^2
		\left(
		Q^{*,\sigma}_{h,\varepsilon}(x'),
		\mathcal A_\sigma\left(\frac{x'}{\varepsilon}\right)
		\right)
		\,\mathrm dx'
		\longrightarrow
		\int_\omega
		\overline f_s(\Qc^\ast(x'))
		\,\mathrm dx'.
	\end{aligned}
\end{equation}

If $\gamma=1$, \cref{eq:recovery-surface-scaled} and
\cref{eq:recovery-surface-unscaled-limit} give
\[
\frac1h
E_{\mathrm{surf}}^{h,\varepsilon}
(Q^\ast_{h,\varepsilon})
\longrightarrow
\int_\omega
\overline f_s(\Qc^\ast(x'))
\,\mathrm dx'
=
E^1_{\mathrm{surf}}(\Qc^\ast).
\]
If $\gamma>1$, the integral on the right-hand side of
\cref{eq:recovery-surface-scaled} is uniformly bounded. Hence, the surface energy limit vanishes.
%\[
%\frac1h
%E_{\mathrm{surf}}^{h,\varepsilon}
%(Q^\ast_{h,\varepsilon})
%\leq
%Ch^{\gamma-1}
%\longrightarrow0
%=
%E^\gamma_{\mathrm{surf}}(\Qc^\ast).
%\]

\textbf{Surface term ($0\leq\gamma<1$).}
Recall that, by Assumption~\ref{asm:compatible}, the sets $\mathscr A_\sigma
:=
\mathscr A\left(\frac{\sigma}{2},\sigma e_3\right),
\qquad \sigma\in\{-1,1\}$,
are independent of $y'$, and $\mathscr A=\mathscr A_+\cap\mathscr A_-$.
Since $Q^\ast(x')\in\mathscr A$ for every $x'\in\overline\omega$, we have
\[
Q^\ast(x')\in\mathscr A_\sigma
\qquad
\text{for every }x'\in\overline\omega
\text{ and }\sigma\in\{-1,1\}.
\]
Hence, by the $1$-Lipschitz property of the distance function,
\[
\operatorname{dist}
\left(
Q^\ast_{h,\varepsilon,\sigma}(x'),
\mathscr A_\sigma
\right)
\leq
\left|
Q^\ast_{h,\varepsilon,\sigma}(x')-Q^\ast(x')
\right|,\quad\text{where}\quad
Q^\ast_{h,\varepsilon,\sigma}(x')
:=
Q^\ast_{h,\varepsilon}\left(x',\frac{\sigma h}{2}\right).
\]

Suppose first that $\rho\in[0,\infty)$. From \eqref{eq:recovery-rho},
\[
\begin{aligned}
	Q^\ast_{h,\varepsilon,\sigma}(x')-Q^\ast(x')
	={}&
	h\,\widehat R^\ast
	\left(
	x',\frac{x'}{\varepsilon},\frac{\sigma}{2}
	\right)
	+
	\varepsilon\,\widehat Q^\ast
	\left(
	x',\frac{x'}{\varepsilon}
	\right).
\end{aligned}
\]
The smoothness and periodicity of the correctors therefore give
\[
\sup_{x'\in\omega}
\left|
Q^\ast_{h,\varepsilon,\sigma}(x')-Q^\ast(x')
\right|
\leq C(h+\varepsilon),
\qquad \sigma\in\{-1,1\}.
\]
Consequently,
\[
\begin{aligned}
	\frac{1}{h}
	E^{\rm surf}_{h,\varepsilon}
	(Q^\ast_{h,\varepsilon})
	&=
	h^{\gamma-1}
	\sum_{\sigma\in\{-1,1\}}
	\int_\omega
	W_\sigma\left(\frac{x'}{\varepsilon}\right)
	\operatorname{dist}^2
	\left(
	Q^\ast_{h,\varepsilon,\sigma}(x'),
	\mathscr A_\sigma
	\right)
	\,dx'
	\leq
	C\left(
	h^{\gamma+1}
	+
	\frac{\varepsilon^2}{h^{1-\gamma}}
	\right).
\end{aligned}
\]
If $\rho\in(0,\infty)$, then
$\varepsilon\leq Ch$ for all sufficiently small $h$ and $\varepsilon$, and If $\rho=0$, with the additional scale condition \cref{eq:subcritical-scale}, we obtain
\[
\frac{1}{h}
E^{\rm surf}_{h,\varepsilon}
(Q^\ast_{h,\varepsilon})
\longrightarrow0.
\]
For $\rho=\infty$ ($\varepsilon/h\to0$), from \eqref{eq:recovery-infty},
%\[
%\begin{aligned}
%	Q^\ast_{h,\varepsilon,\sigma}(x')-Q^\ast(x')
%	={}&
%	h\,\widehat R^\ast
%	\left(x',\frac{\sigma}{2}\right)
%	+
%	\varepsilon\,\widehat Q^\ast_\infty
%	\left(
%	x',\frac{x'}{\varepsilon},\frac{\sigma}{2}
%	\right).
%\end{aligned}
%\]
%Since $\varepsilon/h\to0$, 
the smoothness of the correctors yields
\[
\sup_{x'\in\omega}
\left|
Q^\ast_{h,\varepsilon,\sigma}(x')-Q^\ast(x')
\right|
\leq C(h+\varepsilon)\leq Ch
\]
for all sufficiently small $h$ and $\varepsilon$. Therefore,
\[
\frac{1}{h}
E^{\rm surf}_{h,\varepsilon}
(Q^\ast_{h,\varepsilon})
\leq
C h^{\gamma-1}h^2
=
C h^{\gamma+1}
\longrightarrow0.
\]
Thus, in every regime $\rho\in[0,\infty]$, the rescaled surface enrgy vanishes.

\noindent\textbf{Conclusion.}
Combining the convergence of the elastic, bulk, and surface terms, we obtain
\[
\frac1h
\Fhe(Q^\ast_{h,\varepsilon})
\longrightarrow
E^\rho_{\mathrm{el}}
(\Qc^\ast,\wh Q^\ast_\rho,\wh R^\ast)
+
E_{\mathrm{bulk}}(\Qc^\ast)
+
E^\gamma_{\mathrm{surf}}(\Qc^\ast)
=
\GF^\gamma_\rho
(\Qc^\ast,\wh Q^\ast_\rho,\wh R^\ast).
\]
Since $\Gm_{h,\varepsilon}
\leq
\Fhe(Q^\ast_{h,\varepsilon})$,
we obtain~\cref{eq:limsupR}, which completes the proof.
%
%\noindent\textbf{Conclusion.}
%Summing the three contributions,
%\[
%\frac{1}{h}\Fhe(Q^*_{h,\varepsilon})
%\;\longrightarrow\;
%E^\rho_{\mathrm{el}}(\Qc^*,\widehat{Q}^*_\rho,\widehat{R}^*)
%+ E_{\mathrm{bulk}}(\Qc^*)
%+ E^\gamma_{\mathrm{surf}}(\Qc^*)
%= \GF^\gamma_\rho(\Qc^*,\widehat{Q}^*_\rho,\widehat{R}^*).
%\]
%Since $\Gm_{h,\e}\leq \Fhe(Q^\ast_{h,\e})$, we obtain \cref{eq:limsupR}. This completes the proof.
\end{proof}
		
		\section{Proof of the main results}
	\label{sec:proofs}
	
		We now collect the compactness, lower-bound, recovery-sequence, and minimisation
	consequences from Lemmas \ref{lem:71}--\ref{lem:trace-bound}, Proposition \ref{prop:liminf}, and Theorem \ref{thm:reco} into a single frame to present the proof of the main sequential $\Gamma$-convergence statement of Theorem \ref{thm:gamma}. 
	
	\subsection{Derivation of the two-scale limit energy}		
%	\paragraph{Sequential $\Gamma$-convergence.}
%	\begin{proof}[Proof of Theorem \ref{thm:gamma}]
%		The proof is a direct consequence of 
%		
%			 \textbf{(i) Compactness.} Theorem \ref{prop:existence} and Lemmas \ref{lem:71}--\ref{lem:trace-bound}.
%			 
%			\textbf{(ii) Liminf inequality.} Proposition \ref{prop:liminf}.
%			
%			\textbf{(iii) Limsup inequality.} Theorem \ref{thm:reco} with the strong convergences \cref{th:recM}, more specifically \cref{eq:strong-grad}, \cref{Lims1}--\cref{Lims2}.
%			
%			\textbf{(iv) Convergence of minima and minimisers.} 	We show that	\begin{equation*}
%				\Gm^\gamma_\rho= \GF^\gamma_\rho(\Qc,\wh Q_\rho,\wh R)= \lim_{(h,\e)\to(0,0)}{\mhe\over h}=\lim_{(h,\e)\to(0,0)}{\Fhe(Q_{h,\e})\over h},
%			\end{equation*}
%			where $(\Qc,\wh Q_\rho,\wh R)\in\A^\gamma_\rho$ is a minimiser of the two-scale limit energy.
%			
%
%		Hence
%		\[
%		\GF^\gamma_\rho(\Qc,\wh Q_\rho,\wh R)
%		=\Gm^\gamma_\rho,
%		\]
%		so $(\Qc,\wh Q_\rho,\wh R)\in\A^\gamma_\rho$ is a minimiser of the two-scale limit energy $\GF^\gamma_\rho$. This completes the proof.
%	\end{proof}
	\paragraph{Sequential $\Gamma$-convergence.}
	
	\begin{proof}[Proof of Theorem~\ref{thm:gamma}]
		The sequential $\Gamma$-convergence follows from:
		\begin{enumerate}
			\item[(i)] the compactness result given by
			Lemmas~\ref{lem:71}--\ref{lem:trace-bound};
			
			\item[(ii)] the liminf inequality established in
			Proposition~\ref{prop:liminf};
			
			\item[(iii)] the recovery-sequence construction of
			Theorem~\ref{thm:reco}, together with the strong convergences
			in~\cref{th:recM}, and in particular
			\cref{eq:strong-grad,Lims1,Lims2}.
		\end{enumerate}
					The proof proceeds using a liminf-limsup argument.
		
		\noindent\textbf{Lower bound.}
		The lower bound
		\begin{equation}\label{Ginf}
			\Gm^\gamma_\rho\leq \GF^\gamma_\rho(\Qc,\wh Q_\rho,\wh R)\leq \liminf_{(h,\e)\to(0,0)}{\mhe\over h}=\liminf_{(h,\e)\to(0,0)}{\Fhe(Q_{h,\e})\over h}.
		\end{equation}
		is the content of Proposition~\ref{prop:liminf}.
		
		\noindent\textbf{Upper bound.}
		Let $(\Qc^*, \widehat{Q}^*_\rho, \widehat{R}^*) \in \A^\gamma_\rho$ be a minimiser of 
		$\GF^\gamma_\rho$, which exists by Theorem~\ref{thm:existence-limit}. Then, using Theorem \ref{thm:reco}, there exists a sequence $\{Q^\ast_{h,\e}\}_{h,\e}\subset H^1(\O_h,\Szero)$ such that
		\[
		\limsup_{(h,\varepsilon)\to(0,0)}\frac{\Gm_{h,\varepsilon}}{h}
		\;\leq\;
		\lim_{(h,\varepsilon)\to(0,0)}\frac{1}{h}\Fhe(Q^*_{h,\varepsilon})
		= \GF^\gamma_\rho(\Qc^*,\widehat{Q}^*_\rho,\widehat{R}^*)
		= \Gm^\gamma_\rho.
		\]
		Combined with the lower bound~\cref{Ginf}, this gives 
		$\lim_{(h,\varepsilon)\to(0,0)} \Gm_{h,\varepsilon}/h = \Gm^\gamma_\rho$, 
		and that $(\Qc,\wh Q_\rho,\wh R)\in\A^\gamma_\rho$ is a minimiser of $\GF^\gamma_\rho$.
		
		Finally, let $Q_{h,\varepsilon}$ be any sequence of minimisers and suppose that
		\[
		Q_{h,\varepsilon}\rightharpoonup^{\mathrm{Qhom}}(\Qc,\wh Q_\rho,\wh R)\in\A^\gamma_\rho
		\]
		along a subsequence. By parts \emph{(ii)--(iv)}, we obtain
		\[
		\GF^\gamma_\rho(\Qc,\wh Q_\rho,\wh R)
		\leq
		\liminf_{(h,\varepsilon)\to(0,0)}
		\frac{\Gm_{h,\varepsilon}}{h}
		\leq \Gm^\gamma_\rho,\quad\text{which imply}\quad
		\GF^\gamma_\rho(\Qc,\wh Q_\rho,\wh R)
		=\Gm^\gamma_\rho,
		\]
		so $(\Qc,\wh Q_\rho,\wh R)\in\A^\gamma_\rho$ is a minimiser of the two-scale limit energy $\GF^\gamma_\rho$. This completes the proof.
	\end{proof}
	
	\paragraph{Separate energy convergence and strong convergence of minimisers.}
	\begin{proof}[Proof of Theorem~\ref{cor:strong-conv-minimisers}]
		 We have 	the sequence $\{Q_{h,\varepsilon}\}_{h,\e}\subset H^1(\Omega_h;\Szero)$ satisfy
		\[
		Q_{h,\varepsilon}\rightharpoonup^{\mathrm{Qhom}}(\Qc,\wh Q_\rho,\wh R),\quad \text{and}\quad
		\lim_{(h,\varepsilon)\to(0,0)}
		\frac1h\GF_{h,\varepsilon}(Q_{h,\varepsilon})
		=
		\GF^\gamma_\rho(\Qc,\wh Q_\rho,\wh R),
		\]
		with $(\Qc,\hQ_\rho,\hR)\in\A^\gamma_\rho$.
		That is,
		\begin{equation*}
			\lim_{(h,\varepsilon)\to(0,0)}
			\Bigg[
			\frac1h E^{h,\varepsilon}_{\el}(Q_{h,\varepsilon})
			+
			\frac1h E^{h,\e}_{\bulk}(Q_{h,\varepsilon})
			+
			\frac1h E^{h,\varepsilon}_{\surf}(Q_{h,\varepsilon})
			\Bigg]
			=
			E^\rho_{\el}(\Qc,\hQ_\rho,\hR)
			+
			\Ebulk(\Qc)
			+
			E^\gamma_\surf(\Qc).
		\end{equation*}
		On the other hand, the separate lower-bound results for the elastic, bulk,
		and surface contributions give (see Proposition \ref{prop:liminf}) the liminf inequalities \cref{eq:liminfE}, \cref{eq:liminfS} and \cref{Sur01}.
		Since the sum of the three sequences converges to the sum of the three
		lower bounds, each contribution must converge separately. This proves
		\cref{eq:elastic-conv-min}--\cref{eq:surf-conv-min}.
		
		We now prove the strong convergence \cref{eq:strong-grad-min} along the original sequence of Q-tensors $\{Q_{h,\varepsilon}\}$.
		
		From Lemma~\ref{prop:admissible-correctors}, we already have 
		$\Pi_{h,\varepsilon}(\nabla Q_{h,\varepsilon}) \rightharpoonup \Gc_\rho$ weakly in 
		$L^2(\omega \times Y; \Szero \otimes \mathbb{R}^3)$.
		By the elastic term convergence \cref{eq:elastic-conv-min}, we have
		\[
		\int_{\omega \times Y} \bA(y)\,\Pi_{h,\varepsilon}(\nabla Q_{h,\varepsilon}) 
		: \Pi_{h,\varepsilon}(\nabla Q_{h,\varepsilon})\,dy\,dx'
		\;\longrightarrow\; 
		\int_{\omega \times Y} \bA(y)\,\Gc_\rho : \Gc_\rho\,dy\,dx'.
		\]
		Expanding the quadratic form using uniform ellipticity and the symmetry assumption~\ref{asm:elastic}:
		\begin{align*}
			C\,\|\Pi_{h,\varepsilon}(\nabla Q_{h,\varepsilon}) - \Gc_\rho\|^2_{L^2(\omega\times Y)}
			&\leq 
			\int_{\omega\times Y} \bA(y)\,\bigl(\Pi_{h,\varepsilon}(\nabla Q_{h,\varepsilon}) - \Gc_\rho\bigr) 
			: \bigl(\Pi_{h,\varepsilon}(\nabla Q_{h,\varepsilon}) - \Gc_\rho\bigr)\,dy\,dx' \\
			&= \int_{\omega\times Y} \bA(y)\,\Pi_{h,\varepsilon}(\nabla Q_{h,\varepsilon}) 
			: \Pi_{h,\varepsilon}(\nabla Q_{h,\varepsilon})\,dy\,dx' \\
			&\quad - 2\int_{\omega\times Y} \bA(y)\,\Pi_{h,\varepsilon}(\nabla Q_{h,\varepsilon}) 
			: \Gc_\rho\,dy\,dx' 
			+ \int_{\omega\times Y} \bA(y)\,\Gc_\rho : \Gc_\rho\,dy\,dx'.
		\end{align*}
		The first integral converges to $\int \bA\,\Gc_\rho:\Gc_\rho$; the second converges to 
		$2\int \bA\,\Gc_\rho:\Gc_\rho$ by weak convergence of $\Pi_{h,\varepsilon}(\nabla Q_{h,\varepsilon})$ 
		and the fact that $\bA(\cdot)\Gc_\rho \in L^2(\omega \times Y; \Szero\otimes\mathbb{R}^3)$ 
		is an admissible test function. The right-hand side therefore tends to zero, giving the strong convergence \cref{eq:strong-grad-min}.
		
		It remains to prove the strong $L^4$ convergence. By compactness, we already have the strong convergence \cref{con55} in $L^p$ with $2\leq p<4$  and the weak convergence \cref{con56} in $L^4$.
		The bulk energy convergence \cref{eq:bulk-conv-min} gives
		\[
		\int_{\omega\times Y}
		f_b(\Pi_{h,\varepsilon}Q_{h,\varepsilon})\,dy\,dx'
		\to
		\int_{\omega\times Y}
		f_b(\Qc)\,dy\,dx'.
		\]
		Using the Landau--de Gennes form \cref{eq:fb} and the convergence \cref{con55}, we have that the quadratic and cubic terms converge, which, combined the convergence of the bulk energy \cref{eq:bulk-conv-min}, gives
		\[
		\int_{\omega\times Y}
		|\Pi_{h,\varepsilon}Q_{h,\varepsilon}|^4\,dy\,dx'
		\to
		\int_{\omega\times Y}
		|\Qc|^4\,dy\,dx'.
		\]
		The convergence \cref{con56} implies
		\[
		\|\Pi_{h,\varepsilon}Q_{h,\varepsilon}\|_{L^4(\omega\times Y)}
		\to
		\|\Qc\|_{L^4(\omega\times Y)}.
		\]
		The uniform convexity of $L^4$ therefore implies the strong convergence \cref{eq:strong-L4-min}.
	\end{proof}
	
	%\section{Convergence of minima and minimisers}
	
%	\section{Derivation of the Homogenised Energy}
%	\label{sec:Gamma}
%	In this section, we derive the homogenised energy and prove the existence of minimisers when $\gamma\geq 1$ and $\rho\in[0,\infty]$. To this end, we first prove the convergence of the energy using a form of $\Gamma$-convergence, and then we give the expression of the homogenised energy using correctors and cell problems.
	
%	\subsection{Energy convergence along minimisers}
%	Below is the main energy convergence result for $\gamma\geq 1$ and $\rho\in[0,\infty]$.
%	\begin{theorem}\label{thm:mainC}
%		Under Assumptions~\ref{asm:elastic}, \ref{asm:bulk}, and \ref{asm:anchoring} with
%		$\gamma\geq 1$, for each $\rho\in[0,\infty]$, we have, up to a subsequence
%	\end{theorem}
%	\begin{proof}
%		
%		This completes the proof.
%	\end{proof}

	\subsection{Derivation of the homogenised energy}
	\label{sec:hom-energy}
	
	Having established the full energy convergence in Theorem~\ref{thm:gamma}, we now identify the
	effective two-dimensional energy. The bulk and surface terms depend only on the macroscopic
	field, while the elastic term still contains the microscopic correctors. Therefore the derivation of
	the homogenised energy reduces to a cell-wise minimisation of the two-scale elastic energy.

	{\bf Step 1: Variational derivation of the cell problems}
	
	Fix $\Qc\in H^1(\omega;\mathscr{E}_\gamma)$, where $\mathscr{E}_\gamma$ is defined in Section \ref{sec:main-results}. Since $\Ebulk(\Qc)$ and $E^\gamma_\surf(\Qc)$ are independent
	of the microscopic $Q$-tensors, minimising $\GF^\gamma_\rho$ over $\A^\gamma_\rho$ with $\Qc$
	fixed is equivalent to
	\begin{equation}
		\label{eq:micro-min}
		\inf_{(\hQ_\rho,\,\hR)\,\in\, \C_\rho}
		E^\rho_\el(\Qc,\hQ_\rho,\hR),
		\qquad
		E^\rho_\el(\Qc,\hQ_\rho,\hR)
		\;=\;
		\frac12\int_{\omega\times Y} \bA(y)\,\Gc_\rho : \Gc_\rho\;dy\,dx',
	\end{equation}
	where $\Gc_\rho$ is defined from $(\Qc,\hQ_\rho,\hR)$ as in
	Lemma~\ref{prop:admissible-correctors}. By the uniform ellipticity of $\bA$
	and the Poincare-type coercivity estimate of Lemma~\ref{lem:coercivity}, the functional
	$E^\rho_\el(\Qc,\cdot,\cdot)$ is coercive on $\C_\rho$. Since it is quadratic and strictly
	convex on the quotient fixed by the imposed zero-mean normalisations, it has a unique minimiser.
	
	Let $(\hQ_\rho,\hR)\in \C_\rho$ be the minimiser of \cref{eq:micro-min}. Since $\C_\rho$ is a
	linear subspace of its ambient Hilbert space, for every perturbation direction
	$(\hQ^\dagger_\rho,\hR^\dagger)\in \C_\rho$ and every $t\in\mathbb{R}$, the pair
	$(\hQ_\rho+t\hQ^\dagger_\rho,\,\hR+t\hR^\dagger)$ remains admissible in $\C_\rho$. By minimality,
	\begin{equation}
		\label{eq:minimality-ineq}
		\GF^\gamma_\rho\bigl(\Qc,\,\hQ_\rho+t\hQ^\dagger_\rho,\,\hR+t\hR^\dagger\bigr)
		\;\geq\;
		\GF^\gamma_\rho(\Qc,\hQ_\rho,\hR)
		\qquad\forall\,t\in\mathbb{R}.
	\end{equation}
	Denote by $\Gc^\dagger_\rho$ the two-scale gradient corresponding to the perturbation direction
	$(0,\hQ^\dagger_\rho,\hR^\dagger)$ with zero macroscopic loading. Since $\Gc_\rho$ depends
	affine-linearly on $(\Qc,\hQ_\rho,\hR)$, the perturbed gradient decomposes as
	\[
	\Gc_\rho\bigl[\Qc,\hQ_\rho+t\hQ^\dagger_\rho,\hR+t\hR^\dagger\bigr]
	\;=\; \Gc_\rho + t\,\Gc^\dagger_\rho .
	\]
	Expanding the quadratic elastic energy gives
	\begin{equation}
		\label{eq:quad-expand}
		E^\rho_\el(\Qc,\hQ_\rho+t\hQ^\dagger_\rho,\hR+t\hR^\dagger)
		=
		E^\rho_\el(\Qc,\hQ_\rho,\hR)
		+
		t\!\int_{\omega\times Y}\!\bA(y)\,\Gc_\rho:\Gc^\dagger_\rho\;dy\,dx'
		+
		t^2 E^\rho_\el(0,\hQ^\dagger_\rho,\hR^\dagger).
	\end{equation}
	Since the bulk and surface terms are unaffected by the perturbation, substituting
	\cref{eq:quad-expand} into \cref{eq:minimality-ineq} gives
	\begin{equation}
		\label{eq:perturb-inequality}
		t\!\int_{\omega\times Y}\!\bA(y)\,\Gc_\rho:\Gc^\dagger_\rho\;dy\,dx'
		+
		t^2 E^\rho_\el(0,\hQ^\dagger_\rho,\hR^\dagger)
		\;\geq\;0
		\qquad\forall\,t\in\mathbb{R}.
	\end{equation}
	Dividing by $t>0$, sending $t\to0^+$, and repeating with $t<0$ and $t\to0^-$ yields
	the Euler--Lagrange equation
	\begin{equation}
		\label{eq:EL-global}
		\int_{\omega\times Y} \bA(y)\,\Gc_\rho:\Gc^\dagger_\rho\;dy\,dx' \;=\; 0
		\qquad\forall\,(\hQ^\dagger_\rho,\hR^\dagger)\in \C_\rho.
	\end{equation}
	By the density of tensor-product perturbations in the admissible spaces and a standard localisation
	argument in $x'$, we obtain the corresponding cell equations: for a.e.\ $x'\in\omega$,
	\begin{equation}
		\label{eq:EL-local}
		\int_Y \bA(y)\,\Gc_\rho(x',y):\Gc^\dagger_\rho(y)\;dy \;=\; 0
		\qquad\forall\;\text{admissible cell-level perturbation gradients }\Gc^\dagger_\rho .
	\end{equation}
	Thus the macroscopic loading
	\[
	M:=\nabla'\Qc(x')\in \Szero\otimes\mathbb{R}^2
	\]
	enters the cell problem as a fixed parameter. We now state the space of correctors and weak cell problems in the three
	regimes. 	
	
	For $\rho\in[0,\infty]$, define the cell-level corrector spaces $\C^0_\rho$ as follows:
	\begin{itemize}
		\item \textit{Regime $\rho=0$:}
		\[
		\C^0_0
		:=
		H^1_{\mathrm{per},0}(Y';\Szero)
		\times
		L^2(Y';H^1_\sharp(I;\Szero)).
		\]
		
		\item \textit{Regime $\rho\in(0,\infty)$:}
		\begin{equation*}
			\C^0_\rho
			:=		H^1_{\mathrm{per},0}(Y';\Szero)
			\times
			H^1_{\mathrm{per},\sharp}(Y;\Szero).
		\end{equation*}		
		\item \textit{Regime $\rho=\infty$:}
		\[
		\C^0_\infty
		:=
		L^2(I;H^1_{\mathrm{per},0}(Y';\Szero))
		\times
		H^1_\sharp(I;\Szero).
		\]
	\end{itemize}

	{\bf (R3) Regime $\rho=0$.}
	For a fixed loading $M\in\Szero\otimes\mathbb{R}^2$, define
	\[
	\Gc_0^M
	:=
	\bigl(M+\nabla_{y'}\hQ_0,\;\partial_{y_3}\hR\bigr).
	\]
	The cell problem is: find $(\hQ_0,\hR)\in\C^0_0$
	such that
	\begin{align}
		&\int_Y \bA(y)\,\Gc_0^M:\bigl(\nabla_{y'}\phi,\;0\bigr)\;dy
		=0
		\qquad
		\forall\,\phi\in H^1_{\mathrm{per},0}(Y';\Szero),
		\label{eq:CP0a}\tag{CP$0$a}\\[4pt]
		&\int_Y \bA(y)\,\Gc_0^M:\bigl(0,\;\partial_{y_3}\psi\bigr)\;dy
		=0
		\qquad
		\forall\,\psi\in L^2(Y';H^1_\sharp(I;\Szero)).
		\label{eq:CP0b}\tag{CP$0$b}
	\end{align}
	Equation~\cref{eq:CP0a} is the weak in-plane equilibrium equation averaged through the thickness.
	Equation~\cref{eq:CP0b} determines the transverse corrector by relaxation in the thickness
	direction. These weak equations are the cell problem used below; no classical strong form is needed
	under the present measurability assumptions on $\bA$.

	{\bf (R2) Regime $\rho\in(0,\infty)$.}
	For a fixed loading $M\in\Szero\otimes\mathbb{R}^2$, the independent microscopic variables are
	$\hQ$ and $\hR$, with $\hQ_\rho:=\hQ+\rho\hR$.
	Thus
	\[
	\Gc_\rho^M
	:=
	\bigl(M+\nabla_{y'}\hQ+\rho\nabla_{y'}\hR,\;\partial_{y_3}\hR\bigr).
	\]
	The cell problem is: find
	$(\hQ,\hR)\in\C^0_\rho$
	such that
	\begin{align}
		&\int_Y \bA(y)\,\Gc_\rho^M:\bigl(\nabla_{y'}\phi,\;0\bigr)\;dy
		=0
		\qquad
		\forall\,\phi\in H^1_{\mathrm{per},0}(Y';\Szero),
		\label{eq:CPrho-a}\tag{CP$\rho$a}\\[4pt]
		&\int_Y \bA(y)\,\Gc_\rho^M:
		\bigl(\rho\nabla_{y'}\psi,\;\partial_{y_3}\psi\bigr)\;dy
		=0
		\qquad
		\forall\,\psi\in H^1_{\mathrm{per},\sharp}(Y;\Szero).
		\label{eq:CPrho-b}\tag{CP$\rho$b}
	\end{align}
	Equation~\cref{eq:CPrho-a} is the weak equilibrium equation associated with the purely in-plane
	corrector $\hQ$. Equation~\cref{eq:CPrho-b} is the coupled weak equilibrium equation associated
	with the thickness corrector $\hR$. The coupling is caused by the term
	$\rho\nabla_{y'}\hR$ in the in-plane part of $\Gc_\rho^M$.

	{\bf (R1) Regime $\rho=\infty$.}
	For a fixed loading $M\in\Szero\otimes\mathbb{R}^2$, define
	\[
	\Gc_\infty^M
	:=
	\bigl(M+\nabla_{y'}\hQ_\infty,\;\partial_{y_3}\hR\bigr).
	\]
	The cell problem is: find
	$(\hQ_\infty,\hR)\in\C^0_\infty$
	where $\hR$ is $y'$-independent, such that
	\begin{align}
		&\int_{Y} \bA(y)\,\Gc_\infty^M:
		\bigl(\nabla_{y'}\eta,\;0\bigr)\;dy
		=0\quad
		\forall\,\eta\in L^2(I;H^1_{\mathrm{per},0}(Y';\Szero)),
		\label{eq:CPinf-a}\tag{CP$\infty$a}\\[4pt]
		&\int_Y \bA(y)\,\Gc_\infty^M:
		\bigl(0,\;\partial_{y_3}\psi\bigr)\;dy
		=0
		\qquad
		\forall\,\psi\in H^1_\sharp(I;\Szero).
		\label{eq:CPinf-b}\tag{CP$\infty$b}
	\end{align}
	Equation~\cref{eq:CPinf-a} is the slice-wise in-plane cell problem at fixed $y_3$. Equation
	\cref{eq:CPinf-b} determines the thickness corrector through the $Y'$-averaged transverse
	stress.

	{\bf Step 2: Homogenised elastic tensor}
	
	\begin{lemma}[Linearity of correctors]
		\label{lem:linearity-correctors}
		For each $\rho\in[0,\infty]$ and each $M\in\Szero\otimes\mathbb{R}^2$, the cell problems 
		in the corresponding regime admit a unique solution
		$(\chi_\rho(\cdot;M),\varPhi_\rho(\cdot;M))\in\C^0_\rho$.
		In the intermediate regime $\rho\in(0,\infty)$, $\chi_\rho$ denotes the independent
		in-plane corrector $\hQ$, and the corrector entering the in-plane gradient is
		$\chi_\rho+\rho\varPhi_\rho$. Moreover, the solution maps
		\[
		M\mapsto \chi_\rho(\cdot;M),
		\qquad
		M\mapsto \varPhi_\rho(\cdot;M)
		\]
		are linear.
	\end{lemma}
	
	\begin{proof}
		Uniqueness follows from the strict convexity and coercivity established in Step~1.
		For linearity, let $M_1,M_2\in\Szero\otimes\mathbb{R}^2$ and $\alpha,\beta\in\mathbb{R}$.
		Each cell problem is affine in the loading $M$: the corrector space $\C^0_\rho$ is
		independent of $M$, and $M$ enters only as a fixed additive term in the cell gradient
		$\Gc^M_\rho$. Therefore
		$(\alpha\chi_\rho(\cdot;M_1)+\beta\chi_\rho(\cdot;M_2),\,
		\alpha\varPhi_\rho(\cdot;M_1)+\beta\varPhi_\rho(\cdot;M_2))$
		solves the cell problem for loading $\alpha M_1+\beta M_2$, and by uniqueness it
		is the solution.
	\end{proof}
	
	\begin{definition}[Homogenised elastic tensor]
		\label{def:hom-tensor}
		For each $\rho\in[0,\infty]$ and $M\in\Szero\otimes\mathbb{R}^2$, define
		$\bA^{\mathrm{hom}}_\rho M:M$ by the cell formula
		\begin{equation}
			\label{eq:hom-tensor-def}
			\bA^{\mathrm{hom}}_\rho M:M
			:=
			\min_{(\wh{Q},\wh{R})\in\C^0_\rho}
			\int_Y \bA(y)\,\Gc^M_\rho:\Gc^M_\rho\,dy,
		\end{equation}
		where the minimising cell gradient $\Gc^M_\rho$ is given in each regime by
		\begin{alignat}{2}
			\Gc^M_0
			&:=
			\bigl(M+\nabla_{y'}\wh{Q}_0,\;\partial_{y_3}\wh{R}\bigr),
			&\qquad &\rho=0,
			\label{eq:cell-grad-0}\\
			\Gc^M_\rho
			&:=
			\bigl(M+\nabla_{y'}\wh{Q}+\rho\nabla_{y'}\wh{R},\;\partial_{y_3}\wh{R}\bigr),
			&\qquad &\rho\in(0,\infty),
			\label{eq:cell-grad-rho}\\
			\Gc^M_\infty
			&:=
			\bigl(M+\nabla_{y'}\wh{Q}_\infty,\;\partial_{y_3}\wh{R}\bigr),
			&\qquad &\rho=\infty.
			\label{eq:cell-grad-inf}
		\end{alignat}
		The minimum is attained at the corrector pair $(\chi_\rho(\cdot;M),\varPhi_\rho(\cdot;M))$
		from Lemma~\ref{lem:linearity-correctors}, and the associated symmetric bilinear form
		$\bA^{\mathrm{hom}}_\rho M:M'$ is obtained by polarisation.
	\end{definition}
	
	\begin{lemma}[Reduced formula for the homogenised tensor]
		\label{lem:hom-tensor-reduced}
		For each $\rho\in[0,\infty]$ and $M\in\Szero\otimes\mathbb{R}^2$,
		\begin{equation}
			\label{eq:hom-tensor-reduced}
			\bA^{\mathrm{hom}}_\rho M:M
			=
			\int_Y \bA(y)\,\Gc^M_\rho:(M,0)\,dy,
		\end{equation}
		where $\Gc^M_\rho$ denotes the minimising cell gradient evaluated at the corrector
		$(\chi_\rho(\cdot;M),\varPhi_\rho(\cdot;M))$.
	\end{lemma}
	
	\begin{proof}
		Testing the weak cell problem with the corrector itself gives the Galerkin orthogonality
		\begin{equation}
			\label{eq:Galerkin}
			\int_Y \bA(y)\,\Gc^M_\rho:\bigl(\Gc^M_\rho-(M,0)\bigr)\,dy = 0.
		\end{equation}
		Rearranging,
		\[
		\int_Y \bA(y)\,\Gc^M_\rho:\Gc^M_\rho\,dy
		=
		\int_Y \bA(y)\,\Gc^M_\rho:(M,0)\,dy,
		\]
		and the left-hand side equals $\bA^{\mathrm{hom}}_\rho M:M$ by definition.
	\end{proof}

	\begin{proposition}
		\label{prop:hom-props}
		For each $\rho\in[0,\infty]$, the tensor
		$\bA^{\mathrm{hom}}_\rho\in\mathcal{L}(\Szero\otimes\mathbb{R}^2,\Szero\otimes\mathbb{R}^2)$
		is symmetric and satisfies
		\begin{equation}
			\label{eq:hom-bounds}
			\lambda \,|M|^2
			\leq
			\bA^{\mathrm{hom}}_\rho M:M
			\leq
			\Lambda\,|M|^2
			\qquad
			\forall\,M\in\Szero\otimes\mathbb{R}^2,
		\end{equation}
		with $\lambda,\Lambda>0$ the constants from Assumption~\ref{asm:elastic}\ref{A3}.
	\end{proposition}
	
	\begin{proof}
		\textbf{Upper bound.}
		Testing the cell formula with the zero corrector pair $(\wh{Q},\wh{R})=(0,0)$,
		for which $G = (M,0)$, gives
		\[
		\bA^{\mathrm{hom}}_\rho M:M
		\leq \int_Y \bA(y)(M,0):(M,0)\,dy \leq \Lambda|M|^2.
		\]
		
		\textbf{Lower bound.}
		In all three regimes the minimising gradient $\Gc^M_\rho$ defined in \cref{eq:cell-grad-0}--\cref{eq:cell-grad-inf} satisfy the following $L^2$-norm identity using Lemma \ref{lem:coercivity}
		\[
		\|\Gc^M_\rho\|^2_{L^2(Y)}
		= |M|^2 + \|\nabla_{y'}\wh Q_\rho\|^2_{L^2(Y)} + \|\partial_{y_3}\wh R_\rho\|^2_{L^2(Y)}
		\geq |M|^2.
		\]
		By uniform ellipticity and the fact that the minimum is attained:
		\[
		\bA^{\mathrm{hom}}_\rho M:M
		= \int_Y \bA(y)\,\Gc^M_\rho : \Gc^M_\rho\,dy
		\geq \lambda\,\|\Gc^M_\rho\|^2_{L^2(Y)}
		\geq \lambda\,|M|^2.
		\]
		
		\textbf{Symmetry.}
		Polarisation of $M \mapsto \bA^{\mathrm{hom}}_\rho M:M$ defines a bilinear form.
		Using the Galerkin orthogonality~\cref{eq:Galerkin}, one identifies
		\[
		\bA^{\mathrm{hom}}_\rho M : N
		= \int_Y \bA(y)\,G^M_\rho : (N,0)\,dy.
		\]
		The major symmetry of $\bA$ (Assumption~\ref{asm:elastic}\ref{A2} then gives
		$\bA^{\mathrm{hom}}_\rho M:N = \bA^{\mathrm{hom}}_\rho N:M$.
	\end{proof}

	{\bf Step 3: Homogenised energy}
	
	\begin{proof}[Proof of Theorem \ref{thm:hom-energy}]
		The functional $\GF^{\mathrm{hom},\gamma}_\rho$ is coercive and weakly lower semicontinuous
		on $H^1(\omega;\mathscr{E}_\gamma)$. Indeed, coercivity of the elastic term follows from
		Proposition~\ref{prop:hom-props}, coercivity of the bulk term in $L^4$ follows from
		Lemma~\ref{lem:fB}, and the surface term is continuous along strongly convergent sequences
		in $L^2(\omega;\Szero)$ by the quadratic growth estimate \cref{eq:LimG}. Since
		$H^1(\omega;\mathscr{E}_\gamma)\hookrightarrow\hookrightarrow L^2(\omega;\mathscr{E}_\gamma)$, the direct method
		gives a minimiser of $\GF^{\mathrm{hom},\gamma}_\rho$.
		
		By Theorem~\ref{thm:gamma}, the rescaled infima converge and, up to a subsequence, the
		microscopic minimisers $\{Q_{h,\varepsilon}\}$ generate a limit
		$(\Qc,\hQ_\rho,\hR)\in\A^\gamma_\rho$ which minimises the two-scale limit energy
		$\GF^\gamma_\rho$ over $\A^\gamma_\rho$. Thus
		\begin{equation}
			\label{eq:twoscale-lim}
			\lim_{(h,\varepsilon)\to(0,0)}
			\frac{\Gm_{h,\varepsilon}}{h}
			=
			\GF^\gamma_\rho(\Qc,\hQ_\rho,\hR)
			=
			\min_{(\Qc^\dagger,\hQ^\dagger_\rho,\hR^\dagger)\in\A^\gamma_\rho}
			\GF^\gamma_\rho(\Qc^\dagger,\hQ^\dagger_\rho,\hR^\dagger).
		\end{equation}
		
		Since the bulk and surface terms are independent of the correctors, the minimisation
		separates as
		\[
		\min_{(\Qc^\dagger,\hQ^\dagger_\rho,\hR^\dagger)\in\A^\gamma_\rho}
		\GF^\gamma_\rho(\Qc^\dagger,\hQ^\dagger_\rho,\hR^\dagger)
		=
		\min_{\Qc^\dagger\in H^1(\omega;\mathscr{E}_\gamma)}
		\left[
		\min_{(\hQ^\dagger_\rho,\hR^\dagger)\in\C_\rho}
		E^\rho_{\el}(\Qc^\dagger,\hQ^\dagger_\rho,\hR^\dagger)
		+
		\Ebulk(\Qc^\dagger)
		+
		E^\gamma_\surf(\Qc^\dagger)
		\right].
		\]
		For fixed $\Qc^\dagger$, the inner minimisation is local in $x'$ and, by
		Definition~\ref{def:hom-tensor}, gives
		\begin{equation}
			\label{eq:inner-min-hom}
			\min_{(\hQ^\dagger_\rho,\hR^\dagger)\in\C_\rho}
			E^\rho_{\el}(\Qc^\dagger,\hQ^\dagger_\rho,\hR^\dagger)
			=
			\frac12
			\int_\omega
			\bA^{\mathrm{hom}}_\rho\nabla'\Qc^\dagger:
			\nabla'\Qc^\dagger\,dx'.
		\end{equation}
		Therefore
		\begin{equation}
			\label{eq:twoscale-to-hom-min}
			\min_{(\Qc^\dagger,\hQ^\dagger_\rho,\hR^\dagger)\in\A^\gamma_\rho}
			\GF^\gamma_\rho(\Qc^\dagger,\hQ^\dagger_\rho,\hR^\dagger)
			=
			\min_{\Qc^\dagger\in H^1(\omega;\mathscr{E}_\gamma)}
			\GF^{\mathrm{hom},\gamma}_\rho(\Qc^\dagger).
		\end{equation}
		
		It remains to identify the value of the two-scale minimiser
		$(\Qc,\hQ_\rho,\hR)$ with the homogenised energy evaluated at $\Qc$. Since
		$(\Qc,\hQ_\rho,\hR)$ is a global minimiser of $\GF^\gamma_\rho$ over $\A^\gamma_\rho$, the pair
		$(\hQ_\rho,\hR)$ must minimise the elastic corrector problem for the fixed macroscopic field
		$\Qc$. Otherwise, replacing only the correctors while keeping $\Qc$ fixed would strictly
		decrease $\GF^\gamma_\rho$, because $\Ebulk(\Qc)$ and $E^\gamma_\surf(\Qc)$ do not depend
		on the correctors. Hence, using \cref{eq:inner-min-hom} with
		$\Qc^\dagger=\Qc$, we obtain
		\[
		E^\rho_{\el}(\Qc,\hQ_\rho,\hR)
		=
		\frac12
		\int_\omega
		\bA^{\mathrm{hom}}_\rho\nabla'\Qc:\nabla'\Qc\,dx'.
		\]
		Consequently,
		\[
		\GF^\gamma_\rho(\Qc,\hQ_\rho,\hR)
		=
		\GF^{\mathrm{hom},\gamma}_\rho(\Qc).
		\]
		
		Combining this identity with \cref{eq:twoscale-lim} and
		\cref{eq:twoscale-to-hom-min} gives
		\[
		\lim_{(h,\varepsilon)\to(0,0)}
		\frac{\Gm_{h,\varepsilon}}{h}
		=
		\GF^\gamma_\rho(\Qc,\hQ_\rho,\hR)
		=
		\GF^{\mathrm{hom},\gamma}_\rho(\Qc)
		=
		\min_{\Qc^\dagger\in H^1(\omega;\mathscr{E}_\gamma)}
		\GF^{\mathrm{hom},\gamma}_\rho(\Qc^\dagger).
		\]
		In particular, $\Qc$ is a minimiser of the homogenised energy. This completes the proof.
	\end{proof}

	% -------------------------------------------------------
	%  Continuity of the homogenised elastic tensor in rho
	% -------------------------------------------------------
\textbf{Continuity of the homogenised elastic tensor in $\rho$.}

\begin{proof}[Proof of Theorem \ref{thm:hom-tensor-continuity}]
	Fix $M \in \Szero \otimes \mathbb{R}^2$ throughout.
	Let $\{\rho_n\}_{n \geq 1} \subset [0,\infty]$ with $\rho_n \to \rho \in [0,\infty]$,
	and write $F_{\rho'} := \bA^{\mathrm{hom}}_{\rho'} M:M$ for every $\rho' \in [0,\infty]$.
	Recall from~\cref{eq:hom-tensor-reduced} that
	\begin{equation}
		\label{eq:Fform}
		F_{\rho'} = \int_Y \bA(y)\, \Gc^M_{\rho'} : (M,0)\,dy,
	\end{equation}
	where $\Gc^M_{\rho'}$ is the minimising cell gradient for parameter $\rho'$.

	\textbf{Step 1: Uniform bound and weak extraction.}
	Testing the cell minimisation with the zero corrector pair $(\hQ,\hR)=(0,0)$
	gives $F_{\rho_n} \leq \Lambda|M|^2$ uniformly in $n$,
	where $\Lambda := \|\bA\|_{L^\infty}$.
	Combined with the Galerkin orthogonality~\cref{eq:Galerkin} and the
	uniform ellipticity of $\bA$ (constant $\lambda > 0$), this yields (see \cref{eq:hom-bounds})
	\begin{equation}
		\label{eq:Gbound}
		\sup_{n}\, \|\Gc^M_{\rho_n}\|_{L^2(Y)}^{\,2}
		\;\leq\;
		\frac{\Lambda}{\lambda}\,|M|^2.
	\end{equation}
	After passing to a subsequence (not relabelled), there exists
	$\Gc^\ast \in L^2(Y;\Szero\otimes\mathbb{R}^3)$ such that
	\begin{equation}
		\label{eq:weak-conv}
		\Gc^M_{\rho_n} \;\rightharpoonup\; \Gc^*
		\quad\text{weakly in }L^2(Y).
	\end{equation}
	Once we identify $\Gc^\ast = \Gc^M_\rho$ in each regime, the conclusion
	\[
	F_{\rho_n}
	= \int_Y \bA(y)\,\Gc^M_{\rho_n}:(M,0)\,dy
	\;\xrightarrow{n\to\infty}\;
	\int_Y \bA(y)\,\Gc^M_\rho:(M,0)\,dy
	= F_\rho
	\]
	follows from~\cref{eq:Fform}--\cref{eq:weak-conv}, since
	$G \mapsto \int_Y \bA(y)\,G:(M,0)\,dy$ is a bounded linear functional on $L^2(Y)$.
	The minimiser $\Gc^M_\rho$ is unique, so the convergence holds for the full
	sequence; it therefore suffices to identify $\Gc^\ast$ in each of the three cases.

	\textbf{Step 2: Auxiliary bounds on the correctors.}
	For each $n$, let $(\hQ_n, \hR_n) \in \Cc^0_{\rho_n}$ denote the R2 minimising
	corrector pair for parameter $\rho_n$, so that
	\[
	\Gc^M_{\rho_n}
	= \bigl(M + \nabla_{y'}\hQ_n + \rho_n\nabla_{y'}\hR_n,\;
	\partial_{y_3}\hR_n\bigr)
	= \bigl(M + \nabla_{y'}\hQ_{\rho_n},\;
	\partial_{y_3}\hR_n\bigr),
	\]
	where $\hQ_{\rho_n} := \hQ_n + \rho_n\hR_n \in H^1_{\mathrm{per},0}(Y;\Szero)$
	is the combined in-plane corrector.
	The equality in Lemma~\ref{lem:coercivity} and~\cref{eq:Gbound} give
	\begin{equation}
		\label{eq:aux-combined}
		\sup_n \|\nabla_{y'}\hQ_{\rho_n}\|_{L^2(Y)} \leq C,
		\qquad
		\sup_n \|\partial_{y_3}\hR_n\|_{L^2(Y)} \leq C.
	\end{equation}
	Since $\hQ_n$ is $y_3$-independent and the zero-mean condition gives
	$\int_I \hR_n(\cdot,y_3)\,dy_3 = 0$ a.e.\ in $Y'$, we have
	\begin{equation}
		\label{eq:Qn-mean}
		\hQ_n(y') = \int_I \hQ_{\rho_n}(y',y_3)\,dy_3
		\qquad\text{a.e.\ in }Y'.
	\end{equation}
	Since $\hQ_n$ is independent of $y_3$, differentiating the 
	identity~\cref{eq:Qn-mean} in $y'$ and commuting the derivative 
	with the $y_3$-integral via Fubini gives
	$\nabla_{y'}\hQ_n(y') = \int_I \nabla_{y'}\hQ_{\rho_n}(y',y_3)\,dy_3$
	a.e.\ in $Y'$.
	Applying Jensen's inequality to this average, using the convexity 
	of $|\cdot|^2$ and the fact that $|I|=1$, yields
	$$|\nabla_{y'}\hQ_n(y')|^2 \leq \int_I |\nabla_{y'}\hQ_{\rho_n}(y',y_3)|^2\,dy_3$$
	pointwise a.e.\ in $Y'$.
	Integrating over $Y'$ and using that $\hQ_n$ is $y_3$-independent 
	(so that $\|\nabla_{y'}\hQ_n\|_{L^2(Y')}
	= \|\nabla_{y'}\hQ_n\|_{L^2(Y)}$, since $|I|=1$) gives 
	the bound
	\begin{equation}
		\label{eq:aux-Qn}
		\|\nabla_{y'}\hQ_n\|_{L^2(Y')}^2
		\;\leq\; \|\nabla_{y'}\hQ_{\rho_n}\|_{L^2(Y)}^2
		\;\leq\; C.
	\end{equation}
	Writing $\rho_n\nabla_{y'}\hR_n = \nabla_{y'}\hQ_{\rho_n} - \nabla_{y'}\hQ_n$
	and applying the triangle inequality with~\cref{eq:aux-combined}--\cref{eq:aux-Qn}:
	\begin{equation}
		\label{eq:aux-rhoR}
		\sup_n\,\|\rho_n\nabla_{y'}\hR_n\|_{L^2(Y)}
		\;\leq\;
		\sup_n\,\|\nabla_{y'}\hQ_{\rho_n}\|_{L^2(Y)}
		+ \sup_n\,\|\nabla_{y'}\hQ_n\|_{L^2(Y)}
		\;\leq\; C.
	\end{equation}
	The Poincar\'{e}--Wirtinger inequality in $y_3$ (using $\int_I\hR_n\,dy_3 = 0$)
	combined with~\cref{eq:aux-combined} gives
	$$\sup_n\|\hR_n\|_{L^2(Y';H^1_\sharp(I;\Szero))} \leq C.$$
	
	\textbf{Step 3: Interior regime, $\rho \in (0,\infty)$.}
	Since $\rho_n \to \rho > 0$, there exists $c > 0$ such that
	$\rho_n \geq c$ for all $n$ sufficiently large.
	From~\cref{eq:aux-rhoR} and the lower bound $\rho_n \geq c$,
	\begin{equation}
		\label{eq:corrector-interior}
		\sup_n\,\|\nabla_{y'}\hR_n\|_{L^2(Y)}
		\;\leq\; \frac{1}{c}\sup_n\,\|\rho_n\nabla_{y'}\hR_n\|_{L^2(Y)}
		\;\leq\; \frac{C}{c}.
	\end{equation}
	Together with~\cref{eq:aux-Qn},~\cref{eq:corrector-interior},
	and~\cref{eq:aux-combined}, the sequences $\{\hQ_n\}$ and $\{\hR_n\}$
	are bounded in $H^1_{\mathrm{per},0}(Y';\Szero)$ and
	$H^1_{\mathrm{per},\sharp}(Y;\Szero)$ respectively.
	Extracting a further subsequence,
	\[
	\hQ_n \rightharpoonup \hQ^*
	\;\text{ in }H^1_{\mathrm{per},0}(Y';\Szero),
	\qquad
	\hR_n \rightharpoonup \hR^*
	\;\text{ in }H^1_{\mathrm{per},\sharp}(Y;\Szero).
	\]
	Since $\rho_n \to \rho$, we have
	$\rho_n\nabla_{y'}\hR_n \rightharpoonup \rho\nabla_{y'}\hR^*$
	weakly in $L^2(Y)$, and hence
	\[
	\Gc^M_{\rho_n}
	= \bigl(M + \nabla_{y'}\hQ_n + \rho_n\nabla_{y'}\hR_n,\;
	\partial_{y_3}\hR_n\bigr)
	\;\rightharpoonup\;
	\bigl(M + \nabla_{y'}\hQ^* + \rho\nabla_{y'}\hR^*,\;
	\partial_{y_3}\hR^*\bigr)
	=: \Gc^\ast
	\]
	weakly in $L^2(Y)$.
	Passing $n\to\infty$ in the cell
	equations~\cref{eq:CPrho-a}--\cref{eq:CPrho-b} and using $\rho_n\to\rho$
	shows that $(\hQ^*,\hR^*)\in\C^0_\rho$ satisfies the R2 cell problem for
	parameter $\rho$.  By uniqueness of the minimiser,
	$(\hQ^*,\hR^*) = (\hQ_\rho,\hR_\rho)$ and $\Gc^\ast = \Gc^M_\rho$.

	\textbf{Step 4: Endpoint $\rho = 0$, i.e.\ $\rho_n \to 0^+$.}
	From~\cref{eq:aux-Qn},~\cref{eq:aux-rhoR}, and~\cref{eq:aux-combined},
	\begin{equation}
		\label{eq:rho0-bounds}
		\sup_n\,\|\nabla_{y'}\hQ_n\|_{L^2(Y)} \leq C,
		\qquad
		\sup_n\,\|\rho_n\nabla_{y'}\hR_n\|_{L^2(Y)} \leq C,
		\qquad
		\sup_n\,\|\partial_{y_3}\hR_n\|_{L^2(Y)} \leq C,
	\end{equation}
	with $\sup_n\|\hR_n\|_{L^2(Y';H^1_\sharp(I;\Szero))} \leq C$
	from the Poincar\'{e}--Wirtinger inequality in $y_3$.
	We claim that
	\begin{equation}
		\label{eq:rho0-claim}
		\rho_n\nabla_{y'}\hR_n \;\rightharpoonup\; 0
		\quad\text{weakly in }L^2(Y).
	\end{equation}
	Indeed, for any $\Phi \in L^2(Y;\Szero\otimes\mathbb{R}^2)$ and
	$\delta > 0$, choose smooth periodic $\Phi_\delta$ with
	$\|\Phi - \Phi_\delta\|_{L^2} < \delta$.
	Integration by parts in $y'$ and the uniform $L^2$-bound on $\hR_n$
	then give
	\begin{align*}
		\Bigl|\int_Y \rho_n\nabla_{y'}\hR_n : \Phi\,dy\Bigr|
		&\leq
		\rho_n\|\hR_n\|_{L^2(Y)}
		\|\mathrm{div}_{y'}\Phi_\delta\|_{L^2(Y)}
		+ C\,\|\Phi - \Phi_\delta\|_{L^2}
		\\
		&\leq
		C\rho_n\|\mathrm{div}_{y'}\Phi_\delta\|_{L^2}
		+ C\delta
		\;\xrightarrow{n\to\infty}\; C\delta,
	\end{align*}
	since $\rho_n \to 0$.  As $\delta > 0$ was arbitrary,
	\cref{eq:rho0-claim} follows.
	
	Extracting a further subsequence from~\cref{eq:rho0-bounds}
	and using~\cref{eq:rho0-claim},
	\begin{equation}
		\label{eq:Gstar0}
		\Gc^\ast = \bigl(M + \nabla_{y'}\hQ^*,\;\partial_{y_3}\hR^*\bigr),
	\end{equation}
	where $\hQ^* \in H^1_{\mathrm{per},0}(Y';\Szero)$ is the weak
	$H^1_{\mathrm{per},0}$-limit of $\hQ_n$, and
	$\hR^* \in L^2(Y';H^1_\sharp(I;\Szero))$ is the weak limit of $\hR_n$.
	In particular, $(\hQ^*,\hR^*) \in \C^0_0$.
	
	It remains to verify the R3 cell equations for $(\hQ^*,\hR^*)$.
	Equation~\cref{eq:CP0a} follows by passing $n\to\infty$ directly
	in~\cref{eq:CPrho-a}, using weak convergence of $\Gc^M_{\rho_n}$ to $\Gc^\ast$.
	For~\cref{eq:CP0b}, rewrite~\cref{eq:CPrho-b} as
	\[
	\rho_n\!\int_Y \bA(y)\,\Gc^M_{\rho_n}:(\nabla_{y'}\psi,0)\,dy
	\;+\;
	\int_Y \bA(y)\,\Gc^M_{\rho_n}:(0,\partial_{y_3}\psi)\,dy
	\;=\; 0.
	\]
	The first term vanishes as $n\to\infty$: the integral is bounded
	by~\cref{eq:Gbound} and $\|\bA\|_{L^\infty}$, and $\rho_n\to 0$.
	The second term converges to
	$\int_Y\bA(y)\,\Gc^\ast:(0,\partial_{y_3}\psi)\,dy$
	by weak convergence~\cref{eq:weak-conv}.
	Hence
	\[
	\int_Y\bA(y)\,\Gc^\ast:(0,\partial_{y_3}\psi)\,dy = 0
	\qquad
	\forall\,\psi \in H^1_{\mathrm{per},\sharp}(Y;\Szero).
	\]
	Since smooth functions $C^\infty_{\mathrm{per},\sharp}(Y;\Szero)$ are
	dense in both $H^1_{\mathrm{per},\sharp}(Y;\Szero)$ and
	$L^2(Y';H^1_\sharp(I;\Szero))$ (Lemma~\ref{lem:dense-subspaces}), the identity extends to all
	$\psi \in L^2(Y';H^1_\sharp(I;\Szero))$, which is precisely
	the test space for~\cref{eq:CP0b}.
	By uniqueness of the R3 minimiser, $\Gc^\ast = \Gc^M_0$.

	\textbf{Step 5: Endpoint $\rho = \infty$,
		i.e.\ $\rho_n \to +\infty$.}
	From~\cref{eq:aux-combined} and~\cref{eq:aux-rhoR},
	\begin{equation}
		\label{eq:infty-bounds}
		\sup_n\,\|\nabla_{y'}\hQ_n
		+ \rho_n\nabla_{y'}\hR_n\|_{L^2(Y)} \leq C,
		\quad
		\sup_n\,\|\partial_{y_3}\hR_n\|_{L^2(Y)} \leq C,
		\quad
		\sup_n\,\|\rho_n\nabla_{y'}\hR_n\|_{L^2(Y)} \leq C.
	\end{equation}
	
	% -------------------------------------------------------
	% Key coercive estimate at rho = infty
	% -------------------------------------------------------
	We now show that $\nabla_{y'}\hR_n \to 0$ strongly in $L^2(Y)$.
	Testing the cell equation~\cref{eq:CPrho-b} with
	$\psi = \hR_n \in H^1_{\mathrm{per},\sharp}(Y;\Szero)$ gives
	\begin{equation}
		\label{eq:Rn-test}
		\rho_n \int_Y \bA(y)\,\Gc^M_{\rho_n}:(\nabla_{y'}\hR_n,0)\,dy
		\;+\;
		\int_Y \bA(y)\,\Gc^M_{\rho_n}:(0,\partial_{y_3}\hR_n)\,dy
		\;=\; 0,
	\end{equation}
	where
	$\Gc^M_{\rho_n}
	= (M + \nabla_{y'}\hQ_n + \rho_n\nabla_{y'}\hR_n,\;\partial_{y_3}\hR_n)$.
	We expand the first integral by splitting the tensor $\bA$ into
	in-plane/in-plane, in-plane/transverse, and transverse/in-plane blocks:
	\begin{multline}
				\label{eq:Rn-expand}
		\rho_n \int_Y \bA(y)\,\Gc^M_{\rho_n}:(\nabla_{y'}\hR_n,0)\,dy
		= \rho_n^2 \int_Y \bA''(y)\,\nabla_{y'}\hR_n:\nabla_{y'}\hR_n\,dy\\
		\quad + \rho_n \int_Y \bA''(y)\,(M+\nabla_{y'}\hQ_n):\nabla_{y'}\hR_n\,dy + \rho_n \int_Y \bA_{13}(y)\,\partial_{y_3}\hR_n:\nabla_{y'}\hR_n\,dy,
	\end{multline}
	where $\bA''(y)$ denotes the in-plane/in-plane block and $\bA_{13}(y)$
	the transverse/in-plane coupling block of $\bA(y)$.

	For any in-plane matrix $F' \in \Szero\otimes\mathbb{R}^2$, the full
	uniform ellipticity assumption gives
	\[
	\bA''(y)\,F':F'
	= \bA(y)\,(F',0):(F',0)
	\;\geq\; \lambda\,|F'|^2,
	\]
	so the first term in~\cref{eq:Rn-expand} satisfies
	\begin{equation}
		\label{eq:block-coercive}
		\rho_n^2 \int_Y \bA''(y)\,\nabla_{y'}\hR_n:\nabla_{y'}\hR_n\,dy
		\;\geq\; \lambda\,\rho_n^2\,\|\nabla_{y'}\hR_n\|_{L^2(Y)}^2.
	\end{equation}
	
	By Cauchy--Schwarz and~\cref{eq:aux-Qn},
	\begin{multline}
		\label{eq:cross-Q}
		\Bigl|\rho_n \int_Y \bA''(y)\,(M+\nabla_{y'}\hQ_n):\nabla_{y'}\hR_n\,dy\Bigr|
		\;\leq\; \Lambda\,\rho_n\,\|M+\nabla_{y'}\hQ_n\|_{L^2(Y)}\,
		\|\nabla_{y'}\hR_n\|_{L^2(Y)}\\
		\;\leq\; C\,\rho_n\,\|\nabla_{y'}\hR_n\|_{L^2(Y)}.
	\end{multline}
	
	By Cauchy--Schwarz and~\cref{eq:infty-bounds},
	\begin{equation}
		\label{eq:cross-13}
		\Bigl|\rho_n \int_Y \bA_{13}(y)\,\partial_{y_3}\hR_n:\nabla_{y'}\hR_n\,dy\Bigr|
		\;\leq\; \Lambda\,\rho_n\,\|\partial_{y_3}\hR_n\|_{L^2(Y)}\,
		\|\nabla_{y'}\hR_n\|_{L^2(Y)}
		\;\leq\; C\,\rho_n\,\|\nabla_{y'}\hR_n\|_{L^2(Y)}.
	\end{equation}
	
	By~\cref{eq:Gbound},~\cref{eq:infty-bounds}, and $\|\bA\|_{L^\infty}\leq\Lambda$,
	\begin{equation}
		\label{eq:Rn-II}
		\Bigl|\int_Y \bA(y)\,\Gc^M_{\rho_n}:(0,\partial_{y_3}\hR_n)\,dy\Bigr|
		\;\leq\; \Lambda\,\|\Gc^M_{\rho_n}\|_{L^2(Y)}\,\|\partial_{y_3}\hR_n\|_{L^2(Y)}
		\;\leq\; C.
	\end{equation}
	
	Setting $z_n := \|\nabla_{y'}\hR_n\|_{L^2(Y)} \geq 0$ and
	combining~\cref{eq:Rn-test,eq:Rn-expand,eq:block-coercive,eq:cross-Q,%
		eq:cross-13,eq:Rn-II}:
	\[
	\lambda\,\rho_n^2\,z_n^2 - C\,\rho_n\,z_n - C \;\leq\; 0.
	\]
	The quadratic formula in $z_n \geq 0$ gives
	\[
	z_n
	\;\leq\;
	\frac{C\,\rho_n + \sqrt{C^2\rho_n^2 + 4\lambda\,\rho_n^2\,C}}{2\lambda\,\rho_n^2}
	\;=\;
	\frac{C + \sqrt{C^2 + 4\lambda\,C}}{2\lambda\,\rho_n},
	\]
	and hence
	\begin{equation}
		\label{eq:Rn-decay}
		\|\nabla_{y'}\hR_n\|_{L^2(Y)}
		\;=\; z_n
		\;\leq\; \frac{C}{\rho_n}
		\;\to\; 0,
	\end{equation}
	using $\rho_n \to +\infty$.
	Together with the Poincar\'{e}--Wirtinger inequality in $y_3$ and
	$\sup_n\|\partial_{y_3}\hR_n\|_{L^2(Y)} \leq C$, the sequence
	$\{\hR_n\}$ is bounded in $L^2(Y';\,H^1_\sharp(I;\Szero))$,
	and every weak $L^2$-limit $\hR^*$ satisfies
	$\nabla_{y'}\hR^* = 0$, i.e.\
	$\hR^* = \hR^*(y_3) \in H^1_\sharp(I;\Szero)$.
	
	% -------------------------------------------------------
	% Build F* as an explicit in-plane gradient
	% -------------------------------------------------------
	
	For a.e.\ $y_3 \in I$, define
	\begin{equation}
		\label{eq:Vn-def}
		V_n(y',y_3)
		:= \rho_n\hR_n(y',y_3)
		- \int_{Y'}\rho_n\hR_n(\tilde{y}',y_3)\,d\tilde{y}',
	\end{equation}
	so that $\int_{Y'}V_n(\cdot,y_3)\,dy' = 0$ a.e.\ and
	\[
	\nabla_{y'}V_n = \rho_n\nabla_{y'}\hR_n
	\quad\text{in }L^2(Y).
	\]
	By the Poincar\'{e}--Wirtinger inequality on $Y'$,
	\[
	\|V_n(\cdot,y_3)\|_{H^1_{\mathrm{per},0}(Y';\Szero)}
	\;\leq\; C_{\mathrm{PW}}\,\|\nabla_{y'}V_n(\cdot,y_3)\|_{L^2(Y')}
	= C_{\mathrm{PW}}\,\|\rho_n\nabla_{y'}\hR_n(\cdot,y_3)\|_{L^2(Y')},
	\]
	so the third bound in~\cref{eq:infty-bounds} gives
	\[
	\sup_n\,\|V_n\|_{L^2(I;\,H^1_{\mathrm{per},0}(Y';\Szero))} \leq C.
	\]
	 There exists a subsequence (not relabelled) and
	\begin{equation}
		\label{eq:Vstar}
		V^* \;\in\; L^2\bigl(I;\,H^1_{\mathrm{per},0}(Y';\Szero)\bigr),
		\quad
		V_n \rightharpoonup V^* \text{ weakly in }
		L^2\bigl(I;\,H^1_{\mathrm{per},0}(Y';\Szero)\bigr).
	\end{equation}
	Setting $F^* := \nabla_{y'}V^* \in L^2(Y;\Szero\otimes\mathbb{R}^2)$,
	we have
	\[
	\rho_n\nabla_{y'}\hR_n = \nabla_{y'}V_n
	\;\rightharpoonup\; \nabla_{y'}V^* = F^*
	\quad\text{weakly in }L^2(Y),
	\]
	so $F^*$ is an in-plane gradient field with zero $Y'$-mean
	a.e.\ in $I$.  From the first bound in~\cref{eq:infty-bounds}
	and the weak $H^1_{\mathrm{per},0}$-convergence $\hQ_n\rightharpoonup\hQ^*$,
	the in-plane part of $\Gc^M_{\rho_n}$ converges weakly in $L^2(Y)$ to
	$M + \nabla_{y'}\hQ^* + F^* = M + \nabla_{y'}(\hQ^* + V^*)$.
	
	Define
	\begin{equation}
		\label{eq:Qinfty}
		\hQ^*_\infty(y',y_3)
		:= \hQ^*(y') + V^*(y',y_3)
		\;\in\; L^2\bigl(I;\,H^1_{\mathrm{per},0}(Y';\Szero)\bigr).
	\end{equation}
	Since $\hQ^* \in H^1_{\mathrm{per},0}(Y';\Szero)
	\subset L^2(I;\,H^1_{\mathrm{per},0}(Y';\Szero))$
	and $V^* \in L^2(I;\,H^1_{\mathrm{per},0}(Y';\Szero))$,
	both having zero $Y'$-mean a.e.\ in $I$ by construction,
	we have $\hQ^*_\infty \in L^2(I;\,H^1_{\mathrm{per},0}(Y';\Szero))$.
	Together with $\hR^* \in H^1_\sharp(I;\Szero)$, this gives
	$(\hQ^*_\infty,\hR^*) \in \C^0_\infty$,
	and the weak limit of $\Gc^M_{\rho_n}$ is
	\[
	\Gc^* = \bigl(M + \nabla_{y'}\hQ^*_\infty,\;\partial_{y_3}\hR^*\bigr).
	\]
	
	We verify the R1 cell equations.
	
	\textbf{For~\cref{eq:CPinf-a}}: we proceed in two steps.
	
	\textit{Step A ($F_\phi$ is constant in $y_3$).}
	Take a test function of the form
	$\psi = \phi(y')\tilde\eta(y_3)$ with
	$\phi \in H^1_{\mathrm{per},0}(Y';\Szero)$
	and $\tilde\eta \in H^1_\sharp(I)$
	(so that $\int_I \tilde\eta\,dy_3 = 0$, making $\psi$ admissible
	in $H^1_{\mathrm{per},\sharp}(Y;\Szero)$)
	in~\cref{eq:CPrho-b}, and divide by $\rho_n$:
	\[
	\int_Y \bA(y)\,\Gc^M_{\rho_n}:(\nabla_{y'}\phi\cdot\tilde\eta,\,0)\,dy
	\;+\;
	\frac{1}{\rho_n}
	\int_Y \bA(y)\,\Gc^M_{\rho_n}:(0,\,\partial_{y_3}(\phi\tilde\eta))\,dy
	= 0.
	\]
	The second term tends to zero since $\rho_n\to\infty$ and the integral
	is bounded by
	$\|\bA\|_{L^\infty}\|\Gc^M_{\rho_n}\|_{L^2}
	\|\phi\|_{H^1(Y')}\|\tilde\eta\|_{H^1(I)}\leq C$.
	Passing $n\to\infty$ in the first term gives
	\[
	\int_Y \bA(y)\,\Gc^*:(\nabla_{y'}\phi\cdot\tilde\eta,\,0)\,dy = 0
	\qquad
	\forall\,\phi\in H^1_{\mathrm{per},0}(Y';\Szero),\;
	\forall\,\tilde\eta\in H^1_\sharp(I).
	\]
	By density of $H^1_\sharp(I)$ in $\{f\in L^2(I):\int_I f\,dy_3=0\}$,
	this extends to all zero-mean $\tilde\eta\in L^2(I)$.
	Defining $F_\phi(y_3) := \int_{Y'}\bA(y)\,\Gc^*:(\nabla_{y'}\phi,0)\,dy'$,
	the identity says $F_\phi$ is orthogonal in $L^2(I)$ to every
	zero-mean function, hence $F_\phi$ is \emph{constant} a.e.\ in $I$.
	
	\textit{Step B (the constant is zero).}
	Passing $n\to\infty$ directly in~\cref{eq:CPrho-a}
	and using weak convergence~\cref{eq:weak-conv} gives
	\[
	\int_Y \bA(y)\,\Gc^*:(\nabla_{y'}\phi,0)\,dy = 0
	\qquad
	\forall\,\phi\in H^1_{\mathrm{per},0}(Y';\Szero),
	\]
	that is, $\int_I F_\phi(y_3)\,dy_3 = 0$.
	Since $F_\phi$ is constant, it follows that $F_\phi = 0$ a.e.\ in $I$, i.e.\
	\[
	\int_{Y'}\bA(y)\,\Gc^*:(\nabla_{y'}\phi,0)\,dy' = 0
	\quad\text{for a.e.\ }y_3\in I
	\text{ and all }\phi\in H^1_{\mathrm{per},0}(Y';\Szero).
	\]
	Multiplying by an arbitrary $\eta\in L^2(I)$ and integrating over $I$
	yields~\cref{eq:CPinf-a}.
	
	\textbf{For~\cref{eq:CPinf-b}}: take any $\psi = \psi(y_3)
	\in H^1_\sharp(I;\Szero)$ independent of $y'$.
	Then $\nabla_{y'}\psi = 0$, so~\cref{eq:CPrho-b} reduces to
	\[
	\int_Y \bA(y)\,\Gc^M_{\rho_n}:(0,\partial_{y_3}\psi)\,dy = 0,
	\]
	which passes directly to the limit to give~\cref{eq:CPinf-b}.
	
	By uniqueness of the R1 minimiser, $\Gc^* = \Gc^M_\infty$,
	and the proof is complete.
\end{proof}
%\begin{remark}[Structure of the continuity proof at $\rho = \infty$]
%	\label{rem:continuity-structure}
%	Two features of Step~5 are worth highlighting. First, the bound
%	$\|\nabla_{y'}\hR_n\|_{L^2(Y)} \leq C/\rho_n$ is obtained by testing
%	the cell equation with the corrector $\hR_n$ itself and exploiting
%	the coercivity of the in-plane block of $\bA$; the resulting quadratic
%	inequality in $z_n = \|\nabla_{y'}\hR_n\|_{L^2(Y)}$ forces the
%	in-plane gradient of the transverse corrector to vanish at rate
%	$1/\rho_n$, which is the mechanism by which $\hR^*$ loses its
%	$y'$-dependence in the limit. Second, the verification
%	of~\cref{eq:CPinf-a} proceeds by first testing~\cref{eq:CPrho-b} with
%	separated variables $\phi(y')\tilde\eta(y_3)$ and dividing by $\rho_n$
%	to show that the $y_3$-dependent function $F_\phi$ is constant, and
%	then passing to the limit in~\cref{eq:CPrho-a} to show that this
%	constant is zero. This two-step argument extracts slice-wise
%	equilibrium from the full-cell equation.
%\end{remark}

\subsection{Homogenised Oseen--Frank and Ericksen energy}

\paragraph{Derivation of Homogenised Ericksen energy.}
\begin{proof}[Proof of Theorem \ref{thm:Er}]
	The anchoring sets are independent of $y$ and $\nu$ and coincide
	on the two faces, the effective constraint set is
	\[
	\mathscr A
	=
	\mathcal A(+e_3)\cap\mathcal A(-e_3)
	=
	\mathcal A(\nu).
	\]
	By Lemma~\ref{lem:density-Astar}, $C^\infty(\bar\omega;\mathscr A)$ is dense in
	$H^1(\omega;\mathscr A)$.
	Thus Assumption~\ref{asm:compatible}(iii) is satisfied.
	
	Let $\Qc\in H^1(\omega;\mathscr A)$. Define
	\[
	s(x')
	:=
	\left(\frac32\operatorname{tr}(\Qc(x')^2)\right)^{1/2}
	=
	\sqrt{\frac32}\,|\Qc(x')|,
	\qquad
	P(x')
	:=
	\frac1{s(x')}\Qc(x')+\frac13I .
	\]
	Since $\Qc(x')\in\mathscr A$ a.e. and $s_0>0$, we have
	\[
	s_0\leq s(x')\leq s_1,
	\qquad
	P(x')\in\mathbb P
	\quad\text{for a.e. }x'\in\omega.
	\]
	The Sobolev chain rule gives $s\in H^1(\omega;[s_0,s_1])$, $P\in H^1(\omega;\mathbb P)$,
	and hence
	\begin{equation}\label{eq:Er01}
		\Qc=s\left(P-\frac13I\right)
		\quad\text{a.e. in }\omega .
	\end{equation}
	The representation is unique, because $s$ is recovered from $\Qc$ by
	the preceding formula and then $P=s^{-1}\Qc+\frac13I$.
	
	Conversely, if $s\in H^1(\omega;[s_0,s_1])$ and $P\in H^1(\omega;\mathbb P)$,
	then $\Qc=s\left(P-\frac13I\right)$
	belongs to $H^1(\omega;\mathscr A)$. Therefore
	\[
	(s,P)\longmapsto s\left(P-\frac13I\right)
	\]
	identifies the admissible class $H^1(\omega;[s_0,s_1])\times H^1(\omega;\mathbb P)$
	with $H^1(\omega;\mathscr A)$.
	
	From \cref{eq:Er01},
	\[
	\partial_\alpha \Qc
	=
	\partial_\alpha s\left(P-\frac13I\right)
	+
	s\,\partial_\alpha P,
	\qquad \alpha=1,2.
	\]
	Thus, with the notation introduced in the statement,
	\[
	\nabla'\Qc=\Mc(s,P)
	\quad\text{in }L^2(\omega;\Szero\otimes\mathbb R^2).
	\]
	By Proposition~\ref{prop:hom-props}, $\bA^{\hom}_\rho$ is a symmetric
	coercive quadratic form on $\Szero\otimes\mathbb R^2$. Hence the
	elastic part of $\GF^{\hom,\gamma}_\rho$ satisfies
	\[
	\int_\omega
	\bA^{\hom}_\rho\nabla'\Qc:\nabla'\Qc\,dx'
	=
	\int_\omega
	\bA^{\hom}_\rho\Mc(s,P):\Mc(s,P)\,dx'.
	\]
	This keeps the full anisotropy of the homogenised tensor.
	
	Since $P\in\mathbb P$, the
	eigenvalues of $P-\frac13I$ are $\frac23, -\frac13, -\frac13$.
	Consequently, $\operatorname{tr}(\Qc^2)=\frac23s^2$,	$\operatorname{tr}(\Qc^3)=\frac29s^3$.
	Therefore the Landau-de Gennes bulk density restricted to $\mathscr A$
	depends only on $s$, and $f_b(\Qc)=f^{\mathrm{Er}}_b(s)$ a.e.~in $\omega$.
	Hence
	\[
	\int_\omega f_b(\Qc)\,dx'
	=
	\int_\omega f^{\mathrm{Er}}_b(s)\,dx'.
	\]
	Finally, since $\Qc(x')\in\mathscr A$ for a.e. $x'\in\omega$, the
	limiting anchoring contribution in $\GF^{\hom,\gamma}_\rho$ vanishes
	in the regime $0\leq\gamma<1$. Combining the preceding identities gives
	\[
	\GF^{\hom,\gamma}_\rho(\Qc)
	=
	\GF^{\hom}_{\mathrm{Er},\rho}(s,P).
	\]
	The equivalence of the two minimisation problems follows from the
	bijection between the admissible classes described above.
\end{proof}

\paragraph{Derivation of Homogenised Oseen--Frank energy.}
\begin{proof}[Proof of Theorem \ref{thm:OF}]
	The anchoring sets are independent of $y$ and $\nu$ and coincide
	on the two faces, the effective constraint set is $\mathscr A
	=
	\mathcal A(+e_3)\cap\mathcal A(-e_3)
	=
	\mathcal A(\nu)$.
	Because $s_0\neq0$, the map
	\[
	P\mapsto s_0\left(P-\frac13I\right)
	\]
	is a smooth embedding of $\mathbb P$ into $\Szero$. Hence $\mathscr A$
	is a compact smooth submanifold of $\Szero$, diffeomorphic to
	$\mathbb P\cong\mathbb{RP}^2$. Therefore
	Assumption~\ref{asm:compatible}(ii) follows from
	Lemma~\ref{lem:density-Astar}.
	
	Let $\Qc\in H^1(\omega;\mathscr A)$. Since $\Qc(x')\in\mathscr A$ for
	a.e. $x'\in\omega$, there exists a unique
	\[
	P(x')=\frac1{s_0}\Qc(x')+\frac13I
	\]
	with $P(x')\in\mathbb P$ a.e. The affine formula above gives $P\in H^1(\omega;\mathbb P)$,
	and $\Qc=s_0\left(P-\frac13I\right)$ a.e.~in $\omega$.
	Conversely, every $P\in H^1(\omega;\mathbb P)$ defines an element
	$\Qc\in H^1(\omega;\mathscr A)$ by this formula. Hence the two
	admissible classes are identified.
	
	The identity for the energy now follows directly from
	Theorem~\ref{thm:Er} by taking the scalar order parameter to be the
	constant $s=s_0$. Indeed, $\nabla'\Qc=s_0\nabla'P$.
	Therefore the homogenised elastic part is
	\[
	\int_\omega
	\bA^{\hom}_\rho\nabla'\Qc:\nabla'\Qc\,dx'
	=
	s_0^2
	\int_\omega
	\bA^{\hom}_\rho\nabla'P:\nabla'P\,dx'.
	\]
	This expression retains the full anisotropy of the tensor
	$\bA^{\hom}_\rho$ from Proposition~\ref{prop:hom-props}.
	
	For the bulk term, since $P\in\mathbb P$, the eigenvalues of
	$P-\frac13I$ are $\frac23, -\frac13, -\frac13$.
	Hence $\operatorname{tr}(\Qc^2)=\frac23s_0^2$, $\operatorname{tr}(\Qc^3)=\frac29s_0^3$.
	Thus $f_b$ is constant on $\mathscr A$, equal to $c_b$, and the bulk
	contribution is
	\[
	\int_\omega f_b(\Qc)\,dx'=c_b|\omega|.
	\]
	Finally, since $\Qc(x')\in\mathscr A$ a.e.~in $\omega$, the limiting
	anchoring contribution vanishes in the regime $0\leq\gamma<1$.
	Combining these identities gives
	\[
	\GF^{\hom,\gamma}_\rho(\Qc)
	=
	\GF^{\hom}_{\mathrm{OF},\rho}(P).
	\]
	The equivalence of the two minimisation problems follows from the
	bijection between $H^1(\omega;\mathbb P)$ and
	$H^1(\omega;\mathscr A)$ described above.
\end{proof}

\section*{Acknowledgements}
The author thanks Apala Majumdar (University of Manchester, UK) and Georges Griso (LJLL, Sorbonne Université) for valuable comments and suggestions during the preparation of this work.

	\section*{Data availability }
No datasets were generated or analysed during the current study.

\section*{Declarations and Competing interests}
The authors declare no competing interests.

	\appendix

	\section{Appendix}
	\label{app:den}
	
	\subsection{Dense subspaces}
	For the smooth dense subspaces, we also use the notation
	\[
	C^\infty_\sharp(\overline I;\Szero)
	:=
	\left\{
	f\in C^\infty(\overline I;\Szero):
	\int_I f(y_3)\,dy_3=0
	\right\},
	\]
	and
	\[
	C^\infty_{\mathrm{per},\sharp}(Y;\Szero)
	:=
	\left\{
	f\in C^\infty_{\mathrm{per}}(Y;\Szero):
	\int_I f(y',y_3)\,dy_3=0
	\ \text{for every }y'\in Y'
	\right\}.
	\]
	Similarly, \(C^\infty_{\mathrm{per},0}(Y';\Szero)\) and
	\(C^\infty_{\mathrm{per},0}(Y;\Szero)\) denote smooth periodic functions with
	zero average over \(Y'\) and \(Y\), respectively.
	\begin{lemma}[Smooth dense subspaces of \(\mathcal C_\rho\)]
		\label{lem:dense-subspaces}
		For each regime \(\rho\in[0,\infty]\), define
		\(\mathcal D_\rho\subset\mathcal C_\rho\) as follows (for the definition $\Cc_\rho$ see Section \ref{sec:main-results})
		
		\textit{Regime \(\rho=0\).}
		\[
		\mathcal D_0
		:=
		C^\infty_c\!\bigl(\omega;
		C^\infty_{\mathrm{per},0}(Y';\Szero)\bigr)
		\times
		C^\infty_c\!\bigl(\omega;
		C^\infty_{\mathrm{per}}(Y';
		C^\infty_\sharp(\overline I;\Szero))\bigr).
		\]

		\noindent
		\textit{Regime \(\rho\in(0,\infty)\).}
		\[
		\mathcal D_\rho
		:=
		\left\{
		(\rho\,\widehat R+\widehat Q,\widehat R):
		\begin{array}{l}
			\widehat R\in
			C^\infty_c\!\bigl(\omega;
			C^\infty_{\mathrm{per},\sharp}(Y;\Szero)\bigr),\\[2mm]
			\widehat Q\in
			C^\infty_c\!\bigl(\omega;
			C^\infty_{\mathrm{per},0}(Y';\Szero)\bigr)
		\end{array}
		\right\},
		\]
		where \(\widehat Q(x',y')\) is identified with its constant extension in the
		\(y_3\)-variable.

		\noindent
		\textit{Regime \(\rho=\infty\).}
		\[
		\mathcal D_\infty
		:=
		C^\infty_c\!\bigl(\omega\times I;
		C^\infty_{\mathrm{per},0}(Y';\Szero)\bigr)
		\times
		C^\infty_c\!\bigl(\omega;
		C^\infty_\sharp(\overline I;\Szero)\bigr).
		\]
		
		Then \(\mathcal D_\rho\) is dense in \(\mathcal C_\rho\) with respect to the
		natural Hilbert norm of \(\mathcal C_\rho\).
	\end{lemma}
	
	\begin{proof}
		We first recall the elementary density facts used below.
		
		Let \(X\) be a separable Hilbert space. Then
		\(C^\infty_c(\omega;X)\) is dense in \(L^2(\omega;X)\). Similarly,
		\(C^\infty_c(\omega\times I;X)\) is dense in \(L^2(\omega\times I;X)\).
		Moreover, smooth periodic functions are dense in the corresponding periodic
		Sobolev spaces. The zero-mean conditions are preserved by applying the bounded
		linear projections
		\[
		P_0 f := f-\int_{Y'} f(y')\,dy',\quad\text{on \(H^1_{\mathrm{per}}(Y';\Szero)\)},\quad
		P_\sharp f(y_3)
		:=
		f(y_3)-\int_I f(\tau)\,d\tau,\quad\text{on \(H^1(I;\Szero)\)}.
		\] 
		Likewise, on \(H^1_{\mathrm{per}}(Y;\Szero)\) we use
		\[
		P_\sharp^3 f(y',y_3)
		:=
		f(y',y_3)
		-
		\int_I f(y',\tau)\,d\tau .
		\]
		This projection maps \(H^1_{\mathrm{per}}(Y;\Szero)\) continuously onto
		\(H^1_{\mathrm{per},\sharp}(Y;\Szero)\) and preserves smoothness.

		\noindent
		\textit{Regime \(\rho=0\).}
		The space \(\mathcal C_0\) is the product
		\[
		L^2\!\bigl(\omega;H^1_{\mathrm{per},0}(Y';\Szero)\bigr)
		\times
		L^2\!\bigl(\omega\times Y';H^1_\sharp(I;\Szero)\bigr).
		\]
		By the density facts recalled above,
		$C^\infty_c\!\bigl(\omega;
		C^\infty_{\mathrm{per},0}(Y';\Szero)\bigr)$
		is dense in
		$L^2\!\bigl(\omega;H^1_{\mathrm{per},0}(Y';\Szero)\bigr)$,
		and\\
		$C^\infty_c\!\bigl(\omega;
		C^\infty_{\mathrm{per}}(Y';
		C^\infty_\sharp(\overline I;\Szero))\bigr)$
		is dense in
		$L^2\!\bigl(\omega\times Y';H^1_\sharp(I;\Szero)\bigr)$.
		Therefore \(\mathcal D_0\) is dense in \(\mathcal C_0\).

		\noindent
		\textit{Regime \(\rho=\infty\).}
		The space \(\mathcal C_\infty\) is the product
		\[
		L^2\!\bigl(\omega\times I;
		H^1_{\mathrm{per},0}(Y';\Szero)\bigr)
		\times
		L^2\!\bigl(\omega;H^1_\sharp(I;\Szero)\bigr).
		\]
		Again by the same density facts,
		$C^\infty_c\!\bigl(\omega\times I;
		C^\infty_{\mathrm{per},0}(Y';\Szero)\bigr)$
		is dense in
		$L^2\!\bigl(\omega\times I;
		H^1_{\mathrm{per},0}(Y';\Szero)\bigr)$,
		and
		$C^\infty_c\!\bigl(\omega;
		C^\infty_\sharp(\overline I;\Szero)\bigr)$
		is dense in
		$L^2\!\bigl(\omega;H^1_\sharp(I;\Szero)\bigr)$.
		Hence \(\mathcal D_\infty\) is dense in \(\mathcal C_\infty\).

		\noindent
		\textit{Regime \(\rho\in(0,\infty)\).}
		Let
		\[
		X_\rho
		:=
		L^2\!\bigl(\omega;
		H^1_{\mathrm{per},\sharp}(Y;\Szero)\bigr)
		\times
		L^2\!\bigl(\omega;
		H^1_{\mathrm{per},0}(Y';\Szero)\bigr).
		\]
		Define
		\[
		\Psi_\rho:X_\rho\to\mathcal C_\rho,
		\qquad
		\Psi_\rho(\widehat R,\widehat Q)
		:=
		(\rho\,\widehat R+\widehat Q,\widehat R),
		\]
		where \(\widehat Q(x',y')\) is extended constantly in \(y_3\).
		
		The map \(\Psi_\rho\) is linear and bounded. Moreover, by the definition of
		\(\mathcal C_\rho\), every element of \(\mathcal C_\rho\) is of the form
		\[
		(\widehat Q_\rho,\widehat R)
		=
		(\widehat Q+\rho\,\widehat R,\widehat R)
		\]
		for a unique
		\[
		\widehat Q
		\in
		L^2\!\bigl(\omega;
		H^1_{\mathrm{per},0}(Y';\Szero)\bigr).
		\]
		Indeed, one recovers \(\widehat Q\) by
		\[
		\widehat Q(x',y')
		=
		\int_I
		\bigl(
		\widehat Q_\rho-\rho\,\widehat R
		\bigr)(x',y',y_3)\,dy_3.
		\]
		Thus \(\Psi_\rho\) is a Hilbert-space isomorphism from \(X_\rho\) onto
		\(\mathcal C_\rho\), with bounded inverse.
		
		Now
		$C^\infty_c\!\bigl(\omega;
		C^\infty_{\mathrm{per},\sharp}(Y;\Szero)\bigr)$
		is dense in
		$L^2\!\bigl(\omega;
		H^1_{\mathrm{per},\sharp}(Y;\Szero)\bigr)$,
		and
		$C^\infty_c\!\bigl(\omega;
		C^\infty_{\mathrm{per},0}(Y';\Szero)\bigr)$
		is dense in
		$L^2\!\bigl(\omega;
		H^1_{\mathrm{per},0}(Y';\Szero)\bigr)$.
		Therefore the product of these two smooth spaces is dense in \(X_\rho\).
		Applying the continuous isomorphism \(\Psi_\rho\), we obtain the density of
		\(\mathcal D_\rho\) in \(\mathcal C_\rho\).
		
		The proof is complete.
	\end{proof}
	
	\begin{lemma}[Density of smooth constrained maps]
		\label{lem:density-Astar}
		Under Assumption~\ref{asm:compatible}, the space
		$C^\infty(\bar\omega; \mathscr{A})$ is dense in
		$H^1(\omega; \mathscr{A})$ with respect to the $H^1(\omega; \Szero)$
		norm.
	\end{lemma}
		\begin{proof}
		\textbf{Case (i):} $\mathscr{A}\subset \Szero$ is non-empty, closed, and convex.
		
		Let $Q\in H^1(\omega;\mathscr{A})$.  Since
		$\omega$ is a bounded Lipschitz domain, there exists a bounded linear
		extension operator
		\[
		E:H^1(\omega;\Szero)\to H^1(\mathbb R^2;\Szero).
		\]
		Set $\widetilde Q:=EQ$. The function $\widetilde Q$ need not be
		$\mathscr{A}$-valued outside $\omega$. Let
		\[
		\Pi_{\mathscr{A}}:\Szero\to\mathscr{A}
		\]
		denote the metric projection onto the closed convex set $\mathscr{A}$.
		The map $\Pi_{\mathscr{A}}$ is $1$-Lipschitz. Hence
		\[
		\widehat Q:=\Pi_{\mathscr{A}}(\widetilde Q)
		\in H^1(\mathbb R^2;\mathscr{A}).
		\]
		Moreover, since $Q(x')\in\mathscr{A}$ for a.e. $x'\in\omega$, we have
		\[
		\widehat Q=Q
		\quad\text{a.e. in }\omega.
		\]
		
		Let $\{\rho_\eta\}_{\eta>0}$ be a standard family of mollifiers on
		$\mathbb R^2$ and define
		\[
		Q^\eta:=\rho_\eta * \widehat Q .
		\]
		Then $Q^\eta\in C^\infty(\mathbb R^2;\Szero)$. Since
		$\widehat Q(y)\in\mathscr{A}$ for a.e. $y\in\mathbb R^2$ and
		$\mathscr{A}$ is convex, Jensen's inequality gives
		\[
		Q^\eta(x')
		=
		\int_{\mathbb R^2}\rho_\eta(x'-y)\widehat Q(y)\,dy
		\in\mathscr{A}
		\quad\text{for every }x'\in\mathbb R^2.
		\]
		Therefore
		\[
		Q^\eta|_{\bar\omega}\in C^\infty(\bar\omega;\mathscr{A}).
		\]
		Finally, by the standard approximation property of mollifiers,
		\[
		Q^\eta\to \widehat Q
		\quad\text{in }H^1_{\mathrm{loc}}(\mathbb R^2;\Szero).
		\]
		In particular, since $\widehat Q=Q$ a.e. in $\omega$,
		\[
		Q^\eta\to Q
		\quad\text{in }H^1(\omega;\Szero).
		\]
		Thus $C^\infty(\bar\omega;\mathscr{A})$ is dense in
		$H^1(\omega;\mathscr{A})$ in this case.
		
		\textbf{Case (ii):} $\mathscr{A}$ is a compact smooth submanifold.
		Since $\omega \subset \mathbb{R}^2$ is a bounded Lipschitz domain,
		we have $p = 2 = \dim\omega$. At the critical exponent $p = n$,
		the density of $C^\infty(\bar\omega; \mathscr{A})$ in
		$W^{1,2}(\omega; \mathscr{A})$ holds for any compact smooth
		target manifold without boundary $\mathscr{A}$, with no topological condition
		required. This is the content of the approximation theorem of
		Schoen and Uhlenbeck~\cite{schoen1983boundary} (see also \cite{bethuel1991approximation}), which covers the
		case $p = \dim M$ for compact targets.
		
		\textbf{Case (iii):} $\mathscr{A}$ is given by \cref{eq:csmwb+} with
		$0<s_0<s_1<1$.
		
		We identify the quotient manifold $\mathbb S^2/\{\pm 1\}$ with
		\[
		\mathbb P
		:=
		\{P\in\mathbb R^{3\times 3}_{\rm sym}: P^2=P,\ \operatorname{tr}P=1\}
		\]
		through the smooth map $[\bn]\mapsto \bn\otimes \bn$.
		Thus $\mathbb P$ is a compact smooth manifold without boundary, naturally
		diffeomorphic to $\mathbb{RP}^2$, and the Sobolev space of unoriented
		director fields is understood as
		\[
		H^1(\omega;\mathbb S^2/\{\pm1\})
		\cong
		H^1(\omega;\mathbb P).
		\]
		This formulation does not require the existence of a global
		$H^1$-director lifting $\bn\in H^1(\omega;\mathbb S^2)$.
		
		Since $\omega\subset\mathbb R^2$ is a bounded Lipschitz domain and
		$\mathbb P$ is a compact smooth manifold without boundary, the critical
		case $p=\dim\omega=2$ of the Schoen--Uhlenbeck approximation theorem
		\cite{schoen1983boundary} gives
		$C^\infty(\bar\omega;\mathbb P)$ dense in
		$H^1(\omega;\mathbb P)$.		
		Equivalently,
		\[
		C^\infty(\bar\omega;\mathbb S^2/\{\pm1\})
		\quad\text{is dense in}\quad
		H^1(\omega;\mathbb S^2/\{\pm1\}).
		\]
		
		With this notation,
		\[
		\mathscr{A}
		=
		\left\{
		s\left(P-\frac13 I\right):
		P\in\mathbb P,\ s\in[s_0,s_1]
		\right\}.
		\]
		
		Let $\Qc\in H^1(\omega;\mathscr{A})$. Define
		\[
		s(x'):=\sqrt{\frac32}\,|\Qc(x')|,
		\qquad
		P(x'):=\frac{1}{s(x')}\Qc(x')+\frac13 I.
		\]
		Since $s_0\leq s(x')\leq s_1$ a.e. in $\omega$, the map
		$t\mapsto 1/t$ is smooth and bounded on the range of $s$. Hence, by the
		Sobolev chain rule, $s\in H^1(\omega;[s_0,s_1])$,
		$P\in H^1(\omega;\mathbb P)$,
		and
		\[
		\Qc=s\left(P-\frac13 I\right)
		\quad\text{a.e. in }\omega.
		\]
		We now approximate the two factors separately. Since $[s_0,s_1]$ is a
		closed convex interval, a standard mollification argument combined with
		projection onto $[s_0,s_1]$ gives a sequence
		\[
		s_k\in C^\infty(\bar\omega),
		\qquad
		s_k(x')\in[s_0,s_1]\ \text{for every }x'\in\bar\omega,
		\qquad
		s_k\to s\quad\text{in }H^1(\omega).
		\]
		Moreover, by the density result above, there exists a sequence
		\[
		P_k\in C^\infty(\bar\omega;\mathbb P),
		\qquad
		P_k\to P\quad\text{in }H^1(\omega;\mathbb R^{3\times3}).
		\]
		After passing to a subsequence, not relabelled, we may also assume that
		\[
		s_k\to s
		\quad\text{and}\quad
		P_k\to P
		\quad\text{a.e. in }\omega.
		\]
		
		Define
		\[
		\Qc_k(x')
		:=
		s_k(x')\left(P_k(x')-\frac13 I\right).
		\]
		Then $\Qc_k\in C^\infty(\bar\omega;\mathscr{A})$.
		
		It remains to show that $\Qc_k\to\Qc$ in $H^1(\omega;\Szero)$. First,
		\[
		\Qc_k-\Qc
		=
		(s_k-s)\left(P_k-\frac13 I\right)
		+
		s(P_k-P).
		\]
		Since $P_k$ and $P$ take values in the compact set $\mathbb P$, they are
		uniformly bounded. Therefore
		\[
		\|\Qc_k-\Qc\|_{L^2(\omega)}
		\leq
		C\|s_k-s\|_{L^2(\omega)}
		+
		s_1\|P_k-P\|_{L^2(\omega)}
		\to0.
		\]
		
		For the gradients, we write
		\[
		\nabla'(\Qc_k-\Qc)
		=
		(\nabla' s_k-\nabla' s)\left(P_k-\frac13 I\right)
		+
		\nabla' s\,(P_k-P)
		+
		s_k(\nabla'P_k-\nabla'P)
		+
		(s_k-s)\nabla'P .
		\]
		The first term converges to zero in $L^2$ because $P_k$ is uniformly
		bounded:
		\[
		\left\|
		(\nabla' s_k-\nabla' s)\left(P_k-\frac13 I\right)
		\right\|_{L^2}
		\leq
		C\|\nabla' s_k-\nabla' s\|_{L^2}
		\to0.
		\]
		For the second term, since $P_k\to P$ a.e. in $\omega$ and
		$|P_k-P|\leq C$, dominated convergence gives
		\[
		\|\nabla' s\,(P_k-P)\|_{L^2}\to0.
		\]
		For the third term, using $|s_k|\leq s_1$, we obtain
		\[
		\|s_k(\nabla'P_k-\nabla'P)\|_{L^2}
		\leq
		s_1\|\nabla'P_k-\nabla'P\|_{L^2}
		\to0.
		\]
		Finally, since $s_k\to s$ a.e. in $\omega$, $|s_k-s|\leq C$, and
		$\nabla'P\in L^2(\omega)$, dominated convergence gives
		\[
		\|(s_k-s)\nabla'P\|_{L^2}\to0.
		\]
		Combining the preceding estimates, we conclude that
		\[
		\Qc_k\to\Qc
		\quad\text{in }H^1(\omega;\Szero).
		\]
		This proves the desired density in Case (iii).
	\end{proof}
	
	\subsection{Coercivity estimate}
	\begin{lemma}
		\label{lem:coercivity}
		Assume the conclusions of Lemmas~\ref{lem:71} and~\ref{prop:admissible-correctors}. 
		Denote by $C_I > 0$ the Poincar\'{e}-Wirtinger constant of $I$, i.e.\ 
		$\|f\|_{L^2(I)} \leq C_I \|\partial_{y_3} f\|_{L^2(I)}$ for all $f \in H^1(I)$ 
		with $\int_I f \, dy_3 = 0$, and by $C_{Y'}, C_Y > 0$ the Poincar\'{e} constants 
		of $Y'$ and $Y$, respectively. Then the following two-sided bound holds in each 
		regime:
		\begin{multline}
			\label{eq:coercivity}
			c_\rho \Bigl(
			\|\nabla' \Qc\|^2_{L^2(\omega)}
			+ \|\wh{Q}_\rho\|^2_{L^2(\omega \times I;\, H^1(Y'))}
			+ \|\wh{R}\|^2_{L^2(\omega \times Y';\, H^1(I))}
			\Bigr)
			\;\leq\;
			\|\Gc_\rho\|^2_{L^2(\omega \times Y)}\\
			\;=\;
			\|\nabla' \Qc\|^2_{L^2(\omega)}
			+ \|\nabla_{y'}\wh{Q}_\rho\|^2_{L^2(\omega \times Y)}
			+ \|\partial_{y_3}\wh{R}\|^2_{L^2(\omega \times Y)},
		\end{multline}
		where the upper bound holds as an equality, and the lower bound constant 
		$c_\rho > 0$ is given explicitly as follows:
		\begin{enumerate}[label=\textup{(\roman*)}]
			\item \textup{(Regime $\rho = 0$).}
			Here $\wh{Q}_0 \in L^2\!\bigl(\omega \times I;\, H^1_{\mathrm{per},0}(Y';\, \Szero)\bigr)$, and
			\[
			c_0
			\;=\;\min\left\{
			\frac{1}{1 + C_{Y'}},
			\frac{1}{1 + C_I^2}\right\}>0.
			\]
			
			\item \textup{(Regime $\rho \in (0,\infty)$).}
			Here $\wh{Q}_\rho = \wh{Q} + \rho\,\wh{R}
			\in L^2\!\bigl(\omega;\, H^1_{\mathrm{per},0}(Y;\, \Szero)\bigr)$, and
			\[
			c_\rho
			\;=\;\min\left\{
			\frac{1}{C_Y(1 + \rho^2) + 1},
			\frac{1}{1 + C_I^2}\right\}>0.
			\]
			In particular $c_\rho \sim (1 + \rho^2)^{-1}$ as $\rho \to \infty$, 
			so the lower bound degenerates.
			
			\item \textup{(Regime $\rho = \infty$).}
			Here $\wh{Q}_\infty \in L^2\!\bigl(\omega \times I;\, H^1_{\mathrm{per},0}(Y';\, \Szero)\bigr)$ 
			and $\wh{R}$ is $y'$-independent, and
			\[
			c_\infty
			\;=\;\min\left\{
			\frac{1}{1 + C_{Y'}},
			\frac{1}{1 + C_I^2}\right\}>0.
			\]
		\end{enumerate}
	\end{lemma}
	
	\begin{proof}
		We verify the equality (upper bound) first, then establish the lower bound 
		regime by regime.
		
		In all regimes the two-scale limit decomposes as
		\[
		\Gc_\rho
		= \bigl(\nabla'\Qc + \nabla_{y'}\wh{Q}_\rho,\; 0\bigr)
		+ \bigl(0,\; \partial_{y_3}\wh{R}\bigr),
		\]
		and the two summands are orthogonal in $L^2(\omega \times Y;\, \Szero \otimes \mathbb{R}^3)$. 
		It therefore suffices to show that the cross term in the in-plane part vanishes. 
		Since $\nabla'\Qc$ is independent of $y$,
		\[
		\int_{\omega \times Y} \nabla'\Qc \cdot \nabla_{y'}\wh{Q}_\rho \, dy\, dx'
		= \int_\omega \nabla'\Qc(x') \cdot \Bigl(\int_Y \nabla_{y'}\wh{Q}_\rho \, dy\Bigr) dx'.
		\]
		In every regime $\wh{Q}_\rho$ is $Y'$-periodic, so the divergence theorem on 
		$Y'$ gives $\int_{Y'} \nabla_{y'}\wh{Q}_\rho \, dy' = 0$ for a.e.\ $y_3 \in I$. 
		Hence the cross term vanishes, and we obtain 
		\begin{equation}
			\label{eq:upper-equality}
			\|\Gc_\rho\|^2_{L^2(\omega \times Y)}
			= \|\nabla'\Qc\|^2_{L^2(\omega)}
			+ \|\nabla_{y'}\wh{Q}_\rho\|^2_{L^2(\omega \times Y)}
			+ \|\partial_{y_3}\wh{R}\|^2_{L^2(\omega \times Y)}.
		\end{equation}	
		The orthogonality condition~\cref{es:ort} gives $\int_I \wh{R}(\cdot, y_3)\, dy_3 = 0$ 
		a.e.\ in $\omega \times Y'$. The Poincar\'{e} inequality on $I$ yields
		$\|\wh{R}\|^2_{L^2(I)} \leq C_I^2 \|\partial_{y_3}\wh{R}\|^2_{L^2(I)}$,
		and hence
		\[
		\|\wh{R}\|^2_{H^1(I)}
		= \|\wh{R}\|^2_{L^2(I)} + \|\partial_{y_3}\wh{R}\|^2_{L^2(I)}
		\leq (1 + C_I^2)\,\|\partial_{y_3}\wh{R}\|^2_{L^2(I)}.
		\]
		Integrating over $\omega \times Y'$ (and using $y'$-independence of $\wh{R}$ 
		in regime $\rho = \infty$, which only replaces $L^2(\omega \times Y')$ by 
		$|Y'|^{1/2} L^2(\omega)$ without changing the argument):
		\begin{equation}
			\label{eq:R-lower}
			\|\partial_{y_3}\wh{R}\|^2_{L^2(\omega \times Y)}
			\;\geq\;
			\frac{1}{1 + C_I^2}\,
			\|\wh{R}\|^2_{L^2(\omega \times Y';\, H^1(I))}.
		\end{equation}
		This bound is independent of $\rho$.
		
		{\bf For $\rho=0$:} Since $\wh{Q}_0 \in L^2\!\bigl(\omega \times I;\, H^1_{\mathrm{per},0}(Y';\, \Szero)\bigr)$, 
		we have $\int_{Y'} \wh{Q}_0\, dy' = 0$ for a.e.\ $(x', y_3)$. The Poincar\'{e}-Wirtinger 
		inequality on $Y'$ gives
		$\|\wh{Q}_0\|^2_{L^2(Y')} \leq C_{Y'}\|\nabla_{y'}\wh{Q}_0\|^2_{L^2(Y')}$,
		and therefore
		\[
		\|\wh{Q}_0\|^2_{H^1(Y')}
		\leq (1 + C_{Y'})\,\|\nabla_{y'}\wh{Q}_0\|^2_{L^2(Y')}.
		\]
		Integrating over $\omega \times I$:
		\begin{equation}
			\label{eq:Q-lower-rho0}
			\|\nabla_{y'}\wh{Q}_0\|^2_{L^2(\omega \times Y)}
			\;\geq\;
			\frac{1}{1 + C_{Y'}}\,
			\|\wh{Q}_0\|^2_{L^2(\omega \times I;\, H^1(Y'))}.
		\end{equation}
		
		{\bf For $\rho\in(0,\infty)$:} Since $\wh{Q}_\rho  \in L^2\!\bigl(\omega;\, H^1_{\mathrm{per},0}(Y;\, \Szero)\bigr)$ 
		with $\int_Y \wh{Q}_\rho\, dy = 0$, the Poincar\'{e}-Wirtinger inequality on the full 
		cell $Y$ gives
		\[
		\|\wh{Q}_\rho\|^2_{L^2(\omega \times Y)}
		\leq C_Y\,\|\nabla_y\wh{Q}_\rho\|^2_{L^2(\omega \times Y)}
		= C_Y\Bigl(
		\|\nabla_{y'}\wh{Q}_\rho\|^2_{L^2(\omega \times Y)}
		+ \rho^2\,\|\partial_{y_3}\wh{R}\|^2_{L^2(\omega \times Y)}
		\Bigr),
		\]
		where we used $\partial_{y_3}\wh{Q}_\rho = \rho\,\partial_{y_3}\wh{R}$. 
		Consequently,
		\begin{align*}
			\|\wh{Q}_\rho\|^2_{L^2(\omega \times I;\, H^1(Y'))}
			&= \|\wh{Q}_\rho\|^2_{L^2(\omega \times Y)}
			+ \|\nabla_{y'}\wh{Q}_\rho\|^2_{L^2(\omega \times Y)} \\
			&\leq (C_Y + 1)\,\|\nabla_{y'}\wh{Q}_\rho\|^2_{L^2(\omega \times Y)}
			+ C_Y\rho^2\,\|\partial_{y_3}\wh{R}\|^2_{L^2(\omega \times Y)} \\
			&\leq \bigl(C_Y(1+\rho^2) + 1\bigr)
			\Bigl(
			\|\nabla_{y'}\wh{Q}_\rho\|^2_{L^2(\omega \times Y)}
			+ \|\partial_{y_3}\wh{R}\|^2_{L^2(\omega \times Y)}
			\Bigr),
		\end{align*}
		where the last inequality uses $C_Y + 1 \leq C_Y(1+\rho^2)+1$ and 
		$C_Y\rho^2 \leq C_Y(1+\rho^2)+1$. Therefore
		\begin{equation}
			\label{eq:Q-lower-rho}
			\|\nabla_{y'}\wh{Q}_\rho\|^2_{L^2(\omega \times Y)}
			+ \|\partial_{y_3}\wh{R}\|^2_{L^2(\omega \times Y)}
			\;\geq\;
			\frac{1}{C_Y(1+\rho^2)+1}\,
			\|\wh{Q}_\rho\|^2_{L^2(\omega \times I;\, H^1(Y'))}.
		\end{equation}
		
		{\bf For $\rho=\infty$:} Here $\wh{Q}_\infty \in L^2\!\bigl(\omega \times I;\, H^1_{\mathrm{per},0}(Y';\, \Szero)\bigr)$. 
		Both $\wh{Q}$ and $\wh{R}_1$ have zero mean over $Y'$, hence so does 
		$\wh{Q}_\infty$. The Poincar\'{e}-Wirtinger inequality on $Y'$ applies as above,
		\begin{equation}
			\label{eq:Q-lower-infty}
			\|\nabla_{y'}\wh{Q}_\infty\|^2_{L^2(\omega \times Y)}
			\;\geq\;
			\frac{1}{1+C_{Y'}}\,
			\|\wh{Q}_\infty\|^2_{L^2(\omega \times I;\, H^1(Y'))}.
		\end{equation}
		Combining the equality~\cref{eq:upper-equality} with the lower bounds 
		\cref{eq:R-lower} and the appropriate estimate 
		among~\cref{eq:Q-lower-rho0}--\cref{eq:Q-lower-infty}, we obtain in each regime
		\[
		\|G_\rho\|^2_{L^2(\omega \times Y)}
		\;\geq\;
		c_\rho\Bigl(
		\|\nabla'\Qc\|^2_{L^2(\omega)}
		+ \|\wh{Q}_\rho\|^2_{L^2(\omega \times I;\, H^1(Y'))}
		+ \|\wh{R}\|^2_{L^2(\omega \times Y';\, H^1(I))}
		\Bigr),
		\]
		with $c_\rho$ as stated.  This completes the proof.
	\end{proof}
	
%	\subsection{Proof of Lemma \ref{lem:density-Astar}}\label{app:den}

\addcontentsline{toc}{section}{References}
\bibliographystyle{plain}
{ \bibliography{Ref}
}
\end{document}